\documentclass{tac}

\usepackage{amsmath, amsfonts, amssymb, tikz-cd, calrsfs,float}
\usepackage[shortlabels]{enumitem}
\usepackage[T1]{fontenc}

\newtheorem{claim}[theorem]{Claim}

\newcommand\rel[1]{\mathrel{#1}}
\newcommand\nr[1]{\not\mathrel{#1}}

\newcommand{\SSL}{{\sf{StoneMS}}}
\newcommand{\AlgLat}{{\sf{AlgLat}_{Sup}}}

\newcommand{\CohLoc}{{\sf{CohLoc}}}

\newcommand{\AlgFrmJ}{{\sf{AlgFrm}_{Sup}}}
\newcommand{\AlgFrmJB}{{\sf{AlgFrm}_{SupB}}}
\newcommand{\AlgFrmWS}{{\sf{AlgFrm}_{FInf}}}
\newcommand{\AlgFrmS}{{\sf{AlgFrm}_{FInfB}}}

\newcommand{\AlgFrm}{{\sf{AlgFrm}}}

\newcommand{\ArFrmJ}{{\sf{ArFrm}_{Sup}}}
\newcommand{\ArFrmWS}{{\sf{ArFrm}_{FInf}}}
\newcommand{\ArFrmS}{{\sf{ArFrm}_{FInfB}}}
\newcommand{\ArFrm}{{\sf{ArFrm}}}

\newcommand{\KAlgFrmJ}{{\sf{KAlgFrm}_{Sup}}}
\newcommand{\KAlgFrmJB}{{\sf{KAlgFrm}_{SupB}}}
\newcommand{\KAlgFrmWS}{{\sf{KAlgFrm}_{FInf}}}
\newcommand{\KAlgFrmS}{{\sf{KAlgFrm}_{FInfB}}}
\newcommand{\KAlgFrm}{{\sf{KAlgFrm}}}

\newcommand{\CohFrmJ}{{\sf{CohFrm}_{Sup}}}
\newcommand{\CohFrmWS}{{\sf{CohFrm}_{FInf}}}

\newcommand{\CohFrm}{{\sf{CohFrm}}}

\newcommand{\StoneFrmJ}{{\sf{StoneFrm}_{Sup}}}
\newcommand{\StoneFrmWS}{{\sf{StoneFrm}_{FInf}}}

\newcommand{\StoneFrm}{{\sf{StoneFrm}}}

\newcommand{\LStoneFrmJ}{{\sf{LStoneFrm}_{Sup}}}
\newcommand{\LStoneFrmWS}{{\sf{LStoneFrm}_{FInf}}}
\newcommand{\LStoneFrmS}{{\sf{LStoneFrm}_{FInfB}}}
\newcommand{\LStoneFrm}{{\sf{LStoneFrm}}}

\newcommand{\GPries}{{\sf{GPS}}}
\newcommand{\GPriesT}{{\sf{GPS}_T}}
\newcommand{\GPriesPS}{{\sf{GPS}_{PS}}}
\newcommand{\GPriesS}{{\sf{GPS}_S}}
\newcommand{\GPriesF}{{\sf{GPS}_F}}
\newcommand{\GPriesP}{{\sf{GPS}_P}}

\newcommand{\PGPries}{{\sf{PGPS}}}
\newcommand{\PGPriesT}{{\sf{PGPS}_T}}
\newcommand{\PGPriesWS}{{\PGPries_{\sf S}}}
\newcommand{\PGPriesS}{{\PGPries_{\sf ST}}}
\newcommand{\PGPriesSF}{{\PGPries_{\sf FT}}}
\newcommand{\PGPriesF}{{\PGPries_{\sf F}}}
\newcommand{\PGPriesP}{{\PGPries_{\sf P}}}

\newcommand{\Pries}{{\sf{PS}}}
\newcommand{\PriesR}{{\sf{PS}_R}}
\newcommand{\PriesPS}{{\sf{PS}_{PS}}}

\newcommand{\PPries}{{\sf{PPS}}}
\newcommand{\PPriesWS}{{\sf{PPS}_{S}}}
\newcommand{\PPriesS}{{\sf{PPS}_{ST}}}
\newcommand{\PPriesP}{{\sf{PPS}_P}}

\newcommand{\Stone}{{\sf{Stone}}}
\newcommand{\StonePS}{{\sf{Stone}_{PS}}}
\newcommand{\StoneS}{{\sf{Stone}_S}}
\newcommand{\StoneR}{{\sf{Stone}_R}}

\newcommand{\PStone}{{\sf{PStone}}}
\newcommand{\PStoneWS}{{\sf{PStone}_{S}}}
\newcommand{\PStoneS}{{\sf{PStone}_{ST}}}
\newcommand{\PStoneP}{{\sf{PStone}_P}}

\newcommand{\MSLat}{{\sf{MS}}}
\newcommand{\DMSLatWS}{{\sf{DMS}_{FSup}}}
\newcommand{\DMSLat}{{\sf{DMS}}}
\newcommand{\DMSLatB}{{\sf{DMS}_B}}
\newcommand{\DMSLatS}{{\sf{DMS}_{FSupB}}}
\newcommand{\DMSLatP}{{\sf{DMS}_P}}

\newcommand{\BDMSLat}{{\sf{BDMS}}}
\newcommand{\BDMSLatB}{{\sf{BDMS}_B}}
\newcommand{\BDMSLatWS}{{\sf{BDMS}_{FSup}}}
\newcommand{\BDMSLatS}{{\sf{BDMS}_{FSupB}}}
\newcommand{\BDMSLatP}{{\sf{BDMS}_P}}

\newcommand{\DLat}{{\sf{DL}}}

\newcommand{\DLatM}{{\sf{DL}_M}}

\newcommand{\DLatL}{{\sf{DL}_L}}

\newcommand{\DmLat}{{\sf{DL}^-}}
\newcommand{\DmLatM}{{\sf{DL}_M^-}}
\newcommand{\DmLatS}{{\sf{DL}_B^-}}
\newcommand{\DmLatP}{{\sf{DL}_P^-}}

\newcommand{\BAM}{{\sf{BA}_M}}

\newcommand{\BAL}{{\sf{BA}_L}}
\newcommand{\BA}{{\sf{BA}}}

\newcommand{\GBAM}{{\sf{GBA}_M}}
\newcommand{\GBA}{{\sf{GBA}}}
\newcommand{\GBAS}{{\sf{GBA}_B}}
\newcommand{\GBAP}{{\sf{GBA}_P}}

\newcommand{\Clop}{{\sf{Clop}}}
\newcommand{\ClopUp}{{\sf{ClopUp}}}
\newcommand{\Filt}{{\sf{Filt}}}
\newcommand{\ClopFilt}{{\sf{ClopFilt}}}
\newcommand{\Pf}{{\sf{Pr}}}
\newcommand{\Kp}{{\sf{K}}}

\newcommand{\I}{\mathcal{I}}
\newcommand{\X}{\mathcal{X}}
\newcommand{\Y}{\mathcal{Y}}
\newcommand{\PP}{{\sf{PP}}}
\newcommand{\p}{{\sf{P}}}
\newcommand{\D}{{\sf{D}}}
\newcommand{\Opt}{{\sf{Opt}}}
\newcommand{\V}{\mathcal{V}^a}
\newcommand{\A}{\mathcal{A}}
\newcommand{\K}{\mathcal{K}}
\newcommand{\F}{\mathcal{F}}

\newcommand{\up}{{\uparrow}}
\newcommand{\down}{{\downarrow}}
\newcommand{\arr}{blue} 
\newcommand{\rowsep}{1.0pc}

\newcommand{\widthone}{0.89in}

\newcommand{\widththree}{4.09in}
\newcommand{\widthfour}{.82in}
\newcommand{\widthfive}{1.2in}

\newcommand{\widthseven}{2.15in}
\newcommand{\widtheight}{.82in}
\newcommand{\word}{s.t.}

\begin{document}

\title[Connecting Priestley duality to Hofmann-Mislove-Stralka duality]
{Connecting generalized Priestley duality to Hofmann-Mislove-Stralka duality}

\author{G.~Bezhanishvili, L.~Carai, P.~J.~Morandi}
\address{New Mexico State University\\ Universit\`a degli Studi di Milano\\ 
New Mexico State University\\}

\eaddress{guram@nmsu.edu, luca.carai.uni@gmail.com, pmorandi@nmsu.edu}

\amsclass{06A12; 06D22; 06E15; 18F70; 22A26}
\keywords{Stone duality;  Priestley duality; semilattice; 
algebraic lattice; algebraic frame; coherent frame}

\maketitle

\begin{abstract}
We connect Priestley duality for distributive lattices and 
its generalization to distributive meet-semilattices to Hofmann-Mislove-Stralka  
duality for semilattices. Among other things, this involves consideration 
of various morphisms between algebraic frames. We also show how Stone duality 
for boolean algebras and generalized 
boolean algebras fits as a particular case of the general picture we develop. 
\end{abstract}

\tableofcontents

\section{Introduction}

The celebrated Stone duality \cite{Sto36} establishes a dual equivalence 
between the categories of boolean algebras and what we now call 
Stone spaces (zero-dimensional compact Hausdorff spaces). 
Since Stone's groundbreaking work, numerous dualities have been developed for various 
categories of algebras. 
Stone himself generalized his duality for boolean algebras to distributive lattices 
\cite{Sto37c}. 
The resulting dual spaces are now known as spectral spaces and play an 
important role in algebraic geometry as Zariski spectra of commutative 
rings. 

In \cite{Pri70,Pri72} Priestley developed another duality for 
distributive lattices by means of certain
ordered Stone spaces, 
which became known as Priestley spaces. They form a subcategory of the category
of ordered topological spaces studied by Nachbin \cite{Nac65} 
and have numerous applications in diverse areas such as natural dualities 
\cite{CD98a}, formal concept analysis \cite{GW99,DP02}, computer 
science \cite{Pan13,Geh16}, and modal logic \cite{Gol89,CJ99}.

There are various generalizations of Priestley duality. It was extended
to all lattices by Urquhart \cite{Urq78} (see also 
\cite{Har92,Har97,HD97,Plo08,GvG14,MJ14a}).
Duality theory for distributive lattices with operators and the theory 
of their canonical extensions was developed in \cite{Gol89,GJ94,GJ04}, 
and it was further generalized to lattices with operators in \cite{GH01} 
(see also \cite{MJ14b}). Our main interest is in generalized Priestley duality 
for distributive semilattices developed in \cite{BJ08,BJ11} and \cite{HP08}. 
Our aim is to connect 
this duality, as well as Priestley and Stone dualities, 
to the Pontryagin-style duality for semilattices developed by Hofmann, Mislove, and Stralka \cite{HMS74},
which we will refer to as \emph{HMS duality}. 

The classic Pontryagin duality (see, e.g., \cite[Sec.~24]{HR94}) states that 
the category of locally compact abelian groups is self-dual. As a corollary, 
the categories of abelian groups and compact Hausdorff abelian groups are 
dual to each other, as are the categories of torsion abelian groups and 
Stone abelian groups (see \cite[Sec.~24]{HR94} or \cite[Ch.~VI(3,4)]{Joh82}). 
A version of Pontryagin duality for semilattices by
Hofmann, Mislove, and Stralka \cite{HMS74} states that the categories of meet-semilattices and Stone 
meet-semilattices are dual to each other. Since Stone meet-semilattices 
are exactly the algebraic 
lattices, at the object level this result is a reformulation of an earlier result of 
Nachbin \cite{Nac49} (see also Birkhoff and Frink \cite{BF48}) that there is a 1-1 
correspondence between semilattices and algebraic lattices. The restriction of this 
duality to the distributive case yields that the categories of distributive semilattices 
and distributive algebraic lattices are equivalent. This provides an important link to 
pointfree topology \cite{Joh82,PP12} since distributive algebraic lattices are exactly 
the algebraic frames. 

There are various morphisms to consider between distributive meet-semilattices, which 
give rise to various morphisms between algebraic frames. The study of the resulting 
categories is one of our aims. In addition, we show that prime and pseudoprime 
elements of algebraic lattices, which have been extensively studied in domain 
theory in connection with continuous lattices \cite{GHKLMS03}, can be used to 
analyze the spectra of prime and optimal filters of distributive meet-semilattices that 
play a crucial role in generalized Priestley duality. It is this analysis that connects
generalized Priestley duality to HMS duality. 

By a meet-semilattice we mean a poset in which all finite meets exist, including the empty meet. Thus, 
a meet-semilattice $M$ has a top element, but $M$ may not have a bottom element. 
By HMS duality, the category $\MSLat$ of meet-semilattices 
is dual to the category $\SSL$ of Stone meet-semilattices (that is, 
topological meet-semilattices 
whose topology is a Stone topology). Since Stone meet-semilattices are 
exactly the algebraic 
lattices, working with left adjoints of $\SSL$-morphisms yields 
the category $\AlgLat$ of algebraic lattices and maps 
between them that preserve arbitrary joins and compact elements 
(see \cite[p.~272]{GHKLMS03}). Thus, HMS 
duality yields that $\MSLat$ is equivalent to $\AlgLat$
(see \cite[p.~274]{GHKLMS03}).
This equivalence is obtained by the functors
\[
\F\colon\MSLat\to\AlgLat \ \mbox{ and } \ \K\colon\AlgLat\to\MSLat.
\] 
The functor $\F$ sends each $M \in \MSLat$ to the algebraic lattice of its 
filters ordered by inclusion.
The functor $\K$ sends each $L \in \AlgLat$ to the meet-semilattice of its 
compact elements ordered by the dual of the restriction of the order on $L$.

Let $\DMSLat$ 
be the full subcategory of $\MSLat$ consisting of distributive meet-semilattices. 
As we pointed out above, distributive algebraic lattices are exactly the 
algebraic frames, and we denote by 
$\AlgFrmJ$ the full subcategory of $\AlgLat$ consisting of algebraic frames. 
We point out that morphisms in $\AlgFrmJ$ preserve arbitrary suprema, but may not
be frame homomorphisms.
Restricting the equivalence of $\MSLat$ and $\AlgLat$ to the distributive case 
yields that $\DMSLat$ is equivalent to $\AlgFrmJ$ (see 
Corollary~\ref{cor: DMSLat = DAlgLat}). 

A generalization of Priestley duality 
to distributive meet-semilattices was developed in \cite{BJ08,BJ11}. 
The key ingredient of this 
duality is the notion of an optimal filter of $M\in\DMSLat$, which is best 
described by means of the distributive envelope of $M$. We present two 
versions of generalized Priestley duality, for distributive meet-semilattices 
with and without a bottom element. 
When the bottom is present, a $\DMSLat$-morphism may or may not preserve it.
This results in two dualities for bounded distributive meet-semilattices and
generalized Priestley spaces (see Theorem~\ref{thm: gen Pries}).
Each generalizes to a duality between 
distributive meet-semilattices and 
pointed generalized Priestley spaces with appropriate morphisms 
(see Theorem~\ref{thm: pointed Pries} and 
Corollary~\ref{cor: DMSLatb = PGPriesT}).

Our main observation in connecting generalized Priestley duality to 
HMS duality
is that the categories $\AlgFrmJ$ and $\PGPries$ are dual to 
each other. This we do by constructing the contravariant functors 
\[
\V\colon\PGPries\to\AlgFrmJ \ \mbox{ and } \ \Y\colon\AlgFrmJ\to\PGPries.
\] 
The functor $\V$ is a version of 
the upper Vietoris functor, which is constructed by working with admissible 
closed upsets of pointed 
generalized Priestley spaces (see Definition~\ref{def: admissible closed upsets}). 
The functor $\Y$ is constructed by working with pseudoprime and prime 
elements of algebraic frames. 
In Theorem~\ref{thm: duality between DAlgLat and GPries*} we prove that the 
functors $\V$ 
and $\Y$ establish the duality of $\PGPries$ and $\AlgFrmJ$. This together with 
the equivalence of $\AlgFrmJ$ and $\DMSLat$ yields the duality of $\DMSLat$ 
and $\PGPries$. We also show how these results restrict to the bounded case.

It is natural to consider several stronger notions 
of morphisms between algebraic frames. Recalling that our morphisms preserve 
arbitrary suprema and compact elements, obvious choices are to consider those 
morphisms that preserve all finite infima (resp.\ nonempty finite infima)
of compact elements or only those 
finite infima of compact 
elements that are compact. The latter correspond to those $\DMSLat$-morphisms 
that preserve existing finite suprema (resp.~existing nonempty finite suprema), 
which is equivalent to the inverse image of an 
optimal filter being optimal (see Remark~\ref{rem: optimal}). 
The former are simply frame homomorphisms that 
preserve compact elements, and correspond to those $\DMSLat$-morphisms 
that pull prime filters back to prime filters (see Lemma~\ref{lem: hansoul}). 
We characterize the $\PGPries$-morphisms 
that correspond to these classes of morphisms between algebraic frames, thus 
yielding a series of duality results, which in particular imply the results of 
\cite{BJ08,BJ11} and \cite{HP08}. 

We also show how Stone and Priestley 
dualities fit in the general picture developed in this paper. 
We consider Priestley duality 
for distributive lattices with and without bottom. The latter involves working 
with pointed Priestley spaces. On the frame side, the bounded case requires 
working with coherent frames; that is, algebraic frames in which all finite meets of 
compact elements are compact (see Theorem~\ref{thm: coh = PS}). 
The non-bounded case requires working with 
arithmetic frames; that is, algebraic frames where nonempty finite meets of 
compact elements are compact, but the frame itself may not be compact 
(see Theorem~\ref{thm: detailed pointed Pries}). 
Similarly, we consider two versions of Stone duality, for boolean algebras and for 
generalized boolean algebras. For the latter we work with pointed Stone spaces. 
On the frame side, boolean algebras give rise to Stone frames (see 
Theorem~\ref{thm: Stone duality}), while generalized boolean algebras to
locally Stone frames  (see Theorem~\ref{thm: pointed Stone}). 
As a special case we derive the dualities of Halmos \cite{Hal56} and 
Cignoli et.~al.~\cite{CLP91}.

The following diagram summarizes the connection between HMS duality and 
generalized Priestley duality. 
The horizontal arrows below the labeled arrows are their restrictions. 
The arrows $\rightarrowtail$ represent being a subcategory while the arrows
$\hookrightarrow$ being equivalent to a subcategory. 
The \arr\ color indicates the results obtained in this paper. 
The corresponding diagrams summarizing a similar picture for various categories of 
distributive lattices and boolean algebras are given at the ends of 
Sections~\ref{sec: deriving Priestley}~and~\ref{sec: deriving Stone}, respectively.

\begin{figure}[H] 
\[
\begin{tikzcd}[row sep = \rowsep, column sep = 1.0pc]
& \PGPries \arrow[rr, shift left = .6ex, \arr, "\V"]  \arrow[from=dd, hookrightarrow] && 
\AlgFrmJ \arrow[ll, shift left = .6ex, \arr, "\Y"] \arrow[rr, shift left = .6ex, "\K"] 
\arrow[from=dd, tail] && \DMSLat \arrow[ll, shift left = .6ex, "\F"] \\
\GPries \arrow[ru, hookrightarrow]  \arrow[rr, shift left = .6ex, \arr, crossing over] 
&& \KAlgFrmJ \arrow[ll, shift left = .6ex, \arr, crossing over] \arrow[ru, tail] 
\arrow[rr, shift left = .6ex, crossing over]  && \BDMSLat \arrow[ru, tail]  
\arrow[ll, shift left = .6ex, crossing over] \\
& \PGPriesWS \arrow[rr, \arr, shift left = .6ex]  \arrow[from=dd, tail] && \AlgFrmWS 
\arrow[ll, \arr, shift left = .6ex] \arrow[rr, shift left = .6ex] \arrow[from=dd, tail] && 
\DMSLatWS \arrow[ll, shift left = .6ex] \arrow[uu, tail] \\
\GPriesPS \arrow[rr, shift left = .6ex, \arr, crossing over] \arrow[ru, hookrightarrow] 
\arrow[uu, hookrightarrow]
&& \KAlgFrmWS \arrow[ll, shift left = .6ex, \arr, crossing over] \arrow[ru, tail]  
\arrow[uu, crossing over, tail] 
\arrow[rr, shift left = .6ex, crossing over]  && 
\BDMSLatWS \arrow[ll, shift left = .6ex, crossing over] \arrow[ru, tail] 
\arrow[uu, crossing over, tail]
\\
& \PGPriesS \arrow[rr, \arr, shift left = .6ex]  \arrow[from=dd, tail] && \AlgFrmS 
\arrow[ll, \arr, shift left = .6ex] \arrow[rr, shift left = .6ex] \arrow[from=dd, tail] && 
\DMSLatS \arrow[ll, shift left = .6ex] \arrow[uu, tail] \\
\GPriesS \arrow[rr, shift left = .6ex, \arr, crossing over] \arrow[ru, hookrightarrow] 
\arrow[uu, tail]
&& \KAlgFrmS \arrow[ll, shift left = .6ex, \arr, crossing over] \arrow[ru, tail]  
\arrow[uu, crossing over, tail] \arrow[rr, shift left = .6ex, crossing over]  && 
\BDMSLatS \arrow[ll, shift left = .6ex, crossing over] \arrow[ru, tail] 
\arrow[uu, crossing over, tail]\\
& \PGPriesP \arrow[rr, \arr, shift left = .6ex] && \AlgFrm \arrow[ll, shift left = .6ex, \arr] 
\arrow[rr, shift left = .6ex] &&  \DMSLatP \arrow[ll, shift left = .6ex] \arrow[uu, tail] \\
\GPriesP \arrow[rr, \arr, shift left = .6ex] \arrow[uu, tail] \arrow[ru, hookrightarrow] 
&& \arrow[ll, \arr, shift left = .6ex] \KAlgFrm \arrow[rr, shift left = .6ex] 
\arrow[uu, tail, crossing over] \arrow[ru, tail] && \BDMSLatP \arrow[ll, shift left = .6ex] 
\arrow[uu, tail, crossing over] \arrow[ru, tail]
\end{tikzcd}
\]
\caption{Connecting generalized Priestley duality and HMS duality}\label{fig: first diagram}
\end{figure}

The tables below describe the categories listed in Figure~\ref{fig: first diagram}. The order
of the categories in each table corresponds to the diagram from top to bottom.

\begin{center}
\begin{tabular}{|p{\widthone}p{\widththree}p{\widthfour}|} \hline
\multicolumn{3}{|c|}{\textbf{Categories of pointed generalized Priestley spaces}} \\ \hline
\textsf{Category} & \textsf{Morphisms} & \textsf{Location}\\ 
\hline
$\PGPries$  &  generalized Priestley morphisms & Def.~\ref{def: PGPries}\\
$\PGPriesWS$  & strong Priestley morphisms & Def.~\ref{def: functional and strong}\ref{PGPriesS} \\
$\PGPriesS$  & $\PGPriesWS$-morphisms \word~$f[X^-] \subseteq Y^-$ when $m, n$ are isolated & \textquotedbl \\
$\PGPriesP$  & $\PGPriesWS$-morphisms \word~$f[X_0] \subseteq Y_0$ & Def.~\ref{def: PGPriesP}\\
\hline
\end{tabular}
\end{center}

\begin{center}
\begin{tabular}{|p{\widthone}p{\widththree}p{\widthfour}|} \hline
\multicolumn{3}{|c|}{\textbf{Categories of generalized Priestley spaces}} \\ \hline
\textsf{Category} & \textsf{Morphisms} & \textsf{Location}\\ 
\hline
$\GPries$ & generalized Priestley morphisms & Def.~\ref {def: GPries}\\
$\GPriesPS$ & partial strong Priestley morphisms & Def.~\ref{def: partial strong Priestley morphism}\ref{GPriesPS}\\
$\GPriesS$  & strong Priestley morphisms & 
Def.~\ref{def: functional and strong morphisms}\ref{GPriesS}\\
$\GPriesP$  & $\GPriesS$-morphisms  
\word~$f[X_0] \subseteq Y_0$ & Def.~\ref{def: PGPriesP}\\
\hline
\end{tabular}
\end{center}

\begin{center}
\begin{tabular}{|p{\widthone}p{\widththree}p{\widthfour}|} \hline
\multicolumn{3}{|c|}{\textbf{Categories of algebraic frames}} \\ \hline
\textsf{Category} & \textsf{Morphisms} & \textsf{Location}\\ 
\hline
$\AlgFrmJ$  & maps preserving suprema and compact elements & 
Def.~\ref{def: DMSLat and AlgFrmJ}\ref{AlgFrmJ} \\ 
$\AlgFrmWS$  & $\AlgFrmJ$-morphisms satisfying 
(\ref{eqn: dagger}) & Def.~\ref{def: AlgFrmS}\ref{AlgFrmWS}\\
$\AlgFrmS$  & bounded $\AlgFrmWS$-morphisms
& \textquotedbl \\
$\AlgFrm$  & frame homomorphisms preserving compact elements 
& Def.~\ref{def: AlgFrm}\\
\hline
\end{tabular}
\end{center}

\begin{center}
\begin{tabular}{|p{\widthone}p{\widththree}p{\widthfour}|} \hline
\multicolumn{3}{|c|}{\textbf{Categories of compact algebraic frames}} \\ \hline
\textsf{Category}  & \textsf{Morphisms} & \textsf{Location}\\ 
\hline
$\KAlgFrmJ$  & $\AlgFrmJ$-morphisms & 
Def.~\ref{def: KAlgFrmJ and BDMSLat}\ref{KAlgFrmJB}\\
$\KAlgFrmWS$  & $\AlgFrmWS$-morphisms & 
Def.~\ref{def: AlgFrmS}\ref{KAlgFrmWS}\\
$\KAlgFrmS$  & $\AlgFrmS$-morphisms & 
\textquotedbl \\
$\KAlgFrm$  & $\AlgFrm$-morphisms & 
Def.~\ref{def: KAlgFrm and BDMSLatP}\ref{KAlgFrm} \\
\hline
\end{tabular}
\end{center}

\begin{center}
\begin{tabular}{|p{\widthone}p{\widththree}p{\widthfour}|} \hline
\multicolumn{3}{|c|}{\textbf{Categories of distributive meet-semilattices}} \\ \hline
\textsf{Category}  & \textsf{Morphisms} & \textsf{Location}\\ 
\hline
$\DMSLat$ & meet-semilattice homomorphisms & 
Def.~\ref{def: DMSLat and AlgFrmJ}\ref{DMSLat}\\
$\DMSLatWS$  & $\DMSLat$-morphisms preserving existing nonempty finite sups & 
Def.~\ref{def: DMSLatS}\ref{DMSLatWS} \\
$\DMSLatS$  & bounded $\DMSLatWS$-morphisms & 
Def.~\ref{def: DMSLatS}\ref{DMSLatS}\\
$\DMSLatP$  & $\DMSLat$-morphisms satisfying (\ref{eqn: *})
& Def.~\ref{def: DMSLatP}\\
\hline
\end{tabular}
\end{center}

\begin{center}
\begin{tabular}{|p{\widthone}p{\widththree}p{\widthfour}|} \hline
\multicolumn{3}{|c|}{\textbf{Categories of bounded distributive meet-semilattices}} \\ \hline
\textsf{Category} & \textsf{Morphisms} & \textsf{Location}\\ 
\hline
$\BDMSLat$ & $\DMSLat$-morphisms & 
Def.~\ref{def: KAlgFrmJ and BDMSLat}\ref{BDMSLatB}  \\
$\BDMSLatWS$  & $\DMSLatWS$-morphisms & 
Def.~\ref{def: DMSLatS}\ref{BDMSLatWS} \\
$\BDMSLatS$  & $\DMSLatS$-morphisms & 
Def.~\ref{def: DMSLatS}\ref{BDMSLatS}\\
$\BDMSLatP$  & $\DMSLatP$-morphisms & 
Def.~\ref{def: KAlgFrm and BDMSLatP}\ref{BDMSLatP}\\
\hline
\end{tabular}
\end{center}

\section{Hofmann-Mislove-Stralka duality} \label{sec: pontryagin}

We recall that a {\em meet-semilattice} is a poset $M$ in which meets of 
finite subsets exist. In particular, $M$ has a top element, which we
denote by $1_M$ or $1$ if the context is clear. A 
{\em meet-semilattice homomorphism} is a map 
$\alpha\colon M_1\to M_2$ preserving all finite meets. In particular, 
$\alpha(1_{M_1})=1_{M_2}$. 

A {\em topological meet-semilattice} is a meet-semilattice $M$ such that 
$M$ is also a topological space in which the meet operation is continuous. 
If the  topology on $M$ is a Stone topology, then we call $M$ a 
{\em Stone meet-semilattice}. 

\begin{definition} \label{def: MS and StoneMS}
\hfill
\begin{enumerate}[$(1)$]
\item Let $\MSLat$\label{MSLat} be the category of 
meet-semilattices and meet-semilattice homomorphisms.
\item Let $\SSL$\label{SSL} be the category of Stone 
meet-semilattices and continuous meet-semi\-lattice homomorphisms.
\end{enumerate}
\end{definition}

Hofmann, Mislove, and Stralka \cite{HMS74} developed a duality between 
$\MSLat$ and $\SSL$
that is reminiscent of Pontryagin duality. We will refer to it as \emph{HMS duality}.  
There are two contravariant 
functors establishing this duality. The functor $\MSLat\to\SSL$ sends 
$M\in\MSLat$ to the Stone meet-semilattice ${\rm hom}_{\MSLat}(M,2)$, 
and the functor $\SSL\to\MSLat$ sends $L\in\SSL$ to ${\rm hom}_{\SSL}(L,2)$. 
Since meet-semilattice homomorphisms $M\to 2$ correspond to filters of $M$, 
we can alternatively work with filters of $M$, which is more 
convenient for our purposes. 
We thus define the functors establishing HMS duality as follows. 

For $M \in \MSLat$, let $\Filt(M)$ be the poset of filters of $M$ ordered by 
inclusion. For $a \in M$ let $\sigma(a) = \{ F \in \Filt(M) \mid a \in F \}$, and 
topologize $\Filt(M)$ by the subbasis 
\[
\{ \sigma(a) \mid a \in M \} \cup \{ \sigma(b)^c \mid b \in M \}.
\] 
Then $\Filt(M)$ is a Stone meet-semilattice. Moreover, if 
$\alpha \colon M_1 \to M_2$ 
is an $\MSLat$-morphism, then $\alpha^{-1} \colon \Filt(M_2) \to \Filt(M_1)$ is a 
$\SSL$-morphism. This defines a contravariant functor 
$\Filt \colon \MSLat \to \SSL$. 

To define a contravariant functor in the other direction, for $L \in \SSL$ let 
$\ClopFilt(L)$ be the poset of clopen filters of $L$ ordered by inclusion. 
Then $\ClopFilt(L)$ is a meet-semilattice in which finite meets are finite 
intersections (and $L$ is the top element). Moreover, if $\alpha \colon L_1 \to L_2$ 
is a $\SSL$-morphism, then $\alpha^{-1} \colon \ClopFilt(L_2) \to \ClopFilt(L_1)$ 
is an $\MSLat$-morphism. This defines a contravariant functor 
$\ClopFilt \colon \SSL \to \MSLat$, which together with $\Filt$ yields HMS duality:

\begin{theorem} \emph{\cite[Thm.~3.9]{HMS74}} \label{thm: Pontryagin}
$\MSLat$ is dually equivalent to $\SSL$.
\end{theorem}

As we pointed out in the introduction, Stone meet-semilattices are exactly the algebraic 
lattices. To see this, we recall that an element $k$ of a complete lattice $L$ is 
{\em compact} if $k\le\bigvee S$ implies $k\le\bigvee T$ for some finite 
$T\subseteq S$, and that $L$ is {\em algebraic} if the poset $\Kp(L)$ of compact 
elements of $L$ is join-dense in $L$ (meaning that each element of $L$ is a join 
of compact elements). By HMS duality, each Stone 
meet-semilattice is isomorphic to $\Filt(M)$ for some $M\in\MSLat$. Clearly 
$\Filt(M)$ is a complete lattice, where meet is set-theoretic intersection and 
join is the filter generated by the union. Moreover, compact elements of $\Filt(M)$ 
are precisely the principal filters ${\uparrow}a$ (see \cite[Prop.~3.8]{HMS74}), 
which clearly join-generate $\Filt(M)$. 
Thus, each Stone meet-semilattice is an algebraic lattice. 

Conversely, if $L$ is an algebraic lattice, we consider the topology $\lambda$ on $L$ 
generated by the subbasis 
\[
\{ \up k \mid k \in \Kp(L) \} \cup \{ (\up l)^c \mid l \in \Kp(L) \}.
\]
Since $\{ \up k \mid k \in \Kp(L) \}$ is a basis for the Scott topology 
\cite[Cor.~II-1.15]{GHKLMS03}, 
$\lambda$ is the Lawson topology \cite[Def.~III-1.1]{GHKLMS03}. 
Therefore, $\lambda$ is a Stone topology \cite[Thm.~III-1.10]{GHKLMS03}, 
and with this topology, 
$L \in \SSL$ \cite[Thm.~III-2.8]{GHKLMS03}.

Let $\alpha \colon L_1 \to L_2$ be a map between Stone meet-semilattices. 
By  \cite[Thm.~II.3.25]{HMS74}, $\alpha$ is a $\SSL$-morphism iff 
$\alpha$ preserves arbitrary infima and directed suprema. 
But, $\alpha$ preserves arbitrary infima iff it has a left adjoint 
$\beta \colon L_2 \to L_1$, which then preserves 
arbitrary suprema. 
Moreover, $\alpha$ preserves directed suprema iff $\beta$ preserves 
compact elements \cite[Cor.~IV-1.12]{GHKLMS03}.
We thus obtain that $\SSL$ is dually isomorphic to
the following category:

\begin{definition} \label{def: AlgLatJ}
Let $\AlgLat$ be the category of algebraic lattices and maps between 
them preserving arbitrary suprema and compact elements. 
\end{definition}

The above observation together with Theorem~\ref{thm: Pontryagin} yields the 
following version of HMS duality for meet-semilattices:

\begin{corollary} \emph{\cite[p.~274]{GHKLMS03}} \label{cor: MSLat = AlgLat}
$\MSLat$ is equivalent to $\AlgLat$.
\end{corollary}

It is this version that we will mainly be working with in this paper. As we pointed out in 
the introduction, the object level of this equivalence goes back to Nachbin 
\cite[Thm.~1]{Nac49} (see also \cite[Thm.~2]{BF48}). However, Nachbin worked with 
join-semilattices and the corresponding algebraic lattices of ideals. Since we are working 
with meet-semilattices, our functor from $\MSLat$ to $\AlgLat$ is the filter functor. 
We next describe explicitly how it acts on morphisms.

Let $\alpha \colon M_1 \to M_2$ be an $\MSLat$-morphism. The left adjoint of 
$\alpha^{-1}\colon\Filt(M_2) \to \Filt(M_1)$ is the map $\ell \colon \Filt(M_1) 
\to \Filt(M_2)$ 
which sends $F \in \Filt(M_1)$ to
\begin{align*}
\ell(F) &= \bigwedge \{ G \in \Filt(M_2) \mid F \subseteq \alpha^{-1}(G) \} 
= \bigwedge \{ G \in \Filt(M_2) \mid \alpha[F] \subseteq G \} \\
&= \bigwedge \{ G \in \Filt(M_2) \mid \up\alpha[F] \subseteq G \} = \up \alpha[F].
\end{align*}
Let $\F \colon \MSLat \to \AlgLat$ be the 
functor that sends each 
$M \in \MSLat$ to $\Filt(M) \in \AlgLat$ and each $\MSLat$-morphism 
$\alpha\colon M_1 \to M_2$ to the $\AlgLat$-morphism $\ell\colon \Filt(M_1) 
\to \Filt(M_2)$.

The functor in the other direction is the compact element functor 
$\K \colon \AlgLat \to \MSLat$ which sends each algebraic lattice $L$ to the poset 
$\Kp(L)$ of compact elements of $L$ ordered by the dual $\ge$ of the restriction 
of $\le$ to $\Kp(L)$. Since $(\Kp(L),\le)$ is a sub-join-semilattice of $L$, 
$(\Kp(L),\ge)$ is a 
meet-semilattice, where the top element is $0$ (the bottom element of $L$). 
Again, our approach is dual to that of Nachbin \cite{Nac49} who worked with the 
join-semilattice $(\Kp(L),\le)$. 

The functor $\K \colon \AlgLat \to \MSLat$ sends an $\AlgLat$-morphism 
$\alpha \colon L_1 \to L_2$ to its restriction $\Kp(L_1)\to \Kp(L_2)$. 
This is well defined 
since $\alpha$ sends compact elements to compact elements, and it is an 
$\MSLat$-morphism because a finite join of compact elements is compact, 
$\alpha$ preserves suprema, and we work with $\ge$ on $\Kp(L_1)$ and $\Kp(L_2)$.

We conclude this section by observing that the equivalence of
Corollary~\ref{cor: MSLat = AlgLat} restricts to the distributive case. Recall that a 
meet-semilattice $M$ is {\em distributive} if whenever $a, b, c \in M$ with 
$a \wedge b \le c$, there are $a', b' \in M$ with $a \le a'$, $b \le b'$, and 
$a' \wedge b' = c$. 

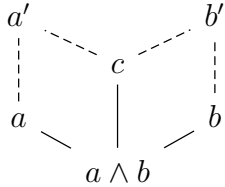
\[
\begin{tikzcd}[row sep = 0.7ex, column sep = 2ex]
a' & & b'\\
& c \arrow[lu, no head, no tail, dashed] \arrow[ru, no head, no tail, dashed] & \\
a \arrow[uu, no head, no tail, dashed] & & b \arrow[uu, no head, no tail, dashed]\\
 & a \wedge b \arrow[lu, no head, no tail] \arrow[ru, no head, no tail] \arrow[uu, no head, 
 no tail] &
\end{tikzcd}
\]

It is well known that $M$ is distributive iff $\Filt(M)$ is a distributive lattice 
(see, e.g., \cite[p.~167]{Gra11} for the dual result that a join-semilattice is distributive 
iff the lattice of its ideals is distributive). 

We recall \cite[p.~10]{PP12} that a complete lattice $L$ is a \emph{frame} if 
it satisfies the join infinite distributive law
$
a \wedge \bigvee S = \bigvee \{ a \wedge s \mid s \in S\}
$
for each $a \in L$ and $S \subseteq L$. An \emph{algebraic frame} is a frame that 
is an algebraic lattice. It is well known (see, e.g., \cite[p.~309]{Joh82}) that a 
distributive algebraic lattice is a frame. 

\begin{definition} \label{def: DMSLat and AlgFrmJ} 
\hfill
\begin{enumerate}[$(1)$]
\item \label{DMSLat} Let $\DMSLat$ be the 
full subcategory of $\MSLat$ consisting of distributive meet-semilattices.
\item \label{AlgFrmJ} Let $\AlgFrmJ$ be the 
full subcategory of $\AlgLat$ consisting of algebraic frames. 
\end{enumerate}
\end{definition}

As a consequence of Corollary~\ref{cor: MSLat = AlgLat} 
we obtain:

\begin{corollary} \label{cor: DMSLat = DAlgLat}
$\DMSLat$ is equivalent to $\AlgFrmJ$. 
\end{corollary}

As we pointed 
out in the introduction, $\mathsf{Sup}$ in the subscript indicates that morphisms 
in $\AlgFrmJ$ preserve arbitrary suprema, but they may not be frame homomorphisms. 
Note that $\AlgFrmJ$-morphisms preserve 0 but may not preserve 1. 

\begin{definition}
\hfill
\begin{enumerate}[$(1)$]
\item A $\DMSLat$-morphism $\alpha \colon M_1 \to M_2$ is \emph{bounded} if whenever $M_1$ and $M_2$ are bounded, 
then $\alpha(0) = 0$. Let $\DMSLatB$ be the wide subcategory of $\DMSLat$ whose morphisms are bounded. 
\item An $\AlgFrmJ$-morphism $\alpha \colon L_1 \to L_2$ is \emph{bounded} if whenever  $L_1$ and $L_2$ are compact, 
then $\alpha(1) = 1$. Let $\AlgFrmJB$ be the wide subcategory of $\AlgFrmJ$ whose morphisms are bounded.
\end{enumerate}
\end{definition}

For an algebraic frame $L$, since $\Kp(L)$ is ordered by $\ge$, $\Kp(L)$ has a bottom iff $L$ is compact. 
Therefore, 
for algebraic frames $L_1$ and $L_2$, an $\AlgFrmJ$-morphism $\alpha \colon L_1 \to L_2$ is an 
$\AlgFrmJB$-morphism iff its restriction $\alpha \colon \Kp(L_1) \to \Kp(L_2)$ is a $\DMSLatB$-morphism. 
Thus, the following is an immediate consequence of Corollary~\ref{cor: DMSLat = DAlgLat}. 

\begin{corollary} \label{cor: DMSB and AlgFrmJ}
$\DMSLatB$ is equivalent to $\AlgFrmJB$.
\end{corollary}

We recall that a meet-semilattice $M$ is {\em bounded} if it has a bottom, and that a frame $L$ is 
{\em compact} if the top element of $L$ is compact. 

\begin{definition} \label{def: KAlgFrmJ and BDMSLat}
\hfill
\begin{enumerate}[$(1)$]
\item \label{BDMSLatB} Let $\BDMSLat$ be the full subcategory of $\DMSLat$ 
and $\BDMSLatB$ the full subcategory of $\DMSLatB$ whose objects are bounded meet-semilattices. 
\item \label{KAlgFrmJB} Let $\KAlgFrmJ$ be the full subcategory of $\AlgFrmJ$ 
and $\KAlgFrmJB$ the full subcategory of $\AlgFrmJB$ whose 
objects are compact algebraic frames.
\end{enumerate}
\end{definition}

As an immediate consequence of Corollaries~\ref{cor: DMSLat = DAlgLat} and \ref{cor: DMSB and AlgFrmJ}, we obtain:

\begin{corollary}\ \label{cor: BMSLat = KDAlgLat}
\begin{enumerate}[$(1)$]
\item \label{BMSLat = KDAlgLat} $\BDMSLat$ is equivalent to $\KAlgFrmJ$. \label{cor: BMSLat = KDAlgLat 1}
\item \label{BMSLatB = KDAlgLatJB} $\BDMSLatB$ is equivalent to $\KAlgFrmJB$. \label{cor: BMSLat = KDAlgLat 2}
\end{enumerate}
\end{corollary}

\section{Priestley duality and its generalizations} \label{sec: Priestley duality}

In this section we recall Priestley duality and its generalizations. We start by briefly 
describing Priestley duality. For a bounded distributive lattice $M$, let 
$X_M$ be the set of prime 
filters of $M$ ordered by inclusion, and let $\varphi \colon M \to \wp(X_M)$ 
be the map given by 
$\varphi(a) = \{ x \in X_M \mid a \in x\}$. We topologize $X_M$ by letting 
\[
\{\varphi(a) \mid a \in M \} \cup \{ \varphi(b)^c \mid b \in M\}
\]
be a subbasis for the topology. Then $X_M$ is a compact space such that if 
$x,y \in X_M$ with $x \not\le y$, then there is a clopen upset $U$ of $X_M$ 
containing $x$ and missing $y$. Such spaces are called \emph{Priestley spaces}. 

\begin{definition} \label{def: DL and PS}
\hfill
\begin{enumerate}[$(1)$]
\item \label{DL} Let $\DLat$\label{DLat} be the category of bounded 
distributive lattices and bounded lattice homomorphisms. 
\item \label{PS} Let $\Pries$\label{Pries} be the category of Priestley spaces and 
continuous order-preserving maps. 
\end{enumerate}
\end{definition}

We then have a contravariant functor $\DLat \to \Pries$ which sends $M$ to $X_M$ 
and a $\DLat$-morphism 
$\alpha \colon M_1 \to M_2$ to $\alpha^{-1} \colon X_{M_2} \to X_{M_1}$. 
We also have a contravariant functor $\Pries \to \DLat$ which sends $X \in \Pries$ to 
the lattice $\ClopUp(X)$ of clopen upsets of $X$ and a $\Pries$-morphism 
$f \colon X \to Y$ to the $\DLat$-morphism $f^{-1} \colon \ClopUp(Y) \to 
\ClopUp(X)$. 
These functors establish Priestley duality:
\begin{theorem} \emph{\cite{Pri70,Pri72}} \label{thm: Priestley}
$\DLat$ is dually equivalent to $\Pries$.
\end{theorem}

We will be interested in the generalization of Priestley duality to distributive 
meet-semilattices established in \cite{BJ08,BJ11} (see also \cite{HP08} for a similar 
duality for distributive join-semilattices, but with more restrictive morphisms). There 
are two versions of this duality for $\BDMSLat$ and $\DMSLat$. We first describe 
the duality for $\BDMSLat$ from which we derive the duality for $\DMSLat$.  

Let $X$ be a Priestley space. For a closed set $C \subseteq X$, let $\max C$ be 
the set of maximal points and $\min C$ the set of minimal points of $C$. It 
is well known that for every $x \in C$ there are $y \in \max C$ and $z \in \min C$ 
such that $z \le x \le y$.

Let $X_0$ be a fixed dense subset of $X$. We call a clopen upset $U$ of $X$ 
\emph{admissible} if $\max(X\setminus U) \subseteq X_0$. Let $\A(X)$ be the set of 
admissible clopen upsets of $X$. For $x \in X$ set 
$\I_x = \{ U \in \A(X) \mid x \notin U\}$. 

\begin{definition} \label{def: generalized Priestley space}
A \emph{generalized Priestley space} is a tuple $X = \langle X, \tau, \le, X_0\rangle$ 
satisfying
\begin{enumerate}[$(1)$]
\item \label{GPS 1} $\langle X, \tau, \le\rangle$ is a Priestley space;
\item \label{GPS 2} $X_0$ is a dense subset of $X$;
\item \label{GPS 3} $X_0$ is order-dense in $X$ \emph{(}meaning that for each 
$x \in X$ there is $y \in X_0$ with $x \le y$\emph{)};
\item \label{GPS 4} $x \in X_0$ iff $\I_x$ is directed;
\item \label{GPS 5} for all $x,y \in X$, we have $x \le y$ iff 
$\forall U \in \A(X),  \, x \in U \Longrightarrow y \in U$.
\end{enumerate}
\end{definition}

\begin{remark}
\mbox{}\begin{enumerate}[$(1)$]
\item By Definition~\ref{def: generalized Priestley space}\ref{GPS 3}, 
$\max X \subseteq X_0$. 
Thus, $\varnothing \in \A(X)$, so $\varnothing \in \I_x$, and hence 
$\I_x$ is nonempty for each $x \in X$.
\item If $X_0=X$, then $\A(X) = \ClopUp(X)$, so Conditions~\ref{GPS 2}--\ref{GPS 5} 
of Definition~\ref{def: generalized Priestley space} 
become redundant, and hence $X$ becomes a Priestley space.
\end{enumerate}
\end{remark}

Let $R \subseteq X \times Y$ be a relation between sets $X$ and $Y$. For 
$U \subseteq Y$ we follow the standard notation in modal logic and write 
$\Box_RU = \{ x\in X \mid R[x] \subseteq U\}$.

\begin{definition} \label{def: GPS morphism}
A {\em generalized Priestley morphism} between generalized Priestley spaces $X, Y$ 
is a relation $R \subseteq X \times Y$ satisfying
\begin{enumerate}[$(1)$]
\item \label{GPS morphism 1} If $x \nr{R} y$, then there is $U \in \A(Y)$ with 
$R[x] \subseteq U$ and $y 
\notin U$. 
\item \label{GPS morphism 2} If $U \in \A(Y)$, then $\Box_RU \in \A(X)$.
\end{enumerate}
\label{GPS morphism 3} We call $R$ \emph{total} if $R^{-1}[Y] = X$.
\end{definition}

This gives rise to two categories.

\begin{definition} \label{def: GPries}
Let $\GPries$ be the category of generalized Priestley spaces and generalized 
Priestley morphisms, and let $\GPriesT$ be the wide subcategory of $\GPries$ whose 
morphisms are total. 
\end{definition}

The identity morphism on $X\in\GPries$ is $\le$ and the 
composition $S * R$ of two morphisms 
$R \subseteq X \times Y$ and $S \subseteq Y \times Z$ is defined by 
\[
x \rel{(S * R)} z \Longleftrightarrow \forall U \in \A(Z),\, x \in \Box_R\Box_S U 
\Longrightarrow z \in U.
\]
We then have the following generalization of Priestley duality:

\begin{theorem} \label{thm: gen Pries}
\emph{\cite[Sec.~6]{BJ11}} 
$\BDMSLat$ is dually equivalent to $\GPries$ and $\BDMSLatB$ is dually equivalent to $\GPriesT$. 
\end{theorem}

The contravariant functors establishing these dual equivalences are constructed as follows.
The functor $\A \colon \GPries \to \BDMSLat$ sends 
$X \in \GPries$ to $\A(X)$ and a $\GPries$-morphism $R \subseteq X \times Y$ 
to $\Box_R \colon \A(Y) \to \A(X)$. 
The relation $R$ is total iff $\Box_R$ preserves the bottom, so we obtain
the restriction $\A \colon \GPriesT \to \BDMSLatB$.
\color{black}

 To define the contravariant functor $\X \colon \BDMSLat \to \GPries$ we recall the 
 definitions of prime and optimal filters of $M \in \BDMSLat$. A filter $P$ of $M$ 
 is \emph{prime} if $F_1 \cap F_2 \subseteq P$ implies $F_1 \subseteq P$ or 
 $F_2 \subseteq P$ for any $F_1, F_2 \in \Filt(M)$. To define optimal filters, we 
 require the definition of the distributive envelope $\D(M)$ of $M$. Let 
 $\BDMSLatS$ be the wide subcategory of $\BDMSLat$ whose morphisms 
 preserve existing finite suprema. Then the forgetful functor 
 ${\sf U}\colon\DLat\to\BDMSLatS$ 
 has a left adjoint $\D\colon \BDMSLatS\to\DLat$, and we call $\D(M)$ the 
 {\em distributive envelope} of $M$. There are various constructions of $\D(M)$ 
 (see \cite{CH78}, \cite[Sec.~3]{BJ11}, or \cite[Thm.~1.3]{HP08}). What matters to 
 us is that $M$ embeds into $\D(M)$ and we identify $M$ with its image in $\D(M)$. 
 Then a filter $F$ of $M$ is {\em optimal} if $F=P \cap M$ for some prime filter $P$ 
 of $\D(M)$. Thus, optimal filters of $M$ are the restrictions of prime filters 
 of $\D(M)$  to $M$. For other characterizations of optimal filters 
 we refer to \cite[Sec.~4]{BJ11}.

Let $\Opt(M)$ be the set of optimal filters of $M$. Then the set $\Pf(M)$ of prime 
filters of $M$ is contained in $\Opt(M)$. 
We order $\Opt(M)$ by inclusion, and  topologize it by letting 
\[
\{ \varphi(a) \mid a \in M\} \cup \{ \varphi(b)^c \mid b \in M\}
\] 
be a subbasis for the topology $\tau$, where 
$\varphi(a) = \{ x \in \Opt(M) \mid a \in x \}$. 
This yields the generalized Priestley space 
$\X(M) := \langle \Opt(M), \tau, \subseteq, \Pf(M)\rangle$.

If $\alpha \colon M_1 \to M_2$ is a $\BDMSLat$-morphism, we define 
$R_\alpha \subseteq \Opt(M_2) \times \Opt(M_1)$ by $x \rel{R_\alpha} y$ if 
$\alpha^{-1}(x) \subseteq y$ and let $\X(\alpha)=R_\alpha$. 
We have that $R_\alpha$ is total iff $\alpha(0) = 0$.
This defines the 
functor $\X\colon\BDMSLat\to\GPries$ and its restriction $\X \colon \BDMSLatB \to \GPriesT$,
and the functors $\A, \X$ establish the dual equivalences of Theorem~\ref{thm: gen Pries}.

To generalize this
to a duality for $\DMSLat$, we will work 
with pointed generalized Priestley spaces. This approach is similar to the one undertaken 
in \cite[Sec.~3]{BMR17} where Esakia duality for Heyting algebras was generalized to a 
duality for brouwerian algebras. 

\begin{definition} \label{def: PGPries}
We call a tuple $(X,m) = {\langle X, \tau, \le, X_0, m \rangle}$ a {\em pointed generalized 
Priestley space} if
\begin{enumerate}[$(1)$]
\item $\langle X, \tau, \le \rangle$ is a Priestley space;
\item $m$ is the unique maximum of $X$;
\item \label{dense} $X_0$ is a dense subset of $X \setminus \{ m \}$;
\item \label{order-dense} $X_0$ is order-dense in $X \setminus \{ m \}$;
\item $x \in X_0$ iff $\I_x$ is nonempty and directed;
\item \label{PGPS separation} $x \le y$ iff 
$\forall U \in \A(X), \, x \in U \Longrightarrow y \in U$.
\end{enumerate}
Let $\PGPries$ be the category of pointed generalized Priestley spaces and generalized 
Priestley morphisms. 
\end{definition}

\begin{remark} \label{rem: pStone}
We recall that a pointed Stone space is a pair $(X,p)$ where $X$ is a Stone space and 
$p\in X$. Thus, every Stone space can be made into a pointed Stone space. On the 
other hand, for a generalized Priestley space $X$ to be made into a pointed generalized 
Priestley space, $X$ must have a unique  maximum. Nevertheless, $\GPries$ 
is equivalent to a full subcategory of $\PGPries$, as we detail in the next remark. 
\end{remark}

\begin{remark} \label{rem: m iso}
By Definition~\ref{def: PGPries}, $m \notin X_0$, so $\varnothing \notin \A(X)$ and 
$\I_m = \varnothing$. In fact, $\I_x = \varnothing$ iff $x = m$. Moreover, $m$ is an 
isolated point iff $\{ m \}$ is the bottom of $\A(X)$. The full subcategory of 
$\PGPries$ consisting of those $(X,m)$ where $m$ is an isolated point is equivalent to 
$\GPries$. The 
equivalence is obtained as follows. If $(X,m)\in\PGPries$ and $m$ is isolated, then 
$X^- := X \setminus \{m\}$ is a generalized Priestley space. Also, if $(X,m),(Y,n)
\in\PGPries$ 
with $m,n$ isolated and $R \subseteq X \times Y$ is a $\PGPries$-morphism, then 
$R^-:=R \cap (X^- \times Y^-)$ is a $\GPries$-morphism.
Conversely, let $X\in\GPries$. If $X^+$ is obtained from $X$ by adding a new isolated 
top and $(X^+)_0 := X_0$, then $X^+ \in \PGPries$.
Also, if $R \subseteq X \times Y$ is a $\GPries$-morphism, then 
$R^+:= R \cup (X^+ \times \{n\})$ is a $\PGPries$-morphism. Thus, we obtain two 
functors which yield an equivalence of $\GPries$ and the full subcategory of $\PGPries$ 
consisting of those $(X,m) \in \PGPries$ in which $m$ is an isolated point.  
\end{remark}

The contravariant functor from $\PGPries$ to $\DMSLat$ is defined the same way as 
the contravariant functor $\A$ above and we use the same letter to denote it. The only 
difference is that if $(X, m)$ is a pointed generalized Priestley space, then 
$\varnothing\notin\A(X)$. Moreover, $\A(X)$ is bounded iff $\{m\}\in\A(X)$, which 
happens iff $m$ is an isolated point of $X$.
If $R \subseteq X \times Y$ is a generalized Priestley morphism with 
$(X,m),(Y,n) \in \PGPries$, then $\A(R)=\Box_R$. The only difference is that 
$m \in \Box_R U$ for each $U \in \A(Y)$.

The contravariant functor from $\DMSLat$ to $\PGPries$ is defined by a slight 
modification of the contravariant functor $\X$. Again, we use the same letter to 
denote it. The main difference is that if $M\in\DMSLat$, then we work with 
$X_M=\Opt(M)\cup\{M\}$, so $M$ becomes the unique maximum of $X_M$. Then 
$M$ has a bottom iff $\{M\}$ is an isolated point of $X_M$.
If $\alpha\colon M_1 \to M_2$ is a $\DMSLat$-morphism, then 
$\X(\alpha)=R_\alpha \subseteq X_{M_2} \times X_{M_1}$. In this case, 
$y \rel{R_\alpha} M_1$ for each $y \in X_{M_2}$.

Consequently, $\A\colon \PGPries \to\DMSLat$ and $\X\colon\DMSLat\to \PGPries$ 
yield the following generalization of Theorem~\ref{thm: gen Pries}, a version of which 
was established in \cite[Thm.~9.2]{BJ11}: 

\begin{theorem} \label{thm: pointed Pries}
$\DMSLat$ is dually equivalent to $\PGPries$.
\end{theorem}

\begin{definition}
Let $\PGPriesT$ be the wide subcategory of $\PGPries$ consisting of those $\PGPries$-morphisms 
$R \subseteq X \times Y$ 
such that if $X^-,Y^- \in \GPries$, then $R^- \subseteq X^- \times Y^-$ is total.  
\end{definition}

We have the following corollary of Theorem~\ref{thm: pointed Pries}, which generalizes
the duality of Theorem~\ref{thm: gen Pries} between $\BDMSLatB$ and 
$\GPriesT$.

\begin{corollary} \label{cor: DMSLatb = PGPriesT}
$\DMSLatB$ is dually equivalent to $\PGPriesT$.
\end{corollary}

\proof
Let $\alpha \colon M_1 \to M_2$ be a $\DMSLat$-morphism with $M_1, M_2$ bounded.  
Then $X_{M_1}^-, X_{M_2}^- \in \GPries$ (see Remark~\ref{rem: m iso}). It is sufficient to show that $\alpha$ is a 
$\DMSLatB$-morphism iff $R_\alpha^-$ is total. First suppose that $\alpha$ is a 
$\DMSLatB$-morphism, so $\alpha(0) = 0$. 
If $x \in \Opt(M_2)$, then $0 \notin \alpha^{-1}(x)$ since $0 \notin x$. Therefore, by \cite[Lem.~4.7]{BJ11},
there is 
$y \in \Opt(M_1)$ with $\alpha^{-1}(x) \subseteq y$. This implies that $R_\alpha^-[x] \ne \varnothing$. 
Conversely, let $R_\alpha^-$ be total. If $\alpha(0) \ne 0$, then $\up \alpha(0)$ is a proper filter. 
Therefore, there is $x \in \Opt(M_2)$ with $\up \alpha(0) \in x$. 
This implies that $0 \in \alpha^{-1}(x)$, so 
$\alpha^{-1}(x) = M_1$. Thus, $R_\alpha^-[x] = \varnothing$, and hence $R_\alpha^-$ is not total. 
The obtained contradiction proves that $\alpha(0) = 0$. 
\endproof

\section{Connecting generalized Priestley duality to HMS duality} 
\label{sec: deriving generalized Priestley duality}

In this section we connect generalized Priestley duality to HMS duality. 
To do so, we define contravariant functors 
$\V \colon \PGPries \to \AlgFrmJ$ and $\Y \colon \AlgFrmJ \to \PGPries$ and prove 
that they 
establish a dual equivalence, yielding our main result. This together with Corollary~\ref{cor: DMSLat = DAlgLat} 
gives Theorem~\ref{thm: pointed Pries}. As a consequence, we derive 
Theorem~\ref{thm: gen Pries} from Corollary~\ref{cor: BMSLat = KDAlgLat} and
Corollary~\ref{cor: DMSLatb = PGPriesT} from Corollary~\ref{cor: DMSB and AlgFrmJ}, 
thus obtaining the top layer of Figure~\ref{fig: first diagram}.

\subsection*{\textbf{The functor $\V$}.} \label{subsec: V} \label{subsec: 4.1}

Let $X := \langle X, \tau, \le, X_0, m \rangle$ be a pointed generalized Priestley space. 
We generalize the notion of an admissible clopen upset to an admissible closed upset 
of $X$. Recall that $U \in \ClopUp(X)$ is admissible if 
$\max(X \setminus U) \subseteq X_0$. This is equivalent to 
$X\setminus U = \down (X_0 \setminus U)$, which motivates our definition of 
admissible closed upsets. 

\begin{definition} \label{def: admissible closed upsets}
A closed upset $C$ of a pointed generalized Priestley space $X$ is admissible if 
$X \setminus C = \down(X_0 \setminus C)$. Let $\V(X)$ be the set of admissible 
closed upsets of $X$.
\end{definition}

\begin{remark}
The notation $\V(X)$ is motivated by the fact that $\V(X)$ is a version of the upper 
Vietoris functor applied to $X$.
\end{remark}

It is well known that in Priestley spaces, each closed upset is an intersection of clopen
upsets. This result generalizes to admissible closed upsets of pointed generalized 
Priestley spaces. 

\begin{lemma} \label{lem: intersection of admissible clopens}
A closed upset $C$ of a pointed generalized Priestley space $X$ is admissible iff it 
is an intersection of admissible clopen upsets.
\end{lemma}

\proof
First suppose that $C$ is an intersection of a family $\mathcal{S}$ of admissible clopen 
upsets. Let $x \notin C$. Then there is $U \in \mathcal{S}$ with $x \notin U$. Since 
$U$ is admissible, there is $y \in X_0 \setminus U$ with $x \le y$. From $y \notin U$ 
and $C \subseteq U$ it follows that $y \notin C$. Therefore, $X \setminus C \subseteq 
\down (X_0 \setminus C)$, hence the equality. Thus, $C$ is admissible.

Conversely, suppose that $C$ is admissible and $x \notin C$. Then there is $y \in X_0 \setminus C$ with 
$x \le y$. For each $z \in C$ we have $z \not\le y$. Therefore, there is an admissible clopen $U_z$ 
with $z \in U_z$ and $y \notin U_z$. Since $C$ is closed, it is compact, and the $U_z$ cover $C$. 
Thus, there are $z_1, \dots, z_n \in C$ such that $C \subseteq U_{z_1} \cup \cdots \cup U_{z_n}$. 
Clearly each $U_{z_i}$ is in $\I_y$. Because $y \in X_0$, $\I_y$ is 
directed, so there is $V \in \I_y$ with each $U_{z_i}$ contained in $V$. Therefore, $C \subseteq V$ and 
$y \notin V$. 
Since $x \le y$, $ y \notin V$, and $V$ is an upset, we have $x \notin V$. Thus, for each $x \notin C$ there 
is an 
admissible clopen upset containing $C$ and missing $x$. Consequently, $C$ is an intersection of admissible 
clopen upsets.
\endproof

\begin{lemma} \label{lem: LX is a DAlgLat}
 If $X\in\PGPries$,  then $(\V(X),\supseteq) \in \AlgFrmJ$ and $\Kp(\V(X)) = \A(X)$.
\end{lemma}

\proof
We first show that $\V(X)$ is a complete lattice. We have $X \in \V(X)$, so 
$\V(X)$ has a bottom. Let $\mathcal S \subseteq \V(X)$ 
and set $C = \bigcap \mathcal S$. Then $C$ is a closed upset. Lemma~\ref{lem: 
intersection of admissible clopens} yields that 
$C$ is admissible, so $C \in \V(X)$. But then 
$C$ is the join of $\mathcal S$ in $\V(X)$, and hence $\V(X)$ is a complete lattice. 

We next show that $\V(X)$ is distributive. Since the order on $\V(X)$ is $\supseteq$, 
it is enough to show that $(A \vee B) 
\wedge C \supseteq (A \wedge C) \vee (B \wedge C)$
for each $A,B,C \in \V(X)$. We 
first show that if $x \in X_0 \setminus ((A \vee B) \wedge C)$, then $x \notin (A \wedge 
C) \vee (B \wedge C)$. Since join in $\V(X)$ is intersection and meet is the least 
admissible closed upset containing the union, from $x \notin (A \vee B) \wedge C$ it 
follows that $x \notin (A \cap B) \cup C$. Therefore, $x \notin (A \cup C) \cap (B \cup C)
$, and so $x\notin A \cup C$ or $x\notin B \cup C$. If $x \notin A \cup C$, 
Lemma~\ref{lem: intersection of admissible clopens} yields that there exist 
$U_{A,x},U_{C,x} \in \A(X)$ such that $A\subseteq U_{A,x}$, $C\subseteq U_{C,x}$, 
and $x\notin U_{A,x},U_{C,x}$. Since $x \in X_0$ and $U_{A,x},U_{C,x}\in\I_x$, there 
is $U_x\in\I_x$ with $U_{A,x},U_{C,x}\subseteq U_x$, so $A \cup C\subseteq U_x$. 
Similarly, if $x \notin B \cup C$, there is $V_x\in\I_x$ such that $B \cup 
C\subseteq V_x$. 
Thus, if $x \in X_0 \setminus ((A \vee B) \wedge C)$, then $x \notin A \wedge C$ or $x 
\notin B \wedge C$, and hence $x \notin (A \wedge C) \vee (B \wedge C)$.
Since $(A \vee B) \wedge C$ is admissible, if $y \notin (A \vee B) \wedge C$, then there 
is $x \in X_0 \setminus ((A \vee B) \wedge C)$ such that $y \le x$, and so 
$y \notin (A \wedge C) \vee (B \wedge C)$.
This proves that $(A \vee B) \wedge C \supseteq (A \wedge C) \vee (B \wedge C)$. 

It is left to prove that $(\V(X),\supseteq)$ is algebraic. By 
Lemma~\ref{lem: intersection of admissible clopens}, $C \in \V(X)$ 
is the intersection of $U \in \A(X)$ containing it. 
Thus, it suffices to show that $\Kp(\V(X)) = \A(X)$. 
First, let $U \in \A(X)$. If $U \le \bigvee \mathcal S$ for some $\mathcal S \subseteq 
\V(X)$, then $\bigcap \mathcal S \subseteq U$. Since $X$ is compact, there is a finite 
subset $\mathcal T$ of $\mathcal S$ with $\bigcap \mathcal T \subseteq U$, so $U \le 
\bigvee \mathcal T$. Therefore, $U \in \Kp(\V(X))$. Conversely, let $C \in \Kp(\V(X))$. 
Since $C = \bigcap \{ U \in \A(X) \mid C \subseteq U\}$ and $C$ is compact, there exist 
$U_1, \ldots, U_n \in \A(X)$ such that $C = U_1 \cap \cdots \cap U_n$. Thus, $C \in 
\A(X)$, so $\Kp(\V(X)) = \A(X)$.
Consequently, $(\V(X),\supseteq)$ is an algebraic frame.
\endproof
 
The above lemma defines $\V$ on objects. We next define $\V$ on morphisms.

\begin{lemma} \label{lem: Box R is a morphism}
Let $X,Y\in\PGPries$. If $R \subseteq X \times Y$ is a generalized Priestley morphism, 
then $\Box_R \colon \V(Y) \to \V(X)$ is a well-defined $\AlgFrmJ$-morphism.
\end{lemma}

\proof
To see that $\Box_R$ is well defined, let $C \in \V(Y)$. Then 
$C = \bigcap \{ U \in \A(X) \mid C \subseteq U \}$ by 
Lemma~\ref{lem: intersection of admissible clopens}. Since 
$\Box_R$ preserves arbitrary intersections, 
\[
\Box_R C = \Box_R \bigcap \{ U \in \A(X) \mid C \subseteq U \} = 
\bigcap \{ \Box_R U \mid U \in \A(X) \mbox{ and } C \subseteq U \},
\]
which is an intersection of admissible clopen upsets. Thus, $\Box_R C \in \V(X)$ by 
Lemma~\ref{lem: intersection of admissible clopens}.
Since joins in $\V(X)$ are intersections, it is clear that $\Box_R$ preserves arbitrary 
joins. Finally, by Lemma~\ref{lem: LX is a DAlgLat}, compact elements of $\V(X)$ are 
admissible clopen upsets. Therefore, $\Box_R$ preserves compact elements by the 
definition of a generalized Priestley morphism. Thus, $\Box_R$ is an 
$\AlgFrmJ$-morphism.
\endproof

\begin{proposition} \label{prop: functor from GPries to DAlgLat}
There is a contravariant functor $\V\colon \PGPries \to \AlgFrmJ$ which sends 
$X\in\PGPries$ to $\V(X)$ and a $\PGPries$-morphism $R$ to $\Box_R$.
\end{proposition}

\proof
By Lemmas~\ref{lem: LX is a DAlgLat} and \ref{lem: Box R is a morphism}, $\V$ is well 
defined. If $R$ is the identity morphism on $X\in\PGPries$, then $R$ is $\le$. 
Therefore, for $U \in \V(X)$, we have
\[
\Box_R U = \{ x \in X \mid R[x] \subseteq U\} = \{ x \in X \mid \up x \subseteq U \} = U.
\]
Thus, $\Box_R$ is the identity morphism on $\V(X)$. Next, let $R \subseteq X \times 
Y$ and $S \subseteq Y \times Z$ be $\PGPries$-morphisms. We show that $\Box_R 
\circ \Box_S = \Box_{S*R}$. Let $C \in \V(Z)$.
By Lemma~\ref{lem: intersection of admissible clopens}, $C = \bigcap \{ U \in \A(Z) \mid 
C \subseteq U \}$.
Since for each $U \in \A(Z)$ we have $\Box_R \Box_S U = \Box_{S*R} U$ (see 
\cite[p.~106]{BJ11}) and $\Box_R,\Box_S$ commute with arbitrary intersections, we 
obtain
\begin{align*}
\Box_R \Box_S C &= \bigcap \{ \Box_R\Box_S U \mid U \in \A(Z) \mbox{ and } C 
\subseteq U\} \\
&= \bigcap \{ \Box_{S*R} U \mid U \in \A(Z) \mbox{ and } C \subseteq U\} = \Box_{S*R} 
C.
\end{align*}
Thus, $\Box_R \circ \Box_S = \Box_{S*R}$, 
and hence $\V$ is a contravariant functor.
\endproof

\subsection*{\textbf{The functor $\Y$}.} \label{subsec: Y} \label{subsec: 4.2}

Let $L$ be a complete lattice. We recall (see, e.g., \cite[p.~50]{GHKLMS03}) that the 
{\em way below} relation on $L$ is defined by $a \ll b$ if, for each $S \subseteq L$, 
from $b \le \bigvee S$ it follows that $a \le \bigvee T$ for some finite $T \subseteq 
S$. We also recall (see, e.g., \cite[p.~54]{GHKLMS03}) that a complete lattice $L$ is 
\emph{continuous} if $b = \bigvee \{ a \in L \mid a \ll b\}$ for each $b \in L$. Since $a 
\in L$ is compact iff $a \ll a$, for $a \in \Kp(L)$, we have $a \le b$ iff $a \ll b$. 
Therefore, every algebraic lattice is a continuous lattice.

Let $p \in L \setminus \{ 1 \}$. We recall  that $p$ is {\em $($meet-$)$prime} if $a 
\wedge b \le p$ implies $a \le p$ or $b \le p$. In distributive lattices prime elements are 
exactly (meet-)irreducible elements, where $p$ is \emph{irreducible} if $a \wedge b = p$ 
implies $a = p$ or $b = p$. Next, recall that $p$ is \emph{pseudoprime} if for each 
$n \ge 1$, from
$a_1 \wedge \cdots \wedge a_n \ll p$ it follows that $a_i \le p$ for some $i$ (see 
\cite[Prop.~I-3.25]{GHKLMS03}). 

\begin{definition}
For a complete lattice $L$, let $\p(L)$ be the set of prime elements and $\PP(L)$ the set 
of pseudoprime elements of $L$. 
\end{definition}

\begin{definition}
For an algebraic frame $L$, let $Y_L = \PP(L)  \cup \{1\}$.
\end{definition}

As we pointed out in Section~\ref{sec: pontryagin}, the topology $\lambda$ 
generated by the subbasis
\[
\{ \up k \mid k \in \Kp(L)\} \cup \{ (\up l)^c \mid l \in \Kp(L)\}
\]
turns $L$ into a Stone meet-semilattice. This, in particular, implies that 
$\langle L,\lambda, \le \rangle$ is a Priestley space. 
We restrict the topology and order on $L$ to $Y_L$. We then have:

\begin{lemma} \label{lem: X is Priestley}
If $L$ is an algebraic frame, then $\langle Y_L,\lambda,\le \rangle$ is a Priestley space. 
\end{lemma}

\proof
Since $L$ is a Priestley space, it suffices to show that $Y_L$ is a closed subset of 
$L$. Let $a \in L \setminus Y_L$. Then $a \ne 1$ and $a \notin \PP(L)$. Therefore, 
there are $b_1, \dots, b_n \in L$ with $b_1 \wedge \cdots \wedge b_n \ll a$ and 
$b_i\not\le a$ for each $i$. Because $L$ is an algebraic lattice, there are $k_i \in \Kp(L)$ 
with $k_i \le b_i$ but $k_i \not\le a$. Clearly $k_1 \wedge \cdots \wedge k_n \ll a$. 
Since $a = \bigvee \left(\down a \cap \Kp(L)\right)$, there are 
$t_1, \dots, t_p \in \down a \cap \Kp(L)$ 
with $k_1 \wedge \cdots \wedge k_n \le t_1 \vee \cdots \vee t_p$. Let $t = t_1 \vee 
\cdots \vee t_p$. Then $t \in \down a \cap \Kp(L)$ and $k_1 \wedge \cdots \wedge k_n 
\le t$. Therefore, $\up t \cap (\up k_1)^c \cap \cdots \cap (\up k_n)^c$ is an open 
neighborhood of $a$. Moreover, if $x \in \up t \cap (\up k_1)^c \cap \cdots \cap (\up 
k_n)^c $, then $k_1 \wedge \cdots \wedge k_n \le t \le x$ and $k_1, \dots, k_n \not\le 
x$. Since $t$ is compact, $ k_1\wedge \cdots \wedge k_n \le t \le x$ implies $k_1 
\wedge \cdots \wedge k_n \ll x$. Thus, $k_1 \wedge \cdots \wedge k_n \ll x$ but $k_i 
\not\le x$ for each $i$. Since $x \ne 1$, we conclude that $x \notin Y_L$. 
This implies that $\up t \cap (\up k_1)^c \cap \cdots \cap (\up k_n)^c$ 
is an open neighborhood of $a$ that 
misses $Y_L$. Thus, $Y_L$ is a closed subset of $L$.
\endproof

\begin{remark}
The proof of Lemma~\ref{lem: X is Priestley} does not require that $L$ is a frame. It 
only requires that $L$ is an algebraic lattice.
\end{remark}

We next show that the tuple $\langle Y_L, \lambda, \le, \p(L), 1 \rangle$ is a pointed 
generalized Priestley space. For this we need the following lemma. 
The proof of \ref{primes} can for 
example be found in \cite[Cor.~I-3.10]{GHKLMS03}.
We include the proof of \ref{clopens} because we were unable to find a reference for it. 

\begin{lemma} \label{lem: existence of primes and clopens}
Let $L$ be an algebraic frame.
\begin{enumerate}[$(1)$]
\item \label{primes} If $a, b \in L$ with $a \not\le b$, then there is 
$p \in \p(L)$ with $a \not\le p$ and $b \le p$. 
\item \label{clopens} Let $U$ be a clopen upset of $L$. Then 
$U = \up k_1 \cup \cdots \cup \up k_n$ for some $k_i \in \Kp(L)$.
\end{enumerate}
\end{lemma}

\proof
\ref{primes}. In a continuous lattice each element is a meet of irreducible elements  
(see, e.g., \cite[Cor.~I-3.10]
{GHKLMS03}). Since every algebraic lattice is continuous and irreducible elements are 
prime in distributive lattices, the result follows.

\ref{clopens}. We first show that $U = \bigcup \{ \up k \mid k \in \Kp(L) \cap U\}$. 
Let $a \in U$. 
Since $U$ is an upset, $\up a \subseteq U$, so $\up a \cap U^c = \varnothing$. Because 
$L$ is an algebraic lattice, $\up a = \bigcap \{ \up k \mid k \in \Kp(L), k \le a\}$. 
Therefore, $U^c \cap \bigcap \{ \up k \mid k \in \Kp(L), k \le a\} = \varnothing$. Since 
$U^c$ is closed, compactness of $L$ shows there are $k_1, \dots, k_n \in \Kp(L)$ with 
$k_i \le a$ and $U^c \cap \up k_1 \cap \cdots \cap \up k_n = \varnothing$. If $k = k_1 
\vee \cdots \vee k_n$, then $k \in \Kp(L)$, $k \le a$, and $\up k = \up k_1 \cap \cdots 
\cap \up k_n$, so $\up k \subseteq U$. Thus, $U = \bigcup \{ \up k \mid k \in \Kp(L) \cap 
U\}$, as desired. Since $U$ is closed, hence compact, this union is a finite union, 
completing the proof.
\endproof

\begin{lemma} \label{lem: dense}
$\p(L)$ is dense in $\PP(L) = Y_L \setminus \{1\}$.
\end{lemma}

\proof
Since $\{ \up k \mid k \in \Kp(L)\}$ is closed under finite intersections, 
\[
\{ \up k \cap (\up l_1)^c \cap \cdots \cap (\up l_n)^c \mid k, l_1, \dots, l_n \in \Kp(L) \}
\]
is a basis for the topology on $Y_L$. Therefore, it suffices to show that if $k, l_1, \dots, 
l_n \in \Kp(L)$ and $U = \up k \cap (\up l_1)^c \cap \cdots \cap (\up l_n)^c$ with $U \cap 
\PP(L) \ne \varnothing$, then $U \cap \p(L) \ne \varnothing$. Let $q \in \PP(L)$ with $q 
\in U$. Then $k \le q$ and each $l_i \not\le q$. If $l_1 \wedge \cdots \wedge l_n \le 
k$, then $l_1 \wedge \cdots \wedge l_n \ll q$ since $k$ is compact. This contradicts $q 
\in \PP(L)$. Therefore, $l_1 \wedge \cdots \wedge l_n \not\le k$, so by 
Lemma~\ref{lem: existence of primes and clopens}\ref{primes} there is $p \in \p(L)$ 
with $k \le p$ and $l_1 \wedge \cdots \wedge l_n \not\le p$, so each $l_i \not\le p$. 
Thus, $p \in U$, and hence $U \cap \p(L) \ne \varnothing$.
\endproof

\begin{lemma} \label{lem: admissible}
Let $U$ be a clopen upset of $Y_L$. Then $U$ is admissible iff $U = \up k \cap Y_L$ 
for some $k \in \Kp(L)$.
\end{lemma}

\proof
First suppose that $U = \up k \cap Y_L$ for some $k \in \Kp(L)$. If 
$a \in Y_L \setminus U$, then $k 
\not\le a$. By Lemma~\ref{lem: existence of primes and clopens}\ref{primes},
there is $p \in \p(L)$ with $k \not\le p$ and $a \le p$. Thus, 
$p \in \p(L) \setminus U$. This shows that $U$ is admissible.

Conversely, let $U$ be admissible. Since $U$ is a clopen upset of  $Y_L$ and $Y_L$ 
is a closed subspace of $L$, there is a 
clopen upset $V$ of $L$ with $U = V \cap Y_L$ (see \cite[Lem.~4.4]{BH21}). By 
Lemma~\ref{lem: existence of primes and clopens}\ref{clopens},
there are $k_i \in \Kp(L)$ with $V = \up k_1 \cup \cdots \cup \up k_n$. 
Consequently, $U = (\up k_1 \cup \cdots \cup \up k_n) \cap Y_L$. 
If $k_1 \wedge \cdots \wedge k_n \notin U$, then since $U$ is admissible, there is $p 
\in \p(L) \setminus U$ with $k_1 \wedge \cdots \wedge k_n \le p$. 
But then $k_i \le p$ for some $i$, yielding $p \in U$, which is false. Therefore, $k_1 
\wedge \cdots \wedge k_n \in U$, and so $k_i \le k_1 \wedge \cdots \wedge k_n$ for 
some $i$. Thus, $k_i \le k_j$ for each $j$, and hence $U = \up k_i \cap Y_L$. 
\endproof

\begin{lemma} \label{lem: order}
For $a, b \in Y_L$ we have $a \le b$ iff $\forall U \in \A(Y_L),\, a \in U \Longrightarrow b \in U$. 
\end{lemma}

\proof
Suppose that $a \le b$. Let $U\in\A(Y_L)$. Since $U$ is an upset of $Y_L$, if $a \in 
U$, then $b \in U$.
Conversely, suppose that $a \not\le b$. Since $\Kp(L)$ is join-dense in $L$, there is 
$k \in \Kp(L)$ with $k \le a$ and $k \not\le b$. Let $U = \up k \cap Y_L$. Then $U$ is admissible by 
Lemma~\ref{lem: admissible}, $a \in U$, and $b \notin U$.
\endproof

\begin{lemma} \label{lem: directed}
For $a \in Y_L$ we have that $\I_a$ is nonempty and directed iff $a \in \p(L)$.
\end{lemma}

\proof
Let $a \in \p(L)$. Then $a \ne 1$, so there is $k \in \Kp(L)$ with $k \not\le a$. By 
Lemma~\ref{lem: admissible}, $\up k \cap Y_L$ is admissible and does not contain $a$. 
Therefore, $\I_a$ is nonempty. To see it is directed, let $U, V \in \I_a$. 
Then there are $k, l \in \Kp(L)$ with $U = \up k \cap Y_L$ and $V = \up l \cap Y_L$. 
Therefore, $k, l \not\le a$, so $k \wedge l \not \le a$ since $a \in \p(L)$. 
Thus, there is $t \in \Kp(L)$ with 
$t \le k \wedge l$ and $t \not\le a$. Then $\up t \cap Y_L \in \I_a$ and contains both 
$U, V$. This proves that $\I_a$ is directed.

Conversely, suppose that $\I_a$ is nonempty and directed. Since $\I_a$ is nonempty, 
there is $k \in \Kp(L)$ with $\up k \cap Y_L \in \I_a$ by Lemma~\ref{lem: admissible}. 
Thus, $k \not\le a$, and so $a \ne 1$.
If $a \notin \p(L)$, then there are $x, y \in L$ with $x \wedge y \le a$ and 
$x, y \not\le a$. The latter implies that there are $k, l \in \Kp(L)$ with $k \le x$, $l \le 
y$, and $k, l \not\le a$. Set $U = \up k \cap Y_L$ and $V = \up l \cap Y_L$. Then $U, V 
\in \I_a$. Therefore, there is $W \in \I_a$ with $U, V \subseteq W$. By 
Lemma~\ref{lem: admissible}, there is $t \in \Kp(L)$ with $W = \up t\cap Y_L$. But then 
$t \le k, l$, so $t \le k \wedge l \le a$, and hence $W \notin \I_a$. 
The obtained contradiction proves that $a \in \p(L)$.
\endproof

\begin{proposition} \label{prop: X is GPries}
If $L$ is an algebraic frame, then $\Y(L) := \langle Y_L, \lambda,\le, \p(L), 1\rangle$ 
is a pointed generalized Priestley space.
\end{proposition}

\proof
By Lemma~\ref{lem: X is Priestley}, $\langle Y_L,\lambda,\le \rangle$ is a Priestley space. It 
is clear that $1$ is the unique maximum of $Y_L$. By Lemma~\ref{lem: dense}, $\p(L)$ is 
dense in $Y_L \setminus \{1\}$. By 
Lemma~\ref{lem: existence of primes and clopens}\ref{primes}, $\p(L)$ is 
order-dense in $Y_L \setminus \{1\}$. By Lemma~\ref{lem: directed}, $a \in \p(L)$ iff 
$\I_a$ is nonempty and directed. Finally, by Lemma~\ref{lem: order}, $a \le b$ iff $
\forall U \in \A(Y_L)\,(a \in U \Rightarrow b \in U)$. Thus, $\Y(L)$ is a pointed 
generalized Priestley space.
\endproof

We next turn to morphisms.

\begin{definition}\label{def:Ralpha}
Let $L_1, L_2$ be algebraic frames and $\alpha \colon L_1 \to L_2$ an 
$\AlgFrmJ$-morph\-ism. 
If $r$ is the right adjoint of $\alpha$, we define 
$R_\alpha \subseteq Y_{L_2} \times Y_{L_1}$ by 
$p \rel{R_\alpha} q$ if $r(p) \le q$ in~$L_1$. 
\end{definition}

\begin{remark}
It is not necessarily the case that if $p \in Y_{L_2}$, then $r(p) \in Y_{L_1}$. In 
Lemma~\ref{lem: sup = strong WS} we will show exactly when 
$r[Y_{L_2}] \subseteq Y_{L_1}$. 
\end{remark}

\begin{lemma} \label{lem: R is a GPries morphism}
Let $\alpha \colon L_1 \to L_2$ be an $\AlgFrmJ$-morphism. Then 
$R_\alpha \subseteq Y_{L_2} \times Y_{L_1}$ is a generalized Priestley morphism.
\end{lemma}

\proof
Let $p \in Y_{L_2}$ and $q \in Y_{L_1}$ with $p \nr{R_\alpha} q$. Then 
$r(p) \not\le q$, so there is $k \in \Kp(L_1)$ with $k \le r(p)$ and $k \not\le q$. 
If $U = \up k \cap Y_{L_1}$, then $U$ is admissible by Lemma~\ref{lem: admissible} and $q 
\notin U$. Let $q' \in R_\alpha[p]$. Then $r(p) \le q'$. Therefore, $k \le q'$, so 
$q' \in U$. Thus, $R_\alpha[p] \subseteq U$. This shows that $R_\alpha$ satisfies 
Definition~\ref{def: GPS morphism}\ref{GPS morphism 1}. 
To see that $R_\alpha$ also satisfies 
Definition~\ref{def: GPS morphism}\ref{GPS morphism 2}, suppose that $U 
\in \A(Y_{L_1})$. Then $U = \up k \cap Y_{L_1}$ for some $k \in \Kp(L_1)$ by 
Lemma~\ref{lem: admissible}. 

\begin{claim} \label{claim: calc}
$\Box_{R_\alpha}U = \up \alpha(k) \cap Y_{L_2}$.
\end{claim}

\proof 
Suppose that $p \in \up \alpha(k) \cap Y_{L_2}$. Then $\alpha(k) \le p$, so 
$k \le r(p)$. Let $q \in Y_{L_1}$ with $p\rel{R_\alpha} q$. Then $r(p) \le q$, so 
$k \le q$, and hence $q \in U$. 
Therefore, $R_\alpha[p] \subseteq U$, yielding that $p \in \Box_{R_\alpha} U$. This 
proves that $\up \alpha(k) \cap Y_{L_2} \subseteq \Box_{R_\alpha} U$. If $p \notin \up 
\alpha(k) \cap Y_{L_2}$, then $\alpha(k) \not\le p$, so $k \not\le r(p)$. By 
Lemma~\ref{lem: existence of primes and clopens}\ref{primes}, 
there is $q \in \p(L_1)$ with $k \not\le q$ and $r(p) \le q$. Therefore, 
$p\rel{R_\alpha} q$ and 
$q \notin U$. Thus, $p \notin \Box_{R_\alpha}U$, and so $\Box_{R_\alpha}U = \up 
\alpha(k) \cap Y_{L_2}$.
\endproof

The claim together with Lemma~\ref{lem: admissible} shows that $\Box_{R_\alpha}U \in 
\A(Y_{L_2})$. Thus, $R_\alpha$ is a generalized Priestley morphism.
\endproof

\begin{proposition} \label{prop: functor from DAlgLat to GPries}
There is a contravariant functor $\Y \colon \AlgFrmJ \to \PGPries$ which sends $L \in 
\AlgFrmJ$ to $\Y(L)$ and an $\AlgFrmJ$-morphism $\alpha$ to $R_\alpha$.
\end{proposition}

\proof
Proposition~\ref{prop: X is GPries} and Lemma~\ref{lem: R is a GPries morphism} show 
that $\Y$ is well defined. 
If $\alpha$ is the identity on $L$, then $p \rel{R_\alpha} q$ iff $r(p) \le q$ iff $p \le q$, 
so $R_\alpha$ is equal to $\le$. Therefore, $\Y$ sends identity morphisms to identity 
morphisms. To show that $\Y$ preserves composition, let $\alpha \colon L_1 \to L_2$ and $
\beta \colon L_2 \to L_3$ be $\AlgFrmJ$-morphisms. Write 
$R = R_\beta$, $S = R_\alpha$, 
and $T = R_{\beta\alpha}$. We show that $T = S * R$.

First suppose that $x\rel{T}z$, so $r_{\beta\alpha}(x) \le z$, and hence 
$r_\alpha r_\beta(x) 
\le z$ because $r_{\beta\alpha} = r_{\alpha} r_\beta$. Let $U$ be an 
admissible clopen upset of $Y_{L_3}$. By Lemma~\ref{lem: admissible}, 
$U = \up k \cap Y_{L_3}$ for some $k \in \Kp(L_3)$. 
By Claim~\ref{claim: calc}, 
$\Box_R \Box_S U = \Box_R (\up \alpha(k) \cap Y_{L_2}) = \up \beta\alpha(k) 
\cap Y_{L_1}$. Therefore, 
\[
x \in \Box_R \Box_S U \Longrightarrow \beta\alpha(k) \le x \Longrightarrow k \le 
r_\alpha r_\beta(x) \le z,
\] 
so $z \in U$. Thus, $x\rel{(S*R)}z$.
 
Conversely, if $x \nr{T} z$, then there is $U \in \A(Y_{L_3})$ with $T[x] \subseteq 
U$ and $z \notin U$. We show that $x \in \Box_R \Box_S U$. Let $x\rel{R}y$ and 
$y\rel{S}t$. Then $r_\beta (x) \le y$ and $r_\alpha(y) \le t$, so 
$r_\alpha r_\beta(x) \le t$. Therefore, 
$r_{\beta \alpha}(x) \le t$, so $x\rel{T}t$, and hence $t \in T[x]$. This yields 
$t \in U$. Thus, $x \in \Box_R \Box_S U$ as desired. Since $z \notin U$, 
it is false that $x 
\rel{(S*R)}z$. This shows that $R_{\beta\alpha} = R_\alpha * R_\beta$. Consequently,
 $\Y$ is a contravariant functor.
\endproof

\subsection*{\textbf{Dual equivalence of $\AlgFrmJ$ and $\PGPries$}.} 
\label{subsec: dual equivalence}

We now show that the functors $\V$ and $\Y$ yield a dual equivalence
of $\AlgFrmJ$ and $\PGPries$. To do so, we 
produce natural isomorphisms $\Upsilon \colon 1_{\PGPries} \to \Y\circ \V$ and 
$\eta \colon 1_{\AlgFrmJ} \to \V \circ \Y$.
  
Let $X \in \PGPries$. Define $\varepsilon_X \colon X \to Y_{\V(X)}$ by 
$\varepsilon_X(x) = \up x$ for each $x \in X$.

\begin{lemma} \label{lem: varepsilon is an iso}
For $X \in \PGPries$, we have$:$ 
\begin{enumerate}[$(1)$]
\item \label{varepsilon 1} $\varepsilon_X$ is well defined. 
\item \label{varepsilon 2} $\varepsilon_X$ is an order-isomorphism.
\item \label{varepsilon 3} $\varepsilon_X$ is a homeomorphism.
\item \label{varepsilon 4} $\varepsilon_X[X_0] = (Y_{\V(X)})_0$. 
\end{enumerate}
\end{lemma}

\proof
\ref{varepsilon 1}.  
Let $x\in X$.  For each $y \notin \up x$ we have $x \not\le y$, so there is $U \in \A(X)
$ with $x \in U$ and $y \notin U$ by Definition~\ref{def: PGPries}\ref{PGPS separation}. 
Thus, $\up x = \bigcap\{ U \in \A(X) \mid x \in U \}$, 
and hence $\up x \in \V(X)$.

We next show that $\up x \in Y_{\V(X)}$. If $\up x = X$, then this is trivial. 
Otherwise, we show
that $\up x$ is pseudoprime.
Suppose that $C_i \in \V(X)$ with $C_1 \wedge \cdots \wedge C_n \ll \up x$. Since $
\up x = \bigcap \{ U \in \A(X) \mid x \in U\}$, recalling that $\V(X)$ is ordered by $
\supseteq$, there is $U \in \A(X)$ with $C_1 \wedge \cdots \wedge C_n \le U \le \up 
x$. Therefore, $x \in U \subseteq C_1 \wedge \cdots \wedge C_n$.  We claim that $x 
\in C_1 \cup \cdots \cup C_n$. If not, then there are $V_i \in \A(X)$ with $C_i 
\subseteq V_i$ and $x \notin V_1 \cup \cdots \cup V_n$. Thus, $U \subseteq C_1 
\wedge \cdots \wedge C_n \subseteq V_1 \wedge \cdots \wedge V_n$.  
We show this implies that $U \subseteq V := V_1 \cup \cdots \cup V_n$. 

We first show $U \cap X_0 \subseteq V$. Let $y \in U \cap X_0$. If $y \notin V$, then 
$V_i \in \I_y$ for each $i$. Since $y \in X_0$, there is $W \in \I_y$ 
with $V \subseteq W$. 
Therefore, $V_1, \ldots, V_n \subseteq W$, and so $W \le V_1, \ldots, V_n$, which 
implies $W \le V_1 \wedge \cdots \wedge V_n$. Thus, $V_1 \wedge \cdots \wedge V_n 
\subseteq W$, and hence $U \subseteq W$. This is a contradiction since $y \in U$. 
Consequently, $U \cap X_0 \subseteq V$.

Now, if $U \not\subseteq V$, then $U \cap V^c$ is a nonempty open subset of $X$, and 
$U \cap V^c \subseteq X \setminus \{m\}$ since $m \in V$. Because $X_0$ is dense in 
$X\setminus \{m\}$, we have $X_0 \cap U \cap V^c \ne \varnothing$. This contradicts 
the inclusion $U \cap X_0 \subseteq V$. Therefore, $U \subseteq V = V_1 \cup \cdots 
\cup V_n$. But this is false since $x \in U$. Thus, $x \in C_1 \cup \cdots \cup C_n$, 
proving the claim. This implies that $C_i \le \up x$ for some $i$, so $\up x \in Y_{\V(X)}
$, and hence $\varepsilon_X$ is well defined.

\ref{varepsilon 2}. It is clear that $\varepsilon_X$ is order preserving and 
order reflecting. Thus, 
$\varepsilon_X$ is 1-1. To see that it is onto, let $C \in Y_{\V(X)}$. If $C$ is the top of 
$Y_{\V(X)}$, then $C = \up m$, and hence $C = \varepsilon_X(m)$. Suppose that $C$ is a 
pseudoprime. We show that $\min C$ is a singleton. Otherwise
for each distinct pair $x,y \in \min C$ there is $U_{xy} \in \A(X)$ with $x \in U_{xy}$ 
and $y \notin U_{xy}$. Therefore, $C \not\subseteq U_{xy}$. The various $U_{xy}$ 
cover $C$. By Lemma~\ref{lem: intersection of admissible clopens}, $C$ is an intersection 
from $\A(X)$, so by compactness, 
there are $V \in \A(X)$ and distinct pairs $x_1, y_1, \dots, x_n, y_n \in \min C$ with 
$C \subseteq V \subseteq U_1 \cup \cdots \cup U_n$, where we write $U_i$ for 
$U_{x_iy_i}$. 
In $\V(X)$ this says that $U_1 \wedge \cdots \wedge U_n \le V \le C$. 
By Lemma~\ref{lem: LX is a DAlgLat}, $V \in \Kp(\V(X))$, implying that 
$U_1 \wedge \cdots \wedge U_n \ll C$. Because 
$C$ is a pseudoprime, this forces $U_i \le C$ for some $i$, so $C \subseteq U_i$. This 
is false by construction of the $U_i$. Therefore, $\min C = \{x\}$, and so $C = \up x$ 
for some $x \in X$. Thus, $\varepsilon_X$ is onto, hence an order-isomorphism. 

\ref{varepsilon 3}. Since $X$ and $Y_{\V(X)}$ are compact Hausdorff,
by \ref{varepsilon 2} it is sufficient to show 
that $\varepsilon_X$ is continuous. For this, since the topology on $Y_{\V(X)}$ is 
generated by clopen upsets and their complements, by 
Lemmas~\ref{lem: existence of primes and clopens}\ref{clopens} 
and~\ref{lem: admissible} it suffices to show that if $V$ is an admissible clopen upset of 
$Y_{\V(X)}$, then $\varepsilon_X^{-1}(V)$ is clopen. By Lemmas~\ref{lem: LX is a 
DAlgLat} and~\ref{lem: admissible}, there is $U \in \A(X)$ with $V = \up U \cap 
Y_{\V(X)} = \{ W \in Y_{\V(X)} \mid U \le W \}$. We have
\begin{align*}
\varepsilon_X(x)^{-1}(V) &= \varepsilon_X(x)^{-1}(\up U \cap Y_{\V(X)}) = \{ x\in X \mid 
\up x \in \up U \cap Y_{\V(X)} \} \\
&= \{ x\in X \mid U \le \up x\} = \{ x\in X \mid \up x \subseteq U\} = U.
\end{align*}
Thus, $\varepsilon_X$ is continuous, hence a homeomorphism.

\ref{varepsilon 4}. First let $x \in X_0$. Then $x \ne m$, so $\up x$ is 
not the top of $\V(X)$. 
Suppose that $C, D \in \V(X)$ with $C \wedge D \le \up x$, so $x \in C \wedge D$. If 
$x \notin C \cup D$, then there are $U, V \in \A(X)$ with $C \subseteq U$, $D 
\subseteq V$, and $x \notin U \cup V$. Therefore, $U, V \in \I_x$, so from $x \in X_0$ 
it follows that there is $W \in \I_x$ with $U \cup V \subseteq W$. But $C \wedge D$ is 
the intersection of all $U' \in \A(X)$ with $C \cup D \subseteq U'$. This contradicts $x 
\in C \wedge D$. Thus, $x \in C$ or $x \in D$, so $C \le \up x$ or $D \le \up x$. 
Consequently, $\up x \in \p(\V(X)) =(Y_{\V(X)})_0$.

Conversely, let $C \in (Y_{\V(X)})_0$. Since $\varepsilon_X$ is onto, $C = \up x$ 
for some $x 
\in X$. If $x \notin X_0$, then either $\I_x = \varnothing$ or $\I_x$ is not directed. 
If $\I_x = \varnothing$, then $x = m$, so $\up x$ is the top of $\V(X)$, and hence not 
in $(Y_{\V(X)})_0$, a contradiction. If $\I_x$ is not directed, then there are 
$U, V \in \I_x$ 
but no larger admissible clopen upset misses $x$. Since the meet $U \wedge V$ in $\V(X)
$ is the intersection of all admissible clopen upsets containing $U \cup V$, all such contain 
$x$, and so $x \in U \wedge V$. Therefore, $U \wedge V \le \up x$. This contradicts $
\up x \in (Y_{\V(X)})_0$. Thus, $x \in X_0$. 
\endproof

\begin{definition}
Let $X \in \PGPries$. Define $\Upsilon_X \subseteq X \times Y_{\V(X)}$ by 
$x \rel{\Upsilon_X} C$ iff $
\varepsilon_X(x) \le C$ in $\V(X)$.
\end{definition}

To prove that $\Upsilon \colon 1_{\PGPries} \to \Y\circ\V$ is a natural isomorphism, 
we require the following two lemmas. 

\begin{lemma} \label{lem: R coming from f}
Let $X, Y \in \PGPries$ and suppose that $f \colon X \to Y$ is an order-isomorphism and 
homeomorphism with $f[X_0] = Y_0$. Define $R_f \subseteq X \times Y$ by 
$x\rel{R_f} y$ if $f(x) \le y$.
\begin{enumerate}[$(1)$]
\item \label{Rf 1} $R_f$ is a generalized Priestley morphism.
\item \label{Rf 2} If $g \colon Y \to Z$ is another map satisfying the same 
hypotheses as $f$, then 
$R_{g\circ f} = R_g * R_f$.
\item \label{Rf 3} $R_f$ is a $\PGPries$-isomorphism.
\end{enumerate}
\end{lemma}

\proof
\ref{Rf 1} and \ref{Rf 2} follow from the same proof as \cite[Lem.~9.3]{BJ08} 
since $f$ is a strong Priestley morphism 
(see Definition~\ref{def: functional and strong morphisms}\ref{strong}).
To see \ref{Rf 3}, for the identity morphism $1_X \colon X \to X$ we have 
$x\rel{R_{1_X}} y$ iff $x \le y$, 
so $R_{1_X}$ is equal to $\le$, an identity morphism in $\PGPries$. Since $f$ is 
an order-isomorphism and
homeomorphism, $R_{f^{-1}}$ is a generalized Priestley morphism by \ref{Rf 1}, 
and it is the inverse
of $R_f$ by \ref{Rf 2}. Therefore, $R_f$ is a $\PGPries$-isomorphism.
\endproof

\begin{lemma} \label{lem: ugly calculations}
Let $X \in \PGPries$. For $U \in \A(X)$ set $V = \up U \cap Y_{\V(X)}$. Then $V \in 
\A(Y_{\V(X)})$ and $\Box_{\Upsilon_X} V = U$.
\end{lemma}

\proof
That $V \in \A(Y_{\V(X)})$ follows from Lemmas~\ref{lem: LX is a DAlgLat} 
and~\ref{lem: admissible}. Moreover,
\begin{align*}
\Box_{\Upsilon_X} V &= \{ x \in X \mid \Upsilon_X[x] \subseteq V\} = 
\{ x \in X \mid x\rel{\Upsilon_X} C 
\Rightarrow C \in V\} \\
&= \{ x \in X \mid x \rel{\Upsilon_X} C \Rightarrow C \subseteq U\} = 
\{ x \in X \mid C \subseteq \up x \Rightarrow C \subseteq U\} \\
&= \{ x \in X \mid \up x \subseteq U \} = U,
\end{align*}
where the first equality on the last line holds since $\up x$ is admissible by 
Lemma~\ref{lem: varepsilon is an iso}\ref{varepsilon 1}. 
\endproof

\begin{proposition} \label{prop: R is natural}
Let $X \in \PGPries$. Then $\Upsilon_X$ is a generalized Priestley morphism and 
$\Upsilon \colon 1_{\PGPries} \to \Y \circ \V$ is a natural isomorphism.
\end{proposition}

\proof
Let $X \in \PGPries$. 
By Lemma~\ref{lem: varepsilon is an iso}, $\varepsilon_X \colon X \to Y_{\V(X)}$ is an 
order-isomorphism and homeomorphism such that $\varepsilon_X[X_0] = (Y_{\V(X)})_0$. 
Therefore, $\Upsilon_X$ is a $\PGPries$-isomorphism by 
Lemma~\ref{lem: R coming from f}.

It is left to show that $\Upsilon$ is natural. Let $R \subseteq X \times Y$ be a 
generalized Priestley morphism. We must show that 
$\Upsilon_{Y} * R = \Y\V(R) * \Upsilon_{X}$. Since $\V(R) = 
\Box_R$ and $\Y\V(R) = R_{\Box_R}$, we must show that $\Upsilon_{Y} * R = 
R_{\Box_R} * \Upsilon_{X}$.
\[
\begin{tikzcd}
X \arrow[r, "R"] \arrow[d, "\Upsilon_{X}"'] & Y \arrow[d, "\Upsilon_{Y}"]\\
Y_{\V(X)} \arrow[r, "R_{\Box_R}"'] & Y_{\V(Y)}
\end{tikzcd}
\]

Let $x \in X$ and $C \in Y_{\V(Y)}$. By Lemma~\ref{lem: ugly calculations}, quantifying 
$V = \up U \cap Y_{\V(Y)}$, we have
\begin{align*}
x\rel{(\Upsilon_{Y} * R)} C &\Leftrightarrow (\forall V) (x \in \Box_R\Box_{\Upsilon_{Y}}V 
\Rightarrow C \in V) \\
&\Leftrightarrow (\forall U)  (x \in \Box_R U \Rightarrow C \subseteq U).
\end{align*}

On the other hand, $\Box_R$ is an $\AlgFrmJ$-morphism by 
Lemma~\ref{lem: Box R is a morphism}.
By Claim~\ref{claim: calc}, $\Box_{R_{\Box_R}} V = (\up \Box_R U) \cap Y_{\V(X)}$.
Therefore, applying Lemma~\ref{lem: ugly calculations}, we obtain
\begin{align*}
x \rel{(R_{\Box_R} * \Upsilon_{X})} C &\Leftrightarrow  (\forall V) 
(x \in \Box_{\Upsilon_{X}} 
\Box_{R_{\Box_R}} V \Rightarrow C \in V) \\
&\Leftrightarrow (\forall U) (x \in \Box_{\Upsilon_{X}} ((\up \Box_R U) \cap Y_{\V(X)}) 
\Rightarrow C \subseteq U) \\
&\Leftrightarrow (\forall U) (x \in \Box_R U \Rightarrow C \subseteq U).
\end{align*}
This shows that $\Upsilon_{Y} * R = R_{\Box_R} * \Upsilon_{X}$, 
and hence $\Upsilon$ is a natural isomorphism. 
\endproof

We now turn to the natural isomorphism $\eta \colon 1_{\AlgFrmJ} \to \V \circ \Y$. 
Let $L$ 
be an algebraic frame. Define $\eta_L \colon L \to \V(Y_L)$ by $\eta_L(a) 
= \up a \cap Y_L$.
To prove that $\eta_L$ is well defined, 
we need the following generalization of Lemma~\ref{lem: admissible}.

\begin{lemma} \label{lem: closed admissibles}
 Let $L$ be an algebraic frame and $C$ a closed upset of $Y_L$. Then $C$ is admissible 
 iff $C = \up a \cap Y_L$ for some $a \in L$.
\end{lemma}
 
\proof
Suppose that $C = \up a \cap Y_L$ for some $a \in L$. 
If $q \in Y_L \setminus C$, then $a \not\le q$. By 
Lemma~\ref{lem: existence of primes and clopens}\ref{primes}, 
there is $p \in \p(L)$ with $a \not\le p$ and $q \le p$. Thus, $C$ is admissible.
Conversely, suppose that $C$ is admissible. By Lemmas~\ref{lem: intersection of 
admissible clopens} and \ref{lem: admissible}, 
\[
C = \bigcap \{ \up k \cap Y_L \mid k \in \Kp(L), C \subseteq \up k \cap Y_L\}. 
\]
Set $S = \{ k \in \Kp(L) \mid C \subseteq \up k \cap Y_L\}$ and $a = \bigvee S$. 
If $p \in C$, then $p \in \up k \cap Y_L$ and so $k \le p$ for each $k \in S$. 
Therefore, $a \le 
p$, which gives $p \in \up a \cap Y_L$. For the reverse inclusion, let 
$p \in \up a \cap Y_L$. Then $a \le p$, so $k \le p$ for each $k \in S$. This yields 
$p \in \bigcap \{ \up k \cap Y_L \mid k \in S \} = C$. Thus, $C = \up a \cap Y_L$.
\endproof
  
\begin{proposition} \label{prop: eta is natural}
Let $L$ be an algebraic frame. Then $\eta_L$ is an $\AlgFrmJ$-morphism and 
$\eta \colon 1_{\AlgFrmJ} \to \V \circ \Y$ is a natural isomorphism.
\end{proposition}

\proof
That $\eta_L$ is well defined follows from Lemma~\ref{lem: closed admissibles}. To see 
that $\eta_L$ is an $\AlgFrmJ$-morphism, let $S \subseteq L$ and $a = \bigvee S$. 
Then $\up a = \bigcap \{ \up s \mid s \in S\}$, so 
\[
\eta_L(a) = \up a \cap Y_L = \bigcap \{ \up s \cap Y_L \mid s \in S\} = \bigvee \{ \up s 
\cap Y_L \mid s \in S\} = \bigvee \{ \eta_L(s) \mid s \in S \}. 
\]
Therefore, $\eta_L$ preserves arbitrary joins. To see it preserves compact elements, let 
$k \in \Kp(L)$. Then $\eta_L(k) = \up k \cap Y_L$, which is compact by Lemmas~\ref{lem: 
LX is a DAlgLat} and~\ref{lem: admissible}. Therefore, $\eta_L$ is an 
$\AlgFrmJ$-morphism. It is clearly 1-1, and is onto by 
Lemma~\ref{lem: closed admissibles}. Thus, $\eta_L$ is an 
isomorphism.

To show naturality, let $\alpha \colon L_1 \to L_2$ be an $\AlgFrmJ$-morphism with 
right adjoint $r$. Then $R_\alpha \subseteq Y_{L_2} \times Y_{L_1}$ is given by 
$p \rel{R_\alpha} q$ if $r(p) \le q$.
\[
\begin{tikzcd}
L_1 \arrow[r, "\alpha"] \arrow[d, "\eta_{L_1}"'] & L_2 \arrow[d, "\eta_{L_2}"] \\
\V(Y_{L_1}) \arrow[r, "\Box_{R_\alpha}"'] & \V(Y_{L_2})
\end{tikzcd}
\]
Let $a \in L_1$. Then $\eta_{L_2}(\alpha(a)) = \up \alpha(a) \cap Y_{L_2}$. Also, 
\begin{align*}
\Box_{R_\alpha} \eta_{L_1}(a) &= \Box_{R_\alpha}(\up a \cap Y_{L_1}) = 
\{ x \in Y_{L_2} \mid R_\alpha[x] \subseteq \up a \cap Y_{L_1} \} \\
&= \{ x \in Y_{L_2} \mid (\forall q \in Y_{L_1})(r(x) \le q \Rightarrow a \le q)\}.
\end{align*}
By Lemma~\ref{lem: existence of primes and clopens}\ref{primes}, 
$r(x) = \bigwedge (\up r(x) \cap \p(L))$. Therefore, 
\[
\Box_{R_\alpha} \eta_{L_1}(a) = \{ x \in Y_{L_2} \mid a \le r(x) \} = \{ x \in Y_{L_2} 
\mid \alpha(a) \le x \} = \up \alpha(a) \cap Y_{L_2}.
\]
Thus, $\Box_{R_\alpha} \eta_{L_1}(a) = \up \alpha(a) \cap Y_{L_2}$. This proves 
naturality, and hence $\eta$ is a natural isomorphism.
\endproof

Propositions~\ref{prop: functor from GPries to DAlgLat}, \ref{prop: functor from DAlgLat 
to GPries}, \ref{prop: R is natural}, and \ref{prop: eta is natural} yield our main result.

\begin{theorem} \label{thm: duality between DAlgLat and GPries*}
The contravariant functors $\V$ and $\Y$ establish a dual equivalence between $
\AlgFrmJ$ and $\PGPries$.
\end{theorem}

Putting Theorem~\ref{thm: duality between DAlgLat and GPries*} and Corollary~\ref{cor: 
DMSLat = DAlgLat} together yields Theorem~\ref{thm: pointed Pries}. But we can say 
more:

\begin{theorem} \label{thm: eq and dual eq}
The functors establishing the duality of Theorem~\emph{\ref{thm: pointed Pries}} are 
the compositions of the functors of Theorem~\emph{\ref{thm: duality between DAlgLat 
and GPries*}} and Corollary~\emph{\ref{cor: DMSLat = DAlgLat}}. 
\end{theorem}

\proof
Let $M \in \DMSLat$ and $L = \F(M)$. By \cite[Rem.~3.2]{BJ13} and \cite[Prop.~4.8]
{BJ11}, pseudoprime elements of $L$ are precisely the optimal filters of $M$. Therefore, 
\[
\Y\F(M) = Y_L = \PP(L) \cup \{M\} = \Opt(M) \cup \{M\} = \X(M).
\] 
If $\alpha\colon M_1 \to M_2$ is a $\DMSLat$-morphism, then as we saw in 
Section~\ref{sec: pontryagin}, $\F(\alpha)$ is the left adjoint of $\alpha^{-1}$, 
and hence $\alpha^{-1}$ is 
the right adjoint of $\F(\alpha)$. Therefore, it follows from the definitions of $
\X(\alpha)$ and $\Y \F(\alpha)$ that they coincide.
In the opposite direction, if $X \in \PGPries$, then $\K\V(X) = \A(X)$ by 
Lemma~\ref{lem: LX is a DAlgLat}. 
If $R \subseteq X \times Y$ is a generalized Priestley morphism, then $\V(R)=\Box_R$ 
and $\K \V(R)$ is the restriction of $\Box_R$ to $\K\V(Y) = \A(Y)$, which is exactly $
\A(R)$. 
This shows that $\X = \Y \circ \F$ and $\A = \K \circ \V$.
\[
\begin{tikzcd}
\DMSLat \arrow[rr, shift left = .6ex, "\F"] \arrow[dr,  shift left=.5ex, "\X"] && 
\AlgFrmJ \arrow[ll, shift left = .6ex, "\K"] \arrow[dl, shift left = .6ex, "\Y"] \\
& \PGPries \arrow[ul, shift left = .6ex, "\A"] \arrow[ur, shift left = .6ex, "\V"] &
\end{tikzcd}
\]
\endproof

We conclude this section by discussing what happens when we restrict our attention to 
bounded distributive meet-semilattices and compact algebraic frames. 
Let $X \in \PGPries$. As we pointed out in Remark~\ref{rem: m iso}, $\A(X)$ is bounded 
iff $m$ is an isolated point of $X$. By Corollary~\ref{cor: BMSLat = KDAlgLat}\ref{BMSLat = KDAlgLat}, this is 
equivalent to $\V(X)$ being compact. By Remark~\ref{rem: m iso}, the full subcategory of 
$\PGPries$ consisting of those $X\in\PGPries$ in which $m$ is isolated is equivalent to $
\GPries$. Thus, the diagram above restricts to the categories $\BDMSLat$, $\KAlgFrmJ$, and $\GPries$, 
which further restricts to $\BDMSLatB$, $\KAlgFrmJB$, and $\GPriesT$
by Corollary \ref{cor: BMSLat = KDAlgLat}\ref{BMSLatB = KDAlgLatJB} and Theorem \ref{thm: gen Pries}. 

The restrictions of the functors $\F,\K$ do not require any modification. 
The functors $\A,\X$ are modified as in 
Theorem~\ref{thm: gen Pries}. The functor $\V$ is the same, but this time defined on $
\GPries$. Because of this, $\varnothing\in\V(X)$ for each $X\in\GPries$. In fact, the 
maps $C \mapsto C \cap X$ and $D \mapsto D \cup \{m\}$ are inverse isomorphisms 
between $\V(X)$ and $\V(X^+)$, under which $\varnothing$ corresponds to $\{m\}$ 
(see Remark~\ref{rem: m iso}). 

Finally, we need to slightly modify $\Y$. Indeed, if $L$ is a compact algebraic frame, 
then $1$ is an isolated point of $Y_L$. Therefore, $\PP(L)\in\GPries$. Also, if $\alpha$ is 
an $\KAlgFrmJ$-morphism, then $R_\alpha^-$ is a $\GPries$-morphism (see 
Remark~\ref{rem: m iso}). Thus, we can modify $\Y$ by sending $L$ to $\PP(L)$ and $
\alpha$ to $R_\alpha^-$. Using the same letter $\Y$ for this modified functor, we arrive 
at the following: 

\begin{corollary} \label{cor: bounded}
The functors of Theorem~\emph{\ref{thm: eq and dual eq}} restrict to yield an equivalence and dual equivalence of: 
\begin{enumerate}[$(1)$]
\item \label{J case} $\KAlgFrmJ$, $\BDMSLat$, and $\GPries$.
\item \label{JB case} $\KAlgFrmJB$, $\BDMSLatB$, and $\GPriesT$.
\end{enumerate}
\end{corollary}
 
Putting Theorem~\ref{thm: duality between DAlgLat and GPries*} and Corollaries~\ref{cor: DMSLat = DAlgLat} 
and \ref{cor: bounded}\ref{J case} together yields the top layer of Figure~\ref{fig: first diagram}.

\section{Various morphisms between algebraic frames} \label{sec: various morphisms}

So far we worked with the maps between algebraic frames that preserve 
arbitrary suprema and compact elements. This resulted 
in the category $\AlgFrmJ$, which is equivalent to $\DMSLat$ (Corollary~\ref{cor: 
DMSLat = DAlgLat}) and dually equivalent to $\PGPries$ (Theorem~\ref{thm: duality 
between DAlgLat and GPries*}). The equivalence
and dual equivalence of their bounded versions was established in Corollary~\ref{cor: bounded}.
As we pointed out in the introduction,
there are several stronger notions of morphism between algebraic frames that are 
natural to consider. 
In this section we turn our attention to those and the corresponding morphisms between
generalized Priestley spaces, thus obtaining the bottom 
three layers of Figure~\ref{fig: first diagram}.

\subsection*{\textbf{Strong Priestley morphisms.}}

\begin{definition} \label{def: DMSLatS}
\hfill
\begin{enumerate}[$(1)$]
\item Let \label{DMSLatWS} $\DMSLatWS$ be the wide subcategory of 
$\DMSLat$ whose morphisms preserve all existing nonempty finite suprema,
and let \label{BDMSLatWS} $\BDMSLatWS$ be the full subcategory of $\DMSLatWS$ whose objects are bounded.
\item \label{DMSLatS} Let $\DMSLatS$ be the wide subcategory of $\DMSLatWS$ whose morphisms are bounded, 
and let \label{BDMSLatS} $\BDMSLatS$ be the full subcategory of $\DMSLatS$ whose objects are bounded.
\end{enumerate}
\end{definition}

\begin{remark}
In \cite{BJ08,BJ11} morphisms of $\BDMSLatS$ are called {\em sup-homomorphisms}.
\end{remark}

Let $L_1 ,L_2$ be algebraic frames and $\alpha\colon L_1\to L_2$ an 
$\AlgFrmJ$-morphism. 
Since we work with the dual orders on $\Kp(L_1)$ and $\Kp(L_2)$,
the restriction $\alpha|_{\Kp(L_1)}$ is a $\DMSLatWS$-morphism iff the following 
condition is satisfied:
\begin{equation*}\label{eqn: dagger}
\mbox{If } \varnothing \ne S \subseteq \Kp(L_1) \mbox{ is finite and } \bigwedge S \in \Kp(L_1), 
\mbox{ then } 
\alpha \left(\bigwedge S \right) = \bigwedge \alpha[S]. \tag{\sf{FInf}}
\end{equation*}

\begin{definition} \label{def: AlgFrmS}
\hfill
\begin{enumerate}[$(1)$]
\item Let \label{AlgFrmWS} $\AlgFrmWS$ be the wide subcategory of $\AlgFrmJ$ whose morphisms satisfy 
\eqref{eqn: dagger}, and \label{AlgFrmS} let $\AlgFrmS$ be the wide subcategory of $\AlgFrmWS$ whose morphisms are bounded.
\item \label{KAlgFrmWS} Let $\KAlgFrmWS$ be the full subcategory of $\AlgFrmWS$ and \label{KAlgFrmS} $\KAlgFrmS$ 
the full subcategory of $\AlgFrmS$ consisting of compact algebraic frames.
\end{enumerate}
\end{definition}

As an immediate consequence of Corollaries~\ref{cor: DMSLat = DAlgLat},~\ref{cor: DMSB and AlgFrmJ}, 
\ref{cor: BMSLat = KDAlgLat}, and the above observation, we obtain:

\begin{theorem} \label{thm: KDAlgLatS} 
\hfill
\begin{enumerate}[$(1)$]
\item \label{AlgFrmWS = DMSLatWS} $\AlgFrmWS$ is equivalent to $\DMSLatWS$.
\item \label{AlgFrmS = DMSLatS} $\AlgFrmS$ is equivalent to $\DMSLatS$.
\item \label{KAlgFrmWS = BDMSLatWS}$\KAlgFrmWS$ is equivalent to $\BDMSLatWS$.
\item \label{KAlgFrmS = BDMSLatS} $\KAlgFrmS$ is equivalent to $\BDMSLatS$.
\end{enumerate}
\end{theorem}

We next describe the corresponding categories of generalized Priestley spaces utilizing the notion of strong Priestley morphisms from \cite{BJ08,BJ11}.

\begin{definition} \label{def: functional and strong morphisms}
Let $X,Y$ be generalized Priestley spaces.
\begin{enumerate}[$(1)$]
\item \label{functional} A generalized Priestley morphism 
$R \subseteq X \times Y$ is {\em functional} if 
$R[x]$ has a least element for each $x\in X$. 
Let \label{GPriesF} $\GPriesF$ be the wide subcategory of $\GPries$ consisting of functional generalized
Priestley morphisms.
\item \label{strong} An order-preserving map $f \colon X \to Y$ is a 
{\em strong Priestley morphism} 
if $U \in \A(Y)$ implies $f^{-1}(U) \in \A(X)$. 
Let \label{GPriesS} $\GPriesS$ be the category of generalized Priestley spaces and strong Priestley
morphisms.
\end{enumerate}
\end{definition}
 
 The categories $\GPriesF$ and $\GPriesS$ consist of the same objects. 
If $R \subseteq X \times Y$ is a functional generalized Priestley morphism, then sending 
$x$ to the least element of $R[x]$ defines a strong Priestley morphism 
$f_R \colon X \to Y$. 
Conversely, if $f \colon X \to Y$ is a strong Priestley morphism, then 
$R_f \subseteq X \times Y$ defined by $x \rel{R_f} y$ iff $f(x) \le y$ is a 
functional generalized Priestley morphism. 
Moreover, 
\[
f_{S*R}=f_S \circ f_R, \ R_{g\circ f}=R_g * R_f, \ R=R_{f_R}, \mbox{ and } f=f_{R_f}. 
\]
We thus obtain: 

\begin{theorem} \label{thm: GPSF = GPSS}
\emph{\cite[Prop.~9.5]{BJ08}} The categories $\GPriesF$ and $\GPriesS$ are isomorphic. 
\end{theorem}

In \cite[Thm.~9.6]{BJ08} it was shown that $\GPriesS$ is dually equivalent to $\BDMSLatS$. 
This together with Theorem~\ref{thm: KDAlgLatS}\ref{KAlgFrmS = BDMSLatS} yields:

\begin{corollary} \label{cor: KAlgFrmS = BDMSLatS}
The category $\KAlgFrmS$ is equivalent to $\BDMSLatS$ and dually equi\-valent to $\GPriesS$.
\end{corollary}

The notions of functional generalized Priestley morphism and strong Priestley morphism directly generalize 
to the pointed case. Let $(X, m)$ and $(Y, n)$ be pointed generalized Priestley spaces
with $m, n$ isolated. If $R \subseteq X \times Y$ is a functional
Priestley morphism, then $R^-$ may not be functional since $R^-[x]$ could be empty for some $x \in X^-$.
Similarly, if $f \colon X \to Y $ is a strong Priestley morphism, then the restriction $f^- \colon X^- \to Y^-$ 
may only be a partial function. For $R^-$ to be functional, and hence for $f^-$ to be a total function, 
an additional condition is required. We 
thus arrive at the following wide subcategories of $\PGPries$ consisting of two kinds of functional 
generalized Priestley morphisms, and their corresponding categories of
pointed generalized Priestley spaces and strong Priestley morphisms.

\begin{definition}\ \label{def: functional and strong}
\begin{enumerate}[$(1)$]
\item Let $\PGPriesF$ be the wide subcategory of $\PGPries$ consisting of
functional morphisms, and let $\PGPriesSF$ be the wide subcategory of $\PGPriesF$ whose morphisms 
$R \subseteq X \times Y$ additionally satisfy
$x \ne m$ implies 
$R[x] \ne \{n\}$ provided $m, n$ are isolated.
\item 
Let \label{PGPriesS} $\PGPriesWS$ be the category of pointed generalized Priestley spaces and
strong Priestley morphisms, and let $\PGPriesS$ be the wide subcategory of $\PGPriesWS$ whose morphisms
$f \colon X \to Y$ additionally satisfy $f[X^-] \subseteq Y^-$ provided $m, n$ are isolated.
\end{enumerate}
\end{definition}

\begin{remark}
The subscript {\sf F} abbreviates functional, {\sf S} abbreviates strong, and {\sf T} abbreviates total 
because a $\PGPriesWS$-morphism $f$ is a $\PGPriesS$-morphism iff $f^- \colon X^- \to Y^-$ is a total function provided $m, n$ are isolated. 
\end{remark}

We point out that identity morphisms in $\PGPriesWS$ and $\PGPriesS$ are identity functions and 
composition is usual function composition. 
As a direct generalization of \cite[Prop.~9.5]{BJ08}, we obtain:
 
\begin{proposition} \label{prop: F and S}
\hfill
\begin{enumerate}[$(1)$]
\item \label{F and S 1} $\PGPriesF$ is isomorphic to $\PGPriesWS$.
\item $\PGPriesSF$ is isomorphic to $\PGPriesS$.
\end{enumerate}
\end{proposition}

To describe the corresponding categories of distributive meet-semilattices and algebraic frames, 
we require the following two lemmas.

\begin{lemma} \label{lem: sup = strong WS}
Let $\alpha \colon L_1 \to L_2$ be an $\AlgFrmJ$-morphism and $r \colon L_2 \to L_1$ 
its right adjoint. The following are equivalent.
\begin{enumerate}[$(1)$]
\item \label{sup = strong 1} $\alpha$ is an $\AlgFrmWS$-morphism.
\item \label{sup = strong 2} $r[Y_{L_2}] \subseteq Y_{L_1}$.
\item \label{sup = strong 3} $r \colon Y_{L_2} \to Y_{L_1}$ is a strong 
Priestley morphism whose corresponding 
functional generalized Priestley morphism is $R_\alpha$.
\end{enumerate}
\end{lemma}

\proof
\ref{sup = strong 1}$\Rightarrow$\ref{sup = strong 2}. Let $p \in Y_{L_2}$. 
Since $r$ preserves arbitrary meets, $r(1)=1$, 
so we may assume that $p \in \PP(L_2)$. If $r(p)=1$, there is nothing to prove. Suppose 
that $r(p) \ne 1$. We prove that $r(p) \in \PP(L_1)$.
Let $a_1,\dots,a_n \in L_1$ with $a_1 \wedge \cdots \wedge a_n \ll r(p)$. 
Since $L_1$ is algebraic, there is 
$k \in \Kp(L_1)$ with $a_1 \wedge \cdots \wedge a_n \le k \le r(p)$. We show that $a_i 
\le r(p)$ for some $i$. If not, 
then there are compact $l_i$ with $l_i \le a_i$ and $l_i \not\le r(p)$.
Since $l_1 \wedge \cdots \wedge l_n \le k$ and $L_1$ is distributive, $(l_1 \vee k) 
\wedge \cdots \wedge (l_n \vee k) = k$. Thus, the meet of $l_1 \vee k, \dots, l_n \vee 
k$ exists in $\Kp(L_1)$, so $\alpha(k) = \alpha(l_1 \vee k) \wedge \cdots \wedge 
\alpha(l_n \vee k)$ by (1). Since $\alpha$ preserves joins and $L_2$ is distributive,
\[
\alpha(k) = \bigwedge_{i=1}^n (\alpha(l_i) \vee \alpha(k)) = \left(\bigwedge_{i=1}^n 
\alpha(l_i)\right) \vee \alpha(k).
\]
Therefore, $\alpha(l_1) \wedge \cdots \wedge \alpha(l_n) \le \alpha(k)$. Since $k \le 
r(p)$ implies $\alpha(k) \le p$ and $\alpha(k) \in \Kp(L_2)$, we have $\alpha(l_1) \wedge 
\cdots \wedge \alpha(l_n) \ll p$. Because $p \in \PP(L_2)$, it follows that $\alpha(l_i) 
\le p$ for some $i$. Consequently, $l_i \le r(p)$ for some $i$. The obtained contradiction 
proves that $r(p) \in \PP(L_1)$.

\ref{sup = strong 2}$\Rightarrow$\ref{sup = strong 3}. Let $p\in Y_{L_2}$. 
By \ref{sup = strong 2}, $r(p) \in Y_{L_1}$. Since $q \in R_\alpha[p]$ iff $r(p) \le q$, 
we see that $r(p)$ 
is the least element of $R_\alpha[p]$. Thus, $R_\alpha$ is functional, and $r$ is its 
corresponding strong Priestley morphism.

\ref{sup = strong 3}$\Rightarrow$\ref{sup = strong 1}. Let 
$S \subseteq \Kp(L_1)$ be nonempty finite and $\bigwedge S \in \Kp(L_1)$. 
Since $\alpha$ is order preserving, $\alpha(\bigwedge S) \le \bigwedge \alpha[S]$. 
Suppose that $\bigwedge \alpha[S] \not\le \alpha(\bigwedge S)$. By Lemma~\ref{lem: 
existence of primes and clopens}\ref{primes}, there is $p \in \p(L_2)$ with 
$\bigwedge \alpha[S] \not\le p$ and 
$\alpha(\bigwedge S) \le p$. This yields $\bigwedge S \le r(p)$. Since $\bigwedge S \in 
\Kp(L_1)$, we see that $\bigwedge S \ll r(p)$, so $s \le r(p)$ for some $s \in S$ since 
$r(p) = 1$ or $r(p) \in \PP(L_1)$. Therefore, $\alpha(s) \le p$, and so $\bigwedge 
\alpha[S] \le p$. The obtained contradiction proves that $\bigwedge \alpha[S]\le 
\alpha(\bigwedge S)$, hence the equality. Thus, $\alpha$ is an 
$\AlgFrmWS$-morphism.
\endproof

\begin{lemma} \label{lem: sup = strong S}
Let $L_1$ and $L_2$ be compact algebraic frames, $\alpha \colon L_1 \to L_2$ an $\AlgFrmJ$-morphism, 
and $r \colon L_2 \to L_1$ its right adjoint. The following are equivalent.

\begin{enumerate}[$(1)$]
\item \label{sup = strong 1S} $\alpha$ is an $\AlgFrmS$-morphism.
\item \label{sup = strong 2S} $r[Y_{L_2}] \subseteq Y_{L_1}$ and
$r(p) \ne 1$ for each $p \in \PP(L_2)$.
\item \label{sup = strong 3S} $r \colon Y_{L_2} \to Y_{L_1}$ is a strong 
Priestley morphism whose corresponding 
functional generalized Priestley morphism is $R_\alpha$ and
$R_\alpha[p] \ne \{1\}$ for each $p \in \PP(L_2)$.
\end{enumerate}
\end{lemma}

\proof
\ref{sup = strong 1S}$\Rightarrow$\ref{sup = strong 2S}. By Lemma~\ref{lem: sup = strong WS}, 
$r[Y_{L_2}] \subseteq Y_{L_1}$.  Since $L_1, L_2$ are compact and
$\alpha$ is a $\AlgFrmS$-morphism,  we have $\alpha(1) = 1$. 
If $p \in \PP(L_2)$ with $r(p) = 1$, then  $\alpha(1) \le p$, so $p = 1$, which is false. 
Therefore, $r(p) \ne 1$.

\ref{sup = strong 2S}$\Rightarrow$\ref{sup = strong 3S}. By Lemma~\ref{lem: sup = strong WS}, 
$r$ is a strong Priestley morphism whose corresponding functional generalized Priestley morphism is 
$R_\alpha$.  If $p \in \PP(L_2)$, then $r(p) \ne 1$ by \ref{sup = strong 2}. 
Therefore, by Lemma~\ref{lem: existence of primes and clopens}\ref{primes}, there is 
$q \in \p(L_1)$ with $r(p) \le q$. Consequently, $q \in R_\alpha[p]$, and so $R_\alpha[p] \ne \{1\}$. 

\ref{sup = strong 3S}$\Rightarrow$\ref{sup = strong 1S}. By Lemma~\ref{lem: sup = strong WS}, 
$\alpha$ is an $\AlgFrmWS$-morphism. It then suffices to show that $\alpha(1) = 1$. If 
$\alpha(1) \ne 1$, then there is $p \in \p(L_2)$ with $\alpha(1) \le p$. 
This yields $1 \le r(p)$, so $R_\alpha[p] = \{1\}$, contradicting 
\ref{sup = strong 3}. Therefore, $\alpha(1) = 1$, and thus $\alpha$ is an $\AlgFrmS$-morphism.
\endproof

Lemmas~\ref{lem: sup = strong WS} and \ref{lem: sup = strong S} give the following: 

\begin{theorem} \label{thm: DAlgLatS}
\hfill
\begin{enumerate}[$(1)$]
\item \label{DAlgLatS 1} The duality of Theorem~\emph{\ref{thm: duality between DAlgLat and GPries*}} 
between $\PGPries$ and $\AlgFrmJ$ restricts to a duality between $\PGPriesF$ and $\AlgFrmWS$ 
and yields a duality between $\PGPriesWS$ and $\AlgFrmWS$. 
\item \label{DAlgLatS 2} The duality also restricts to a duality between $\PGPriesSF$ and $\AlgFrmS$ 
and yields a duality between $\PGPriesS$ and $\AlgFrmS$. 
\end{enumerate}
\end{theorem}

\proof
\ref{DAlgLatS 1}.
Let $X$ be a pointed generalized Priestley space. Then $\Upsilon_X$ is a functional 
morphism since the least element of $\Upsilon_X[x]$ is $\varepsilon_X(x)$ by 
the definition of $\Upsilon_X$ and 
Lemma~\ref{lem: varepsilon is an iso}\ref{varepsilon 1}. 
Therefore, $\Upsilon_X$ is a $\PGPriesF$-isomorphism. In addition, 
if $L$ is an algebraic frame, then it follows from the proof of 
Proposition~\ref{prop: eta is natural} that $\eta_L$ is a poset isomorphism. 
Therefore, it is an $\AlgFrmS$-isomorphism. From these observations, 
Proposition~\ref{prop: F and S}\ref{F and S 1}, and Lemma~\ref{lem: sup = strong WS}
it follows that the duality of Theorem~\ref{thm: duality between DAlgLat and GPries*} restricts to 
a duality between $\PGPriesF$ and $\AlgFrmWS$.

\ref{DAlgLatS 2}. The proof is similar to that of \ref{DAlgLatS 1} except that 
Lemma~\ref{lem: sup = strong S} is used instead of Lemma~\ref{lem: sup = strong WS}.
\endproof

As a consequence of Theorems~\ref{thm: KDAlgLatS}
and \ref{thm: DAlgLatS}, 
we obtain: 

\begin{corollary} \label{PGPries = DMSLatS}
\hfill
\begin{enumerate}[$(1)$]
\item $\AlgFrmWS$ is equivalent to $\DMSLatWS$ and dually equivalent to $\PGPriesWS$.
\item $\AlgFrmS$ is equivalent to $\DMSLatS$ and dually equivalent to $\PGPriesS$.
\end{enumerate}
\end{corollary}

\begin{remark} \label{rem: V and Y on S categories}
If $\alpha\colon L_1 \to L_2$ is an $\AlgFrmWS$-morphism, then 
$\Y(\alpha)=R_\alpha$. Applying Lemma~\ref{lem: sup = strong WS}, 
the corresponding strong Priestley morphism is the right 
adjoint $r$ to $\alpha$ restricted to $Y_{L_2}$. Therefore, we may view that 
$\Y\colon \AlgFrmWS \to \PGPriesWS$ acts on morphisms by sending $\alpha$ to 
$r\colon Y_{L_2} \to Y_{L_1}$. 

If $R \subseteq X \times Y$ is a 
$\PGPriesWS$-morphism, then $\V(R)=\Box_R$. Let $f$ be the strong 
Priestley morphism corresponding
to $R$. Then  $\Box_R 
C = f^{-1}(C)$ for each $C \in \V(Y)$ (see \cite[Lem.~9.2]{BJ08}). 
Thus, we may view that $\V\colon  \PGPriesWS \to \AlgFrmWS$ 
acts on morphisms by sending a strong Priestley morphism $f$ to $f^{-1}$.
Similar observations apply to $\Y\colon \AlgFrmS \to \PGPriesS$ and $\V\colon  \PGPriesS \to \AlgFrmS$.
\end{remark}

We next restrict the dualities of Theorem~\ref{thm: DAlgLatS} to 
compact algebraic frames.
As is customary, by a \emph{partial function} from 
$X$ to $Y$ we mean a function from a subset of $X$ to $Y$. We denote such partial function by 
$f \colon X \dashrightarrow Y$.

\begin{definition} \label{def: partial strong Priestley morphism}
\hfill
\begin{enumerate}[$(1)$]
\item Let $X, Y \in \GPries$. A \emph{partial strong Priestley morphism} between $X$ and $Y$ is a 
partial function $f \colon X \dashrightarrow Y$ whose domain is a clopen downset of $X$ such 
that $U \in \A(Y)$ implies $X \setminus f^{-1}(Y \setminus U) \in \A(X)$.
\item Let \label{GPriesPS} $\GPriesPS$ be the category of generalized Priestley spaces and partial 
strong Priestley morphisms.
\end{enumerate}
\end{definition}

\begin{proposition} \label{prop: GPriesPS = subcategory of PGPriesWS}
$\GPriesPS$ is equivalent to the full subcategory of $\PGPriesWS$ consisting 
of generalized Priestley spaces with isolated maxima.
\end{proposition}

\proof
By Remark~\ref{rem: m iso}, $\GPries$ is equivalent to the full subcategory of $\PGPries$ whose
objects have isolated maxima. Let $(X, m)$ and $(Y, n)$ be generalized Priestley spaces with
$m, n$ isolated. As in the remark, set $X^- = X \setminus \{m\}$ and 
$Y^- = Y \setminus \{n\}$. By Proposition~\ref{prop: F and S}\ref{F and S 1}, functional morphisms between $X$ and $Y$ correspond to strong
Priestley morphisms. Therefore, it is enough to show that the latter correspond to partial
strong Priestley morphisms between $X^-$ and $Y^-$.

Let $f \colon X \to Y$ be a strong Priestley morphism. 
Set $C = f^{-1}(Y^-)$. Then $f \colon C \to Y^-$ is a well-defined 
function, and we show that $f \colon X^- \dashrightarrow Y^-$ is a partial strong Priestley morphism. First, $C$ 
is a clopen downset of $X$ since $Y^-$ is a clopen downset of $Y$ and $f$ is a continuous order-preserving function. 
It is then a clopen downset of 
$X^-$ since $m \notin C$. Let $U \in \A(Y^-)$. Then 
 $V := U \cup \{n\} \in \A(Y)$ and we have
\begin{align*}
X^- \setminus f^{-1}(Y^- \setminus U) &= \{ x \in X^- \mid x \notin f^{-1}(Y^- \setminus U) \} \\
&= \{ x \in X^- \mid x \notin C \textrm{ or } x \in C \ \&\  f(x) \in U \} \\
&= \{x \in X^- \mid f(x) = n \textrm{ or } f(x) \in U \} \\
&= \{x \in X^- \mid f(x) \in V \} \\
&= X^- \cap f^{-1}(V) \in \A(X^-)
\end{align*}
because $f^{-1}(V) \in \A(X)$. Therefore, $f \colon X^- \dashrightarrow Y^-$ is a partial strong 
Priestley morphism.

Conversely, let $f \colon X^- \dashrightarrow Y^-$ be a partial strong Priestley morphism, and let 
$C$ be the domain of $f$. Extend $f$ to a function $g \colon X \to Y$ by setting $g(x) = n$ for each 
$x \in X \setminus C$. Since $f$ is order preserving and $n$ is the top of $Y$, we see that
$g$ is order preserving. To show that $g$ is a strong Priestley morphism, let $V \in \A(Y)$. Set 
$U = V \setminus \{n\} \in \A(Y^-)$. We have 
$X^- \setminus f^{-1}(Y^- \setminus U) = (X^- \setminus C) \cup f^{-1}(U)$. Therefore,
\begin{align*}
g^{-1}(V) &= \{ x \in X \mid g(x) \in V\} = g^{-1}(n) \cup g^{-1}(U) \\
&= \{m\} \cup (X^- \setminus C) \cup f^{-1}(U) \\
&= \{m\} \cup (X^- \setminus f^{-1}(Y^- \setminus U)) \in \A(X)
\end{align*}
since $X^- \setminus f^{-1}(Y^- \setminus U) \in \A(X^-)$.
Thus, $g$ is a strong Priestley morphism. 
Consequently, $\GPriesPS$ is equivalent to the full subcategory of $\PGPriesWS$ consisting of generalized Priestley spaces
with isolated maxima.
\endproof

As an immediate consequence of Proposition~\ref{prop: GPriesPS = subcategory of PGPriesWS} we obtain:

\begin{theorem} \label{thm: GPriesS}
The duality of Theorem~\emph{\ref{thm: DAlgLatS}} 
between $\PGPriesWS$ and $\AlgFrmWS$ restricts to a duality between $\GPriesPS$ and $\KAlgFrmWS$.
\end{theorem}

This together with Theorem~\ref{thm: KDAlgLatS}\ref{KAlgFrmWS = BDMSLatWS} gives:

\begin{corollary} \label{GPriesS = BDMSLatS}
The category $\KAlgFrmWS$ is equivalent to $\BDMSLatWS$ and dually equivalent to $\GPriesPS$.
\end{corollary}

Putting Corollaries~\ref{cor: KAlgFrmS = BDMSLatS},  \ref{PGPries = DMSLatS}, and 
\ref{GPriesS = BDMSLatS} together yields the middle two layers of Figure~\ref{fig: first diagram}.

\subsection*{\textbf{Strong Priestley morphisms preserving prime elements.}}

To obtain the bottom layer of Figure~\ref{fig: first diagram}, we turn to 
the most natural category of algebraic frames, in which morphisms are frame 
homomorphisms preserving compact elements (that is, $\AlgFrmJ$-morphisms preserving 
finite infima). 

\begin{definition} \label{def: AlgFrm}
Let $\AlgFrm$ be the wide subcategory of $\AlgFrmJ$ whose morphisms 
preserve finite infima. 
\end{definition}

\begin{remark}
Clearly $\AlgFrm$ is also a wide subcategory of $\AlgFrmS$.
\end{remark}

We next describe the wide subcategory of $\PGPriesS$ that is dually equivalent to 
$\AlgFrm$. 

\begin{definition} \label{def: PGPriesP}
Let $\PGPriesP$ denote the wide subcategory of $\PGPriesWS$ whose morphisms 
$f \colon X \to Y$ satisfy $f[X_0] \subseteq Y_0$, and define $\GPriesP$ similarly.
\end{definition}

 \begin{remark}
Let $f \colon X \to Y$ be a $\PGPriesP$-morphism and suppose that $m, n$ are isolated. 
Then $\{m\} \in \A(X)$, so $\max(X \setminus \{m\}) \subseteq X_0$, which implies that 
$\max X^- \subseteq X_0$. Let $x \in X^-$. By Definition~\ref {def: PGPries}\ref{order-dense}, 
there is $z \in X_0$ with $x \le z$. Therefore, $f(x) \le f(z)$, and $f(z) \in Y_0$ by hypothesis. 
Consequently, $f(x) \ne n$. Thus, $f[X^-] \subseteq Y^-$, and hence $\PGPriesP$ is a wide 
subcategory of $\PGPriesS$.
\end{remark}

\begin{lemma} \label{lem: frame hom}
Let $\alpha \colon L_1 \to L_2$ be an $\AlgFrmJ$-morphism and 
$r \colon L_2 \to L_1$ its right adjoint. The following are equivalent.
\begin{enumerate}[$(1)$]
\item \label{frame hom 1} $\alpha$ is a frame homomorphism.
\item \label{frame hom 2} $r\colon Y_{L_2} \to Y_{L_1}$ is a strong Priestley 
morphism and $r[\p(L_2)] \subseteq \p(L_1)$.
\item \label{frame hom 3} Let $S$ be a finite subset of $\Kp(L_1)$ and 
$k \in \Kp(L_2)$ with $k \le \bigwedge 
\alpha[S]$. Then there is $c \in \Kp(L_1)$ with $c \le \bigwedge S$ and $k \le \alpha(c)$.
\end{enumerate}
\end{lemma}

\proof
\ref{frame hom 1}$\Rightarrow$\ref{frame hom 3}. Let $S$ be a finite subset 
of $\Kp(L_1)$ and $k \in \Kp(L_2)$ 
with $k \le \bigwedge \alpha[S]$. 
By~\ref{frame hom 1}, $k \le \alpha(\bigwedge S)$. Since $k$ is compact, 
$\alpha(\bigwedge S) = \bigvee \{ \alpha(c) \mid c \in \Kp(L_1), \, c \le \bigwedge S\}$, 
and the join is directed, 
there is $c \in \Kp(L_1)$ with $c \le \bigwedge S$ and $k \le \alpha(c)$.

\ref{frame hom 3}$\Rightarrow$\ref{frame hom 2}. We first show that 
if $S \subseteq \Kp(L_1)$ is finite with $\bigwedge S \in \Kp(L_1)$, then 
$\alpha(\bigwedge S) = \bigwedge \alpha[S]$. The inequality 
$\alpha(\bigwedge S) \le \bigwedge \alpha[S]$ 
holds since $\alpha$ is order preserving. Let $k \in \Kp(L_2)$ with $k \le \bigwedge 
\alpha[S]$. By \ref{frame hom 3}, there is $c \in \Kp(L_1)$ with $c \le \bigwedge S$ 
and $k \le \alpha(c)$. Therefore, $k \le \alpha(c) \le \alpha(\bigwedge S)$. 
Since $\bigwedge \alpha[S]$ is 
the join of the compact elements below it, we see that $\bigwedge \alpha[S] \le 
\alpha(\bigwedge S)$, hence the equality. 
This by Lemma~\ref{lem: sup = strong S} 
implies that $r$ is a strong Priestley morphism. 

We next show that if $p \in \p(L_2)$, then $r(p) \in \p(L_1)$. By the previous paragraph, 
$\alpha(1) = \alpha(\bigwedge \varnothing) = \bigwedge \alpha[\varnothing] = 1$. 
Therefore, $p \ne 1$ implies that $r(p) \ne 1$.
Let $a, b \in L_1$ with 
$a \wedge b \le r(p)$. Then $\alpha(a \wedge b) \le p$. Suppose that 
$\alpha(a) \wedge \alpha(b) \not\le p$. Then there is $k \in \Kp(L_2)$ with 
$k \le \alpha(a) \wedge \alpha(b)
$ and $k \not\le p$. By \ref{frame hom 3}, there is $c \in \Kp(L_1)$ with 
$c \le a \wedge b$ and $k \le 
\alpha(c)$. Therefore, $\alpha(c) \not\le p$, so $c \not\le r(p)$. This 
contradicts $c \le a \wedge b \le r(p)$. Thus, $\alpha(a) \wedge \alpha(b) \le p$, 
so $\alpha(a) \le p$ or $\alpha(b) \le p$ because $p \in \p(L_2)$. 
Consequently, $a \le r(p)$ or $b \le r(p)$, and 
hence $r(p) \in \p(L_1)$.

\ref{frame hom 2}$\Rightarrow$\ref{frame hom 1}. It is sufficient to show that 
$\alpha$ preserves binary meets and $
\alpha(1) = 1$. Let $a, b \in L_1$. Then $\alpha(a \wedge b) \le \alpha(a) \wedge 
\alpha(b)$ since $\alpha$ is order preserving. Suppose that 
$\alpha(a) \wedge \alpha(b) \not\le \alpha(a \wedge b)$. 
By Lemma~\ref{lem: existence of primes and clopens}\ref{primes}, there is $p \in 
\p(L_2)$ with $\alpha(a) \wedge \alpha(b) \not\le p$ and $\alpha(a \wedge b) \le p$. 
This implies that $a \wedge b \le r(p)$. By \ref{frame hom 2}, $a \le r(p)$ or $b \le r(p)$. 
Thus, $\alpha(a) \le p$ or $\alpha(b) \le p$, and hence $\alpha(a) \wedge \alpha(b) \le p$. 
The obtained contradiction shows that $\alpha(a) \wedge \alpha(b) \le \alpha(a \wedge b)$, 
hence the equality. If $\alpha(1) \ne 1$, then there is $p \in \p(L_2)$ with $\alpha(1) \le p$. 
By \ref{frame hom 2}, 
$1 \le r(p) \in \p(L_1)$, a contradiction. Thus, $\alpha(1) = 1$.
\endproof

\begin{theorem} \label{thm: AlgFrm and PGPSP}
The duality of Theorem~\emph{\ref{thm: DAlgLatS}\ref{DAlgLatS 2}} between $\PGPriesS$ and  
$\AlgFrmS$ restricts to a duality between $\PGPriesP$ and $\AlgFrm$.
\end{theorem}

\proof
Let $X$ be a pointed generalized Priestley space. By Lemma~\ref{lem: varepsilon is an 
iso}, $\varepsilon_X$ is a 
$\PGPriesP$-isomorphism. 
In addition, if $L \in \AlgFrm$, then $\eta_L$ is a poset isomorphism, so a frame 
isomorphism.
Thus, Lemma~\ref{lem: frame hom} yields that the duality of Theorem~\ref{thm: DAlgLatS}\ref{DAlgLatS 2} 
between $\PGPriesS$ and  $\AlgFrmS$
restricts to a duality between $\PGPriesP$ and $\AlgFrm$.
\endproof

Lemma~\ref{lem: frame hom}\ref{frame hom 3} suggests the following definition. 
To simplify notation, we 
denote the set of upper bounds of a subset $S$ of a poset by $S^u$.

\begin{definition} \label{def: DMSLatP}
We denote by $\DMSLatP$ the wide subcategory of $\DMSLat$ whose morphisms $
\alpha\colon M_1\to M_2$ satisfy the following condition: 
\begin{equation} \label{eqn: *}
\textrm{If }S \subseteq M_1 \textrm{ is finite and }x \in \alpha[S]^u \textrm{, then } 
\exists c \in S^u : \alpha(c) \le x. \tag{\sf{P}}
\end{equation}
\end{definition}

\begin{remark}
The subscript ${\sf P}$ in the above definition is motivated by Lemma~\ref{lem: hansoul}\ref{hansoul 2}
where we show that $\alpha$ is a $\DMSLatP$-morphism iff $\alpha$ pulls 
prime filters back to prime filters.
\end{remark}

\begin{lemma} \label{lem: P implies S}
$\DMSLatP$ is a wide subcategory of $\DMSLatS$.
\end{lemma}

\proof
Let $\alpha \colon M_1 \to M_2$ be a $\DMSLatP$-morphism. 
Let $\varnothing \ne S$ be a finite subset of $M_1$ such that $\bigvee S$ exists in $M_1$. Since $
\alpha$ is order preserving, $\bigvee\alpha[S] \le \alpha(\bigvee S)$. Let $x \in 
\alpha[S]^u$. Since $\alpha$ is a $\DMSLatP$-morphism, there is $c \in S^u$ 
such that $\alpha(c) \le x$. Therefore, $\bigvee S \le c$, so $\alpha(\bigvee S) \le 
\alpha(c)$, and hence $\alpha(\bigvee S) \le x$. Thus, $\alpha(\bigvee S) = 
\bigvee\alpha[S]$, and so $\alpha$ is a $\DMSLatWS$-morphism. To see that it is a $\DMSLatS$-morphism,
suppose that $M_1, M_2$ are bounded. 
Setting $S = \varnothing$, (\ref{eqn: *}) implies that
for $x = 0$ there is $c \in M_2$ with $\alpha(c) \le 0$. This forces $\alpha(c) = 0$, and therefore,
$\alpha(0) = 0$. Consequently, $\DMSLatP$ is a wide subcategory of $\DMSLatS$.
\endproof

Putting Theorems~\ref{thm: KDAlgLatS}\ref{AlgFrmWS = DMSLatWS}, \ref{thm: AlgFrm and PGPSP} and 
Lemmas~\ref{lem: frame hom}, \ref{lem: P implies S} together yields:

\begin{theorem} \label{thm: DMSP}
$\AlgFrm$ is equivalent to $\DMSLatP$ and dually equivalent to $\PGPriesP
$.
\end{theorem}

\begin{definition} \label{def: KAlgFrm and BDMSLatP}
\hfill
\begin{enumerate}[$(1)$]
\item  Let $\KAlgFrm$ be the full subcategory of $\AlgFrm$ consisting of 
compact algebraic frames.\label{KAlgFrm}
\item \label{BDMSLatP} Let $\BDMSLatP$ be the full subcategory of $\DMSLatP$ 
consisting of bounded distributive meet-semilattices.
\end{enumerate}
\end{definition}

As an immediate consequence of Theorem~\ref{thm: DMSP} we obtain:

\begin{theorem} \label{thm: KAlgFrm}
$\KAlgFrm$ is equivalent to $\BDMSLatP$ and dually equivalent to $\GPriesP$.
\end{theorem}

Putting Theorems~\ref{thm: DMSP} and \ref{thm: KAlgFrm} together yields the bottom layer of 
Figure~\ref{fig: first diagram}. 
We point out that $\BDMSLatP$ is a proper subcategory of $\BDMSLatS$ (see Example~\ref{ex: S not P}).
This contrasts with the bounded distributive lattice case, where the full subcategories of $\BDMSLatS$ 
and $\BDMSLatP$ consisting of lattices are equal (see Remark~\ref{rem: morphisms of DLatM}).

\begin{example} \label{ex: S not P}
Let $M$ be the distributive meet-semilattice shown below.

\begin{center}
\begin{tikzpicture}
\draw [fill] (0, 0) circle[radius=.04];
\draw [fill] (-.5, .5) circle[radius=.04];
\draw [fill] (.5, .5) circle[radius=.04];
\node [below] at (0, -.05) {$0$};
\node [left] at (-.5,.5) {$a$};
\node [right] at (.5, .5) {$b$};
\node [right] at (0, 3) {$1$};
\node at (0, 1) {$\vdots$};
\draw [fill] (0, 1.5) circle[radius=.04];
\draw [fill] (0, 2.0) circle[radius=.04];
\draw [fill] (0, 2.5) circle[radius=.04];
\draw [fill] (0, 3) circle[radius=.04];
\draw (0, 1.5) -- (0,3);
\draw (0, 0) -- (.5, .5);
\draw (0, 0) -- (-.5, .5);
\end{tikzpicture}
\end{center}
Set $F = M \setminus \{ a, b, 0\}$ and define $\alpha \colon M \to M$ by 
\[
\alpha(x) =
\begin{cases}
1 & \text{if } x \in F, \\
x & \text{if } x \in \{a, b, 0\}.
\end{cases}
\]
Then $\alpha$ is a 
$\DMSLatS$-morphism. We show that $\alpha$ is not a $\DMSLatP$-morphism. 
For, let $x \in F$ with $x < 1$. If $S = \{a, b\}$, then $x \in \alpha[S]^u$, but 
since $S^u = F$, there is no $c \in S^u$ with $\alpha(c) \le x$. Therefore, 
Condition (\ref{eqn: *}) does not hold, and hence $\alpha$ is not a 
$\DMSLatP$-morphism. 
This shows that $\BDMSLatP$ is a proper wide subcategory of $\BDMSLatS$. 
Corollary~\ref{GPriesS = BDMSLatS} and Theorem~\ref{thm: KAlgFrm} then show 
that $\KAlgFrm$ is a proper wide subcategory of $\KAlgFrmS$ and $\GPriesP$ is a 
proper wide subcategory of $\GPriesS$.
\end{example}

We conclude this section by pointing out that 
Theorem~\ref{thm: KAlgFrm}
yields the duality result of Hansoul and Poussart \cite{HP08}. To see this, we 
recall that a nonempty downset $I$ of a meet-semilattice $M$ is an {\em ideal} if $a,b 
\in I$ implies $\up a \cap \up b \cap I \ne \varnothing$. It is easy to see that $I$ is an 
ideal iff for each finite subset $S$ of $I$, we have $\bigcap_{s\in S} \up s \cap I \ne 
\varnothing$. As usual, an ideal $I$ is {\em proper} if $I \ne M$ and a proper ideal $I$ 
is {\em prime} if $a\wedge b \in I$ implies $a\in I$ or $b\in I$.

\begin{lemma} \label{lem: hansoul}
Let $M_1, M_2 \in \DMSLat$ and $\alpha \colon M_1 \to M_2$ be a 
$\DMSLat$-morphism. The following are equivalent.
\begin{enumerate}[$(1)$]
\item \label{hansoul 1} $\alpha$ is a $\DMSLatP$-morphism.
\item \label{hansoul 2} If $P$ is a prime filter of $M_2$, then $\alpha^{-1}(P)$ 
is a prime filter of $M_1$.
\item \label{hansoul 3} If $I$ is an ideal of $M_2$, then $\alpha^{-1}(I)$ is an 
ideal of $M_1$.
\end{enumerate}
\end{lemma}

\proof
\ref{hansoul 1}$\Rightarrow$\ref{hansoul 2}. Let $P$ be a prime filter of $M_2$. Then 
$\alpha^{-1}(P)$ is a 
filter of $M_1$. By \ref{hansoul 1} and Lemma~\ref{lem: P implies S}, $\alpha$ is a $\DMSLatS$-morphism. 
Because $P$ is proper, there is $x \in M_2 \setminus P$. Set $S = \alpha^{-1}(\down x)$. Then
$x \in \alpha[S]^u$. By (\ref{eqn: *}), there is $c \in S^u$ with $\alpha(c) \le x$. Consequently,
$\alpha(c) \notin P$, so $c \notin \alpha^{-1}(P)$, and hence $
\alpha^{-1}(P)$ is a proper filter. To see it is prime, suppose that $F,G$ are filters of 
$M_1$ with $F \cap G \subseteq \alpha^{-1}(P)$. 
Then $\up \alpha[F], \up \alpha[G]$ are filters of $M_2$.
We show that $\up \alpha[F] \cap \up 
\alpha[G] \subseteq P$. Let $x \in \up \alpha[F] \cap \up \alpha[G]$. Then $\alpha(a), 
\alpha(b) \le x$ for some $a\in F$ and $b\in G$. By (1), there is $c \in M_1$ with $a, b 
\le c$ and $\alpha(c) \le x$. Therefore, $c \in F \cap G$, so $\alpha(c) \in P$. This 
yields $x \in P$, as desired. Thus, since $P$ is prime, $\up \alpha[F] \subseteq P$ or $
\up \alpha[G] \subseteq P$, so $F \subseteq \alpha^{-1}(P)$ or 
$G \subseteq \alpha^{-1}(P)$. 
Consequently, $\alpha^{-1}(P)$ is a prime filter of $M_1$.

\ref{hansoul 2}$\Rightarrow$\ref{hansoul 3}. We first show that the pullback of a 
prime ideal is a prime ideal. If 
$I$ is a prime ideal of $M_2$, then $M_2 \setminus I$ is a prime filter (see, e.g., 
\cite[Prop.~2.3]{BJ11}). By \ref{hansoul 2}, $\alpha^{-1}(M_2 \setminus I)$ is a prime
filter of $M_1$. Since $\alpha^{-1}(M_2 \setminus I) = M_1 \setminus \alpha^{-1}(I)$, we 
conclude that $\alpha^{-1}(I)$ is a prime ideal of $M_1$. Finally, by the prime filter 
theorem for distributive meet-semilattices (see, e.g., \cite[p.~168]{Gra11} for the dual 
statement for distributive join-semilattices), each ideal is an intersection of prime ideals, 
and hence the pullback of an ideal is an ideal. 

\ref{hansoul 3}$\Rightarrow$\ref{hansoul 1}. Suppose that $S$ is a finite subset of 
$M_1$ and $x\in M_2$ with 
$\alpha(s) \le x$ for each $s\in S$. Since $\down x$ is an ideal of $M_2$, \ref{hansoul 3} 
implies that $\alpha^{-1}(\down x)$ is an ideal of $M_1$. Therefore, because 
$S \subseteq \alpha^{-1}(\down x)$, there is $c \in \alpha^{-1}(\down x)$ with 
$s \le c$ for each $s\in S$. Thus, $\alpha$ is a $\DMSLatP$-morphism.
\endproof

\begin{remark} \label{rem: optimal}
Let $\alpha \colon M_1 \to M_2$ be a $\DMSLatS$-morphism. An argument
similar to \cite[Lem.~9.7]{BJ08} shows that $\alpha$ is a $\DMSLatS$-morphism iff the $
\alpha$-preimage of an optimal filter is an optimal filter. On the other hand, 
Lemma~\ref{lem: hansoul} shows that $\alpha$ is a $\DMSLatP$-morphism iff 
the $\alpha$-preimage of a prime filter is a prime filter.
This shows how morphisms in $\DMSLatS$ and $\DMSLatP$ compare to 
each other in the language of prime and optimal filters. 
\end{remark}

\begin{remark} \label{Hansoul} 
Since distributive meet- and join-semilattices are order-duals of each other, 
it follows from
Lemma~\ref{lem: hansoul} that $\BDMSLatP$ is isomorphic to the category 
of distributive join-semilattices considered in \cite{HP08}. Thus, \cite[Thm.~1.12]{HP08} 
is a consequence of Theorem~\ref{thm: KAlgFrm} and \cite[Prop.~13.6]{BJ08}.
\end{remark}

\section{Priestley duality from the perspective of HMS duality} 
\label{sec: deriving Priestley}

In this section we show how Priestley duality fits in the general picture we 
developed in this paper.
We recall \cite[Sec.~II.3]{Joh82} that an algebraic frame $L$ is {\em coherent} if finite 
meets of compact elements are compact. Therefore, $L$ is coherent iff $\Kp(L)$ is a 
bounded sublattice of $L$. In particular, each coherent frame is compact. 

\begin{definition} \label{def: CohFrm}
Let $\CohFrm$ be the full subcategory of $\KAlgFrm$ consisting of coherent frames. 
\end{definition}

The restriction of the functor $\K$ to $\CohFrm$ lands in $\DLat$. Conversely, if a 
distributive meet-semilattice is a bounded lattice, then $\F(M)$ is a coherent frame 
because $\K\F(M)$ consists of principal upsets and $\up a \vee \up b = \up a \cap \up b = 
\up (a \vee b)$. Thus, $\K$ and $\F$ restrict to yield an equivalence of $\CohFrm$ and 
$\DLat$, and we arrive at the following well-known result, which is the pointfree version of 
Stone duality for distributive lattices:

\begin{theorem} \emph{\cite[p.~65]{Joh82}} \label{thm: coh=DL}
$\CohFrm$ is equivalent to $\DLat$.
\end{theorem}

\begin{remark}
Johnstone \cite{Joh82}, like Nachbin \cite{Nac49}, worked with the ideal functor 
rather than the filter functor. Also, Johnstone worked with the category $\CohLoc$ of 
coherent locales, the objects of which are the same as those of $\CohFrm$, but the 
morphisms of $\CohLoc$ are the right adjoints of morphisms in $\CohFrm$. 
\end{remark}
 
\begin{lemma} \label{lem: coherent frames}
\hfill
\begin{enumerate}[$(1)$]
\item \label{coherent frames 1} Let $L \in \AlgFrm$. If $L \in \CohFrm$, then 
$\PP(L)=\p(L)$.
\item \label{coherent frames 2} Let $X \in \GPries$. If $X = X_0$, then $\V(X)$ is the set of
closed upsets of $X$, and hence $\V(X) \in \CohFrm$.
\end{enumerate}
\end{lemma}

\proof
\ref{coherent frames 1}. 
Since $\p(L) \subseteq \PP(L)$, we only need to prove the other inclusion.
Let $p \in \PP(L)$ and $a, b \in L$ with $a \wedge b \le p$. If $a, b \not\le p$, then 
there are $k, l \in \Kp(L)$ with $k \le a$, $l \le b$, and $k, l \not\le p$. We have $k 
\wedge l \le a \wedge b \le p$. Since $L$ is coherent, $k \wedge l \in \Kp(L)$, so 
$k \wedge l \le p$ implies $k \wedge l \ll p$. Because $p \in \PP(L)$, either $k \le p$ or 
$l \le p$. The obtained contradiction proves that $p \in \p(L)$.

\ref{coherent frames 2}. 
Let $X_0 = X$. Then all clopen upsets and closed upsets are admissible. Therefore, $
\A(X)$ is all clopen upsets and $\V(X)$ is all closed upsets, and hence $\A(X)$ is a 
bounded sublattice of $\V(X)$.
Since $\K\V(X)=\A(X)$ by Lemma~\ref{lem: LX is a DAlgLat}, we conclude that $\V(X)$ 
is a coherent frame.
\endproof

Let $\langle X, \tau, \le \rangle$ be a Priestley space. Then $\langle X, \tau, \le , 
X\rangle$ is a generalized Priestley space. Moreover, a map between Priestley spaces is a 
Priestley morphism iff it is a $\GPriesP$-morphism between the corresponding 
generalized Priestley spaces. Thus, we may view $\Pries$ as a full subcategory of $
\GPriesP$.

\begin{theorem} \label{thm: coh = PS}
$\CohFrm$ is dually equivalent to $\Pries$.
\end{theorem}

\proof
We have that $\CohFrm$ is a full subcategory of $\KAlgFrm$ and we may view $\Pries$ 
as a full subcategory of $\GPriesP$. Therefore, by Lemma~\ref{lem: coherent 
frames} 
the dual equivalence between $\KAlgFrm$ and $\GPriesP$ of Theorem~\ref{thm: 
KAlgFrm} restricts to a dual equivalence between $\CohFrm$ and $\Pries$.
\endproof

Putting Theorems~\ref{thm: coh=DL} and~\ref{thm: coh = PS} together yields Priestley 
duality: 

\begin{theorem} \label{thm: Priestley duality}
$\CohFrm$ is equivalent to $\DLat$ and dually equivalent to $\Pries$. 
\end{theorem}

\begin{remark}
The functors establishing the duality of $\DLat$ and $\Pries$ are the compositions of 
the functors establishing the equivalence of $\DLat$ and $\CohFrm$ and 
the duality of $\CohFrm$ and $\Pries$. 
Indeed, if $M \in \DLat$, then the Priestley space $X_M$ 
of $M$ is equal to $\p(\F(M))$. If $\alpha$ is a $\DLat$-morphism, 
it follows from the proof of Theorem~\ref{thm: eq and dual eq} that
$\X(\alpha) = \Y \F(\alpha)$.
Therefore, $\X = \Y \circ \F$. In the opposite direction, if $X \in \Pries$, then 
$\ClopUp(X) = \A(X)$, which is $\Kp(\V(X))$ by Lemma~\ref{lem: LX is a DAlgLat}. 
Moreover, if $f$ is a $\Pries$-morphism, then $\V(f)=f^{-1}$ by 
Remark~\ref{rem: V and Y on S categories} and $\K \V(f)$ is the restriction of 
$f^{-1}$ to $\A(X)$ which is $\A(f)$. Thus, $\A = \K \circ \V$.
\end{remark}

We next show how to view the duality for distributive lattices with top 
but possibly without bottom from this perspective. This can be done by working with 
pointed Priestley spaces.

\subsection*{\textbf{Various morphisms between pointed Priestley spaces.}}

\begin{definition}\label{def: PPS}
Let $X = \langle X, \tau, \le, X_0, m\rangle \in \PGPries$. If $X_0 = X \setminus \{m\}$, 
then we call $X$ a {\em pointed Priestley space}. Let $\PPries$\label{PPries} be the full subcategory of 
$\PGPries$ consisting of pointed Priestley 
spaces, and define $\PPriesWS$, $\PPriesS$, and $\PPriesP$ similarly.
\end{definition}

We also introduce arithmetic frames in analogy with arithmetic lattices 
\cite[p.~117]{GHKLMS03}. 
Arithmetic frames are also known as M-frames (see, e.g., \cite[p.~2]{IM09}).

\begin{definition}
We call an algebraic frame $L$ {\em arithmetic} if 
\[
a,b \in \Kp(L) \Longrightarrow a \wedge b \in \Kp(L).
\]
\end{definition}

Note that coherent frames are simply compact arithmetic frames. 

\begin{definition} \label{def: ArFrm}
Let $\ArFrm$\label{ArFrm} be the full subcategory of $\AlgFrm$ consisting of 
arithmetic frames, and define $\ArFrmJ$, $\ArFrmWS$, and $\ArFrmS$ similarly.
\end{definition}

\begin{theorem} \label{thm: pointed Pries2}
$\ArFrmJ$ is dual to $\PPries$, $\ArFrmWS$ is dual to $\PPriesWS$, $\ArFrmS$ is dual to $\PPriesS$, 
and $\ArFrm$ is dual to $\PPriesP$. 
\end{theorem}

\proof
Let $L$ be an arithmetic frame. Observe that the proof of  Lemma~\ref{lem: coherent 
frames}\ref{coherent frames 1} only uses that $L$ is arithmetic, so it yields that
$\Y(L)$ is a pointed Priestley space.
Next, let $X$ be a pointed Priestley space. Then $\A(X)$ is all nonempty clopen upsets and $\V(X)$ is
all nonempty closed upsets of $X$. Therefore, the same argument as in 
Lemma~\ref{lem: coherent 
frames}\ref{coherent frames 2} yields that $\V(X)$ is an arithmetic frame.  It is left to apply 
Theorems~\ref{thm: duality between DAlgLat and GPries*}, \ref{thm: DAlgLatS}, 
and~\ref{thm: AlgFrm and PGPSP}.
\endproof

\begin{definition} \label{def: DmLatM} 
Let $\DmLatM$, $\DmLat$, $\DmLatS$, and $\DmLatP$ be the full 
subcategories of $\DMSLat$, $\DMSLatWS$, $\DMSLatS$, and $\DMSLatP$, respectively, 
whose objects are lattices.
\end{definition}

\begin{remark}
Objects in each of $\DmLatM$, $\DLat^-$, $\DmLatS$, and $\DmLatP$ are distributive 
lattices with top, but possibly without bottom.
Morphisms of $\DmLatM$ are meet-semilattice homomorphisms,
morphisms of $\DmLat$ are lattice homomorphisms, morphisms of $\DmLatS$ are lattice
homomorphisms which preserve bottom when it exists, and morphisms of $\DmLatP$ are lattice 
homomorphisms which pull prime filters back to prime filters. 

We point out that not every $\DmLatS$-morphism is a $\DmLatP$-morphism. For example,
let $M$ be a decreasing chain with top but no bottom, and let $\alpha \colon M \to M$ be defined by 
$\alpha(a) = 1$ for each $a \in M$. Then $\alpha$ is a $\DmLatS$-morphism but not a $\DmLatP$-morphism.
\end{remark}

For an algebraic frame $L$, we have that $L$ is arithmetic iff 
$\K(L)$ is a distributive lattice (possibly without bottom), which happens iff $\Y(L)$ is a pointed Priestley space.
Therefore, putting Theorems~\ref{thm: KDAlgLatS}, \ref{thm: DMSP}, and \ref{thm: pointed Pries2} together yields: 

\begin{theorem} \label{thm: detailed pointed Pries}
\hfill
\begin{enumerate}[$(1)$]
\item $\ArFrmJ$ is equivalent to $\DmLatM$ and dually equivalent to $
\PPries$.
\item $\ArFrmWS$ is equivalent to $\DmLat$ and dually equivalent to $\PPriesWS
$.
\item $\ArFrmS$ is equivalent to $\DmLatS$ and dually equivalent to $\PPriesS
$.
\item $\ArFrm$ is equivalent to $\DmLatP$ and dually equivalent to $
\PPriesP$.  
\end{enumerate}
\end{theorem}

\subsection*{\textbf{Various morphisms between Priestley spaces.}}

\begin{definition}\label{def: PS}
\begin{enumerate}[$(1)$]
\item Let $\PriesR$ be the category of Priestley spaces with generalized Priestley morphisms.
\item Let $\PriesPS$ be the category of Priestley spaces with partial strong Priestley morphisms.
\end{enumerate}
\end{definition}

\begin{remark} \label{rem: PS inside PPS}
\begin{enumerate}[$(1)$]
\item \label{PriesR and PriesST} By restricting the
equivalence of Proposition~\ref{prop: GPriesPS = subcategory of PGPriesWS}, 
we obtain that $\PriesR$ is equivalent to the full subcategory of $\PPries$
and $\PriesPS$ to the full subcategory of $\PPriesWS$ consisting of
Priestley spaces with isolated maxima.
\item \label{morphisms of PriesR and PriesST}
The full subcategory of $\PPriesS$ consisting of pointed Priestley spaces with isolated maxima 
is equivalent to $\Pries$. To see this, if $f \colon X \to Y$ is a
$\PPriesS$-morphism, then $f^{-1}(n) = \{m\}$. Therefore, its restriction $f^- \colon X^- \to Y^-$
is a $\Pries$-morphism. 
Such morphisms automatically satisfy $f[X_0] \subseteq Y_0$ since $X_0 = X \setminus \{m\}$ 
and $Y_0 = Y \setminus \{n\}$, so are $\PPriesP$-morphisms. Consequently, $\Pries$ is 
also equivalent to the full subcategory of $\PPriesP$ consisting of pointed Priestley spaces with isolated maxima.
\end{enumerate}
\end{remark}

\begin{definition} \label{def: CohFrmJ}\label{def: DLatM} 
\hfill
\begin{enumerate}[$(1)$]
\item \label{CohFrmJ} Let $\CohFrmJ$ and $\CohFrmWS$ be the full subcategories 
of $\ArFrmJ$ and $\ArFrmWS$, respectively, consisting of coherent frames. 
\item  \label{DLatM} Let $\DLatM$ and $\DLatL$ be the full subcategories of $\BDMSLat$ and 
$\BDMSLatWS$, respectively, whose objects are lattices.
\end{enumerate}
\end{definition}

\begin{remark} \label{rem: morphisms of DLatM}
Objects of $\DLatM$ and $\DLatL$ are 
bounded distributive lattices.
Morphisms of $\DLatM$ are meet-semilattice homomorphisms and
morphisms of $\DLatL$ are lattice homomorphisms (not necessarily preserving bottom).  
Observe that 
the full subcategories of $\BDMSLatS$ and $\BDMSLatP$ consisting of lattices
are both equal to $\DLat$. To see this, 
it is clear that morphisms of $\BDMSLatS$ preserve finite joins and 0, so are $\DLat$-morphisms. 
Since $\BDMSLatP$ is a wide subcategory of $\BDMSLatS$, we also obtain that
$\BDMSLatP$-morphisms are $\DLat$-morphisms.
Similarly, the full subcategory of $\ArFrmS$ consisting of coherent frames 
is equal to $\CohFrm$.
\end{remark}

An arithmetic frame $L$ is coherent iff 
$\K(L)$ is a bounded distributive lattice, which happens iff $\Y(L)$ is a pointed Priestley space with
isolated top.
Therefore, Theorem~\ref{thm: detailed pointed Pries} and Remark~\ref{rem: PS inside PPS} yield: 

\begin{theorem} \label{thm: detailed pointed Pries 2}
\hfill
\begin{enumerate}[$(1)$]
\item $\CohFrmJ$ is equivalent to $\DLatM$ and dually equivalent to $
\PriesR$. 
\item $\CohFrmWS$ is equivalent to $\DLatL$ and dually equivalent to $\PriesPS
$.
\end{enumerate}
\end{theorem}

\begin{remark} \label{rem: CLP}
The dual equivalence between $\DLatM$ and $\PriesR$ was first established in 
\cite{CLP91}, but the authors worked with join-preserving rather than meet-preserving 
maps between bounded distributive lattices.
\end{remark}

Putting Theorems~\ref{thm: Priestley duality}, \ref{thm: detailed pointed Pries}, and \ref{thm: detailed pointed Pries 2} 
together yields Figure~\ref{fig: second diagram}.
The tables after the figure describe the listed categories. 

\begin{figure}[H] 
\[
\begin{tikzcd}[row sep = \rowsep, column sep = 1.0pc]
& \PPries \arrow[rr, \arr, shift left = .6ex, "\V"]  \arrow[from=dd, tail] && 
\ArFrmJ \arrow[ll, \arr, shift left = .6ex, "\Y"] \arrow[rr, shift left = .6ex, "\K"] 
\arrow[from=dd, tail] && \DmLatM \arrow[ll, shift left = .6ex, "\F"] \\
\PriesR \arrow[ru, hookrightarrow]  \arrow[rr, \arr, shift left = .6ex, crossing over] 
&& \CohFrmJ \arrow[ll, \arr, shift left = .6ex, crossing over] \arrow[ru, tail] 
\arrow[rr, shift left = .6ex, crossing over]  && \DLatM \arrow[ru, tail]  
\arrow[ll, shift left = .6ex, crossing over] \\
& \PPriesWS \arrow[rr, \arr, shift left = .6ex]  \arrow[from=dd, tail] && \ArFrmWS 
\arrow[ll, \arr, shift left = .6ex] \arrow[rr, shift left = .6ex] \arrow[from=dd, tail] && 
\DmLat \arrow[ll, shift left = .6ex] \arrow[uu, tail] \\
\PriesPS \arrow[rr, \arr, shift left = .6ex, crossing over] 
\arrow[ru, hookrightarrow] \arrow[uu, tail]
&& \CohFrmWS \arrow[ll, \arr, shift left = .6ex, crossing over] \arrow[ru, tail]  
\arrow[uu, crossing over, tail] \arrow[rr, shift left = .6ex, crossing over]  && 
\DLatL \arrow[ll, shift left = .6ex, crossing over] \arrow[ru, tail] 
\arrow[uu, crossing over, tail]\\
& \PPriesS \arrow[rr, \arr, shift left = .6ex]  \arrow[from=dd, tail] && \ArFrmS 
\arrow[ll, \arr, shift left = .6ex] \arrow[rr, shift left = .6ex] \arrow[from=dd, tail] && 
\DmLatS \arrow[ll, shift left = .6ex] \arrow[uu, tail] \\
\Pries \arrow[rr, \arr, shift left = .6ex, crossing over] 
\arrow[ru, hookrightarrow] \arrow[uu, tail]
&& \CohFrm \arrow[ll, \arr, shift left = .6ex, crossing over] \arrow[ru, tail]  
\arrow[uu, crossing over, tail] \arrow[rr, shift left = .6ex, crossing over]  && 
\DLat \arrow[ll, shift left = .6ex, crossing over] \arrow[ru, tail] 
\arrow[uu, crossing over, tail]\\
& \PPriesP \arrow[rr, \arr, shift left = .6ex] && \ArFrm \arrow[ll, \arr, shift left = .6ex] 
\arrow[rr, shift left = .6ex] &&  \DmLatP \arrow[ll, shift left = .6ex] \arrow[uu, tail] \\
\Pries \arrow[rr, shift left = .6ex] \arrow[uu, equal] \arrow[ru, hookrightarrow] 
&& \arrow[ll, shift left = .6ex] \CohFrm \arrow[rr, shift left = .6ex] 
\arrow[uu, equal, crossing over] \arrow[ru, tail] && \DLat \arrow[ll, shift left = .6ex] 
\arrow[uu, equal, crossing over] \arrow[ru, tail]
\end{tikzcd}
\]
\caption{Connecting Priestley duality and HMS duality}\label{fig: second diagram}
\end{figure}

\begin{center}
\begin{tabular}{|p{\widthfive}p{\widthseven}p{\widtheight}|} \hline
\multicolumn{3}{|c|}{\textbf{Categories of pointed Priestley spaces}} \\ \hline
\textbf{Category} & \textbf{Morphisms} & \textbf{Location}\\ 
\hline
$\PPries$ & $\PGPries$-morphisms & Def.~\ref{def: PPS} \\ 
$\PPriesWS$  & $\PGPriesWS$-morphisms & \textquotedbl \\ 
$\PPriesS$  & $\PGPriesS$-morphisms & \textquotedbl \\ 
$\PPriesP$  & $\PGPriesP$-morphisms & \textquotedbl \\ 
\hline
\end{tabular}
\end{center}

\begin{center}
\begin{tabular}{|p{\widthfive}p{\widthseven}p{\widtheight}|} \hline
\multicolumn{3}{|c|}{\textbf{Categories of Priestley spaces}} \\ \hline
\textbf{Category} & \textbf{Morphisms} & \textbf{Location}\\ 
\hline
$\PriesR$ &  $\GPries$-morphisms & Def.~\ref{def: PS} \\
$\PriesPS$  &  $\GPriesS$-morphisms & \textquotedbl \\
$\Pries$  & $\Pries$-morphisms & Def.~\ref{def: DL and PS}\ref{PS}\\
\hline
\end{tabular}
\end{center}

\begin{center}
\begin{tabular}{|p{\widthfive}p{\widthseven}p{\widtheight}|} \hline
\multicolumn{3}{|c|}{\textbf{Categories of arithmetic frames}} \\ \hline
\textbf{Category} & \textbf{Morphisms} & \textbf{Location}\\ 
\hline
$\ArFrmJ$  & $\AlgFrmJ$-morphisms & Def.~\ref{def: ArFrm} \\
$\ArFrmWS$  & $\AlgFrmWS$-morphisms & \textquotedbl \\
$\ArFrmS$  & $\AlgFrmS$-morphisms & \textquotedbl \\
$\ArFrm$  & $\AlgFrm$-morphisms & \textquotedbl \\
\hline
\end{tabular}
\end{center}

\begin{center}
\begin{tabular}{|p{\widthfive}p{\widthseven}p{\widtheight}|} \hline
\multicolumn{3}{|c|}{\textbf{Categories of coherent frames}} \\ \hline
\textbf{Category}  & \textbf{Morphisms} & \textbf{Location}\\ 
\hline
$\CohFrmJ$  &  $\AlgFrmJ$-morphisms &  Def.~\ref{def: CohFrmJ}\ref{CohFrmJ} \\
$\CohFrmWS$  &  $\AlgFrmWS$-morphisms & \textquotedbl  \\
$\CohFrm$ & $\AlgFrm$-morphisms & Def.~\ref{def: CohFrm} \\
\hline
\end{tabular}
\end{center}

\begin{center}
\begin{tabular}{|p{\widthfive}p{\widthseven}p{\widtheight}|} \hline
\multicolumn{3}{|c|}{\textbf{Categories of distributive lattices}} \\ \hline
\textbf{Category} & \textbf{Morphisms} & \textbf{Location}\\ 
\hline
$\DmLatM$ & $\DMSLat$-morphisms & Def.~\ref{def: DmLatM} \\
$\DmLat$  & $\DMSLatWS$-morphisms & \textquotedbl \\
$\DmLatS$  & $\DMSLatS$-morphisms & \textquotedbl \\
$\DmLatP$  & $\DMSLatP$-morphisms & \textquotedbl \\
\hline
\end{tabular}
\end{center}

\begin{center}
\begin{tabular}{|p{\widthfive}p{\widthseven}p{\widtheight}|} \hline
\multicolumn{3}{|c|}{\textbf{Categories of bounded distributive lattices}} \\ \hline
\textbf{Category} & \textbf{Morphisms} & \textbf{Location}\\ 
\hline
$\DLatM$  & $\DMSLat$-morphisms & 
Def.~\ref{def: DLatM}\ref{DLatM} \\
$\DLatL$  & $\DMSLatWS$-morphisms & \textquotedbl \\
$\DLat$ & $\DLat$-morphisms & Def.~\ref{def: DL and PS}\ref{DL} \\
\hline
\end{tabular}
\end{center}

\section{Stone duality for generalized boolean algebras via HMS duality} 
\label{sec: deriving Stone}

We conclude the paper by discussing how to view Stone duality for boolean algebras and 
generalized boolean algebras from this perspective. 

\begin{definition} \label{def: BA and Stone}
\hfill
\begin{enumerate}[$(1)$]
\item Let $\BA$\label{BA} be the category of boolean algebras and boolean homomorphisms. 
\item Let $\Stone$\label{Stone} be the category of Stone 
spaces and continuous maps. 
\end{enumerate}
\end{definition}

We may view $\BA$ as a full subcategory of $\DLat$ and $\Stone$ as a full subcategory of $\Pries$. 
For a frame $L$, we recall that the {\em pseudocomplement} of $a\in L$ is 
\[
a^* = \bigvee \{ x \in L \mid a \wedge x = 0 \}
\] 
and that $a$ is {\em complemented} if $a\vee a^*=1$. Then $L$ is 
{\em zero-dimensional} if the complemented elements are join-dense in $L$. 

\begin{definition} \emph{\cite{Ban89}} \label{def: Stone frame}
A {\em Stone frame} is a compact zero-dimensional frame. 
Let $\StoneFrm$ be the 
category of Stone frames and frame homomorphisms. 
\end{definition}

Since in Stone frames compact elements are exactly 
complemented elements, every Stone frame is coherent, and every frame homomorphism 
between Stone frames preserves compact elements. Thus, $\StoneFrm$ is a full 
subcategory of $\CohFrm$, and we obtain Stone duality for boolean algebras as 
a consequence of Theorem~\ref{thm: Priestley duality}:

\begin{theorem} \emph{\cite{Sto36,Ban89}} \label{thm: Stone duality}
$\BA$ is equivalent to $\StoneFrm$ and dually equivalent to $\Stone$.
\end{theorem}

\proof
It is well known 
that $B\in\BA$ implies $\F(B)\in\StoneFrm$ and $L\in\StoneFrm$ implies $\K(L)
\in\BA$. Thus, the equivalence between $\DLat$ and $\CohFrm$ (see 
Theorem~\ref{thm: coh=DL}) restricts to an equivalence between $\BA$ and $
\StoneFrm$.

We next show that the dual equivalence between $\CohFrm$ and $\Pries$ (see 
Theorem~\ref{thm: coh = PS}) restricts to a dual equivalence between $\StoneFrm$ and 
$\Stone$. For this it is enough to observe that from $L\in\StoneFrm$ it follows that the 
order on $\p(L)$ is equality, and that $X$ a Stone space implies $\V(X)$ is a Stone frame. 
For the latter, since $\K\V(X)=\A(X)=\Clop(X)$, we see that $\V(X)$ is 
zero-dimensional, hence a Stone frame. For the former, let $L$ be a Stone frame 
and $p, q \in \p(L)$ with $p < q$. Then $q \not\le p$, so there is $k \in \Kp(L)$ with $k \le 
q$ and $k \not\le p$. The latter together with $k \wedge k^*\le p$ implies that $k^* \le 
p \le q$. Therefore, $k \vee k^* \le q$. Since $L$ is a Stone frame, $k$ is 
complemented, so $k \vee k^* = 1$. Thus, $q = 1$, a contradiction. Consequently, the 
order on $\p(L)$ is equality.
\endproof

\subsection*{\textbf{Various morphisms between generalized boolean algebras.}}

We recall that a generalized boolean algebra is a distributive lattice $M$ with bottom 
such that $[0,a]$ is a boolean algebra for each $a\in M$. Since we work with distributive 
lattices with top, we consider the order-dual of $M$. Therefore, by a \emph{generalized 
boolean algebra} $M$ we mean a distributive lattice $M$ with top such that $[a,1]$ is a 
boolean algebra for each $a\in M$. 

\begin{definition} \label{def: GBA}
Let $\GBAM$, $\GBA$, $\GBAS$, and $\GBAP$ be the full subcategories of $
\DmLatM$, $\DmLat$, $\DmLatS$, and $\DmLatP$, respectively, whose objects are 
generalized boolean algebras. 
\end{definition}

\begin{remark}
Objects in each of $\GBAM$, $\GBA$, $\GBAS$, and $\GBAP$ are generalized boolean algebras.
Morphisms of $\GBAM$ are meet-semilattice homomorphisms and
morphisms of $\GBA$ are lattice homomorphisms (not necessarily preserving bottom). 
Morphisms of $\GBAS$ are lattice homomorphisms that preserve bottom when it exists and
morphisms of $\GBAP$ are lattice homomorphisms which pull prime filters back to prime filters. 
\end{remark}

We next generalize Stone frames as follows (see, e.g., \cite{BK23}).

\begin{definition} 
A {\em locally Stone frame} is an algebraic zero-dimensional frame. 
\end{definition}

\begin{remark}
It is easy to see that a frame $L$ is a locally Stone frame iff compact 
complemented elements are join-dense in $L$.
\end{remark}

Clearly Stone frames are compact locally Stone frames. It is also straightforward to 
see that each locally Stone frame is an arithmetic frame. 

\begin{definition} \label{def: LStoneFrm}
Let $\LStoneFrmJ$, $\LStoneFrmWS$, $\LStoneFrmS$, and $\LStoneFrm$ be the full subcategories of $
\ArFrmJ$, $\ArFrmWS$, $\ArFrmS$, and $\ArFrm$, respectively,  whose objects are locally 
Stone frames.
\end{definition}

We next define pointed Stone spaces as special pointed Priestley spaces.

\begin{definition} \label{def: pointed Stone}
A {\em pointed Stone space} is a pointed Priestley space $(X, m)$ such that $\le$ 
restricts to the identity on $X^-$. 
\end{definition}

\begin{remark}
A standard definition of a pointed Stone space is that it is a Stone space $X$ with a designated point $m\in X$ 
(see Remark~\ref{rem: pStone}). The definition above is different in that we make $m$ the maximum of $X$. 
This definition fits nicer in the more
general picture of pointed Priestley spaces developed in this paper. It also makes sense from the perspective 
of Stone duality for generalized
boolean algebras. Indeed, if we order the dual space $X_M$ of a generalized boolean algebra $M$ by inclusion, 
then $M$ is the maximum of $X_M$.
\end{remark}

\begin{definition} \label{def: PStone}
Let $\PStone$\label{PStone}  be the full subcategory of $\PPries$ whose objects are 
pointed Stone spaces, and define $\PStoneWS$, $\PStoneS$, and $\PStoneP$ similarly.
\end{definition}

\begin{remark} \label{rem: PStone objects and morphisms}
Objects of $\PStone$, $\PStoneWS$, $\PStoneS$, and $\PStoneP$ are pointed Stone spaces.
Morphisms of $\PStone$ are $\PPries$-morphisms and hence are relations.
Morphisms of $\PStoneWS$ are strong Priestley morphisms. 
If $(X, m)$ is a pointed Stone space, then $U \in \A(X)$ iff $U$ is a clopen subset of $X$ containing
$m$. Consequently, a $\PStoneWS$-morphism $f \colon (X, m) \to (Y, n)$ is a continuous 
function with $f(m) = n$. A $\PStoneWS$-morphism $f$ is a $\PStoneS$-morphism provided
$f^{-1}(n) = \{m\}$ when $m, n$ are isolated. In this case $f$ restricts to a
continuous function from $X^-$ to $Y^-$. Thus, 
$\Stone$ is equivalent to the full subcategory of $\PStoneS$ consisting of 
pointed Stone spaces with isolated top, as well as 
to the corresponding full subcategory of $\PStoneP$.
\end{remark}

An arithmetic frame $L$ is locally Stone iff $\K(L)$ is a generalized Boolean algebra, which happens 
iff $\Y(L)$ is a pointed Stone space with an isolated maximum. Therefore, 
Theorem~\ref{thm: detailed pointed Pries} and Remark~\ref{rem: PStone objects and morphisms} yield:

\begin{theorem} \label{thm: pointed Stone}
\hfill
\begin{enumerate}[$(1)$]
\item $\LStoneFrmJ$ is equivalent to $\GBAM$ and dually equivalent to $\PStone$.
\item $\LStoneFrmWS$ is equivalent to $\GBA$ and dually equivalent to $\PStoneWS$.
\item $\LStoneFrmS$ is equivalent to $\GBAS$ and dually equivalent to $\PStoneS$.
\item $\LStoneFrm$ is equivalent to $\GBAP$ and dually equivalent to $\PStoneP$.  
\end{enumerate}
\end{theorem}

\subsection*{\textbf{Various morphisms between boolean algebras.}}

\begin{definition} \label{def: BA}\label{def: StoneFrmJ}\label{def: StoneR}
\hfill
\begin{enumerate}[$(1)$]
\item \label{BAM}Let $\BAM$ and $\BAL$ be the full subcategories of $
\DLatM$ and $\DLatL$, respectively, whose objects are 
boolean algebras.
\item \label{StoneFrmJ} Let $\StoneFrmJ$ and $\StoneFrmWS$ be the full subcategories of $\ArFrmJ$ and 
$\ArFrmWS$, respectively, whose objects are Stone frames.
\item \label{StoneR} Let $\StoneR$ and $\StonePS$
be the full subcategories of $\PriesR$ and $\PriesPS$, respectively, whose objects 
are Stone spaces.

\end{enumerate}
\end{definition}

\begin{remark}
\begin{enumerate}[$(1)$]
\item Objects in both $\BAM$ and $\BAL$ are  boolean algebras.
Morphisms of $\BAM$ are meet-semilattice homomorphisms and
those of $\BAL$ are lattice homomorphisms (which may not preserve bottom).
\item By restricting the
equivalence of Remark~\ref{rem: PS inside PPS}\ref{PriesR and PriesST}, we obtain that $\StoneR$ is equivalent to the 
full subcategory of $\PStone$
and $\StonePS$ to the full subcategory of $\PStoneWS$ consisting of pointed
Stone spaces with isolated maxima. 
\item By
Remark~\ref{rem: PS inside PPS}\ref{morphisms of PriesR and PriesST},
the full subcategories of $\PStoneS$ and $\PStoneP$ consisting of pointed Stone spaces with isolated maxima 
are equivalent to $\Stone$. 
\end{enumerate}
\end{remark}

Let $L \in \LStoneFrmJ$. Then $L$ is a Stone frame iff $\K(L)$ is a boolean algebra, which
happens iff $\Y(L)$ is a Stone space. Therefore, Theorem~\ref{thm: pointed Stone} yields:

\begin{theorem} \label{thm: Stone}
\hfill
\begin{enumerate}[$(1)$]
\item $\StoneFrmJ$ is equivalent to $\BAM$ and dually equivalent to $
\StoneR$.
\item $\StoneFrmWS$ is equivalent to $\BAL$ and dually equivalent to $\StonePS$.
\end{enumerate}
\end{theorem}

\begin{remark} \label{rem: Halmos}
The dual equivalence between $\BAM$ and $\StoneR$ was first established in 
\cite{Hal56}, but Halmos worked with join-preserving rather than meet-preserving 
maps between boolean algebras.
\end{remark}

Putting together Theorems~\ref{thm: pointed Stone} and \ref{thm: Stone} yields
Figure \ref{fig: third diagram}, which is similar to 
Figure~\ref{fig: second diagram}. The tables after the figure describe the listed categories.
 
\begin{figure}[H]
\[
\begin{tikzcd}[row sep = \rowsep, column sep = 1.0pc]
& \PStone \arrow[rr, \arr, shift left = .6ex, "\V"]  \arrow[from=dd, tail] && 
\LStoneFrmJ \arrow[ll, \arr, shift left = .6ex, "\Y"] \arrow[rr, shift left = .6ex, "\K"] 
\arrow[from=dd, tail] && \GBAM \arrow[ll, shift left = .6ex, "\F"] \\
\StoneR \arrow[ru, hookrightarrow]  \arrow[rr, \arr, shift left = .6ex, crossing over] 
&& \StoneFrmJ \arrow[ll, \arr, shift left = .6ex, crossing over] \arrow[ru, tail] 
\arrow[rr, shift left = .6ex, crossing over]  && \BAM \arrow[ru, tail]  
\arrow[ll, shift left = .6ex, crossing over] \\
& \PStoneWS \arrow[rr, \arr, shift left = .6ex]  \arrow[from=dd, tail] && \LStoneFrmWS 
\arrow[ll, \arr, shift left = .6ex] \arrow[rr, shift left = .6ex] \arrow[from=dd, tail] && 
\GBA \arrow[ll, shift left = .6ex] \arrow[uu, tail] \\
\StonePS \arrow[rr, \arr, shift left = .6ex, crossing over] \arrow[ru, hookrightarrow] 
\arrow[uu, tail]
&& \StoneFrmWS \arrow[ll, \arr, shift left = .6ex, crossing over] \arrow[ru, tail]  
\arrow[uu, crossing over, tail] \arrow[rr, shift left = .6ex, crossing over]  && 
\BAL \arrow[ll, shift left = .6ex, crossing over] \arrow[ru, tail] 
\arrow[uu, crossing over, tail]\\
& \PStoneS \arrow[rr, \arr, shift left = .6ex]  \arrow[from=dd, tail] && \LStoneFrmS 
\arrow[ll, \arr, shift left = .6ex] \arrow[rr, shift left = .6ex] \arrow[from=dd, tail] && 
\GBAS \arrow[ll, shift left = .6ex] \arrow[uu, tail] \\
\Stone \arrow[rr, \arr, shift left = .6ex, crossing over] \arrow[ru, hookrightarrow] 
\arrow[uu, tail]
&& \StoneFrm \arrow[ll, \arr, shift left = .6ex, crossing over] \arrow[ru, tail]  
\arrow[uu, crossing over, tail] \arrow[rr, shift left = .6ex, crossing over]  && 
\BA \arrow[ll, shift left = .6ex, crossing over] \arrow[ru, tail] 
\arrow[uu, crossing over, tail]\\
& \PStoneP \arrow[rr, \arr, shift left = .6ex] && \LStoneFrm 
\arrow[ll, \arr, shift left = .6ex] 
\arrow[rr, shift left = .6ex] &&  \GBAP \arrow[ll, shift left = .6ex] \arrow[uu, tail] \\
\Stone \arrow[rr, shift left = .6ex] \arrow[uu, equal] \arrow[ru, hookrightarrow] 
&& \arrow[ll, shift left = .6ex] \StoneFrm \arrow[rr, shift left = .6ex] 
\arrow[uu, equal, crossing over] \arrow[ru, tail] && \BA \arrow[ll, shift left = .6ex] 
\arrow[uu, equal, crossing over] \arrow[ru, tail]
\end{tikzcd}
\]
\caption{Connecting Stone duality and HMS duality}\label{fig: third diagram}
\end{figure}

\begin{center}
\begin{tabular}{|p{\widthfive}p{\widthseven}p{\widtheight}|} \hline
\multicolumn{3}{|c|}{\textbf{Categories of pointed Stone spaces}} \\ \hline
\textsf{Category}  & \textsf{Morphisms} & \textsf{Location}\\ 
\hline
$\PStone$  & $\PGPries$-morphisms & Def.~\ref{def: PStone} \\ 
$\PStoneWS$  & $\PGPriesWS$-morphisms & \textquotedbl \\ 
$\PStoneS$  & $\PGPriesS$-morphisms & \textquotedbl \\ 
$\PStoneP$ & $\PGPriesP$-morphisms & \textquotedbl \\ 
\hline
\end{tabular}
\end{center}

\begin{center}
\begin{tabular}{|p{\widthfive}p{\widthseven}p{\widtheight}|} \hline
\multicolumn{3}{|c|}{\textbf{Categories of Stone spaces}} \\ \hline
{\sf Category}  & {\sf Morphisms} & {\sf Location} \\ 
\hline
$\StoneR$  & $\PGPries$-morphisms & Def.~\ref{def: StoneR}\ref{StoneR} \\ 
$\StoneS$  & $\PGPriesS$-morphisms & \textquotedbl \\ 
$\Stone$  & $\PGPriesP$-morphisms & Def.~\ref{def: BA and Stone}\ref{Stone}\\ 
\hline
\end{tabular}
\end{center}

\begin{center}
\begin{tabular}{|p{\widthfive}p{\widthseven}p{\widtheight}|} \hline
\multicolumn{3}{|c|}{\textbf{Categories of locally Stone frames}} \\ \hline
\textsf{Category} & \textsf{Morphisms} & \textsf{Location}\\ 
\hline
$\LStoneFrmJ$  & $\AlgFrmJ$-morphisms & Def.~\ref{def: LStoneFrm} \\
$\LStoneFrmWS$  & $\AlgFrmWS$-morphisms & \textquotedbl \\
$\LStoneFrmS$  & $\AlgFrmS$-morphisms & \textquotedbl \\
$\LStoneFrm$  & $\AlgFrm$-morphisms & \textquotedbl \\
\hline
\end{tabular}
\end{center}

\begin{center}
\begin{tabular}{|p{\widthfive}p{\widthseven}p{\widtheight}|} \hline
\multicolumn{3}{|c|}{\textbf{Categories of Stone frames}} \\ \hline
\textsf{Category}  & \textsf{Morphisms} & \textsf{Location}\\ 
\hline
$\StoneFrmJ$  & $\AlgFrmJ$-morphisms & Def.~\ref{def: StoneFrmJ}\ref{StoneFrmJ}\\
$\StoneFrmWS$  & $\AlgFrmWS$-morphisms & \textquotedbl\\ 
$\StoneFrm$  & frame homomorphisms & Def.~\ref{def: Stone frame}\\ 
\hline
\end{tabular}
\end{center}

\begin{center}
\begin{tabular}{|p{\widthfive}p{\widthseven}p{\widtheight}|} \hline
\multicolumn{3}{|c|}{\textbf{Categories of boolean algebras}} \\ \hline
\textsf{Category}  & \textsf{Morphisms} & \textsf{Location}\\ 
\hline
$\GBAM$ & $\DMSLat$-morphisms & Def.~\ref{def: GBA} \\
$\GBA$  & $\DMSLatWS$-morphisms & \textquotedbl \\
$\GBAS$  & $\DMSLatS$-morphisms & \textquotedbl \\
$\GBAP$  & $\DMSLatP$-morphisms & \textquotedbl \\
\hline
\end{tabular}
\end{center}

\begin{center}
\begin{tabular}{|p{\widthfive}p{\widthseven}p{\widtheight}|} \hline
\multicolumn{3}{|c|}{\textbf{Categories of generalized boolean algebras}} \\ \hline
\textsf{Category}  & \textsf{Morphisms} & \textsf{Location}\\ 
\hline
$\BAM$ & $\DMSLat$-morphisms & Def.~\ref{def: BA}\ref{BAM} \\
$\BAL$  & $\DMSLatS$-morphisms & \textquotedbl \\ 
$\BA$  & boolean homomorphisms & Def.~\ref{def: BA and Stone}\ref{BA}\\ 
\hline
\end{tabular}
\end{center}

\section*{Acknowledgement}

We would like to thank the referee for useful comments which have improved the 
readability of the paper.

\newcommand{\etalchar}[1]{$^{#1}$}
\newcommand{\cprime}{$'$}


\begin{thebibliography}{GHK{\etalchar{+}}03}

\bibitem[Ban89]{Ban89}
B.~Banaschewski.
\newblock Universal zero-dimensional compactifications.
\newblock In {\em Categorical topology and its relation to analysis, algebra
  and combinatorics ({P}rague, 1988)}, pages 257--269. World Sci. Publ.,
  Teaneck, NJ, 1989.
  
\bibitem[BF48]{BF48}
G.~Birkhoff and O.~Frink, Jr.
\newblock Representations of lattices by sets.
\newblock {\em Trans. Amer. Math. Soc.}, 64:299--316, 1948.

\bibitem[BH21]{BH21}
G.~Bezhanishvili and J.~Harding.
\newblock The {Fell} compactification of a poset.
\newblock In {\em Statistical and fuzzy approaches to data processing, with
  applications to econometrics and other areas. In honor of Hung T. Nguyen's
  75th birthday}, pages 31--46. Cham: Springer, 2021.

\bibitem[BJ08]{BJ08}
G.~Bezhanishvili and R.~Jansana.
\newblock Duality for distributive and implicative semi-lattices.
\newblock Barcelona Logic Group preprint, 2008. Available at arXiv:2410.23664.

\bibitem[BJ11]{BJ11}
G.~Bezhanishvili and R.~Jansana.
\newblock Priestley style duality for distributive meet-semilattices.
\newblock {\em Studia Logica}, 98(1-2):83--122, 2011.

\bibitem[BJ13]{BJ13}
G.~Bezhanishvili and R.~Jansana.
\newblock Esakia style duality for implicative semilattices.
\newblock {\em Appl. Categ. Structures}, 21(2):181--208, 2013.

\bibitem[BK23]{BK23}
G.~Bezhanishvili and A.~Kornell.
\newblock On the structure of modal and tense operators on a boolean algebra.
\newblock https://arxiv.org/pdf/2308.08664, 2023.

\bibitem[BMR17]{BMR17}
G.~Bezhanishvili, T.~Moraschini, and J.~G. Raftery.
\newblock Epimorphisms in varieties of residuated structures.
\newblock {\em J. Algebra}, 492:185--211, 2017.

\bibitem[CD98]{CD98a}
D.~M. Clark and B.~A. Davey.
\newblock {\em Natural dualities for the working algebraist}, volume~57 of {\em
  Cambridge Studies in Advanced Mathematics}.
\newblock Cambridge University Press, Cambridge, 1998.

\bibitem[CH78]{CH78}
W.~H. Cornish and R.~C. Hickman.
\newblock Weakly distributive semilattices.
\newblock {\em Acta Math. Acad. Sci. Hungar.}, 32(1-2):5--16, 1978.

\bibitem[CJ99]{CJ99}
S.~Celani and R.~Jansana.
\newblock Priestley duality, a {S}ahlqvist theorem and a {G}oldblatt-{T}homason
  theorem for positive modal logic.
\newblock {\em Log. J. IGPL}, 7(6):683--715, 1999.

\bibitem[CLP91]{CLP91}
R.~Cignoli, S.~Lafalce, and A.~Petrovich.
\newblock Remarks on {P}riestley duality for distributive lattices.
\newblock {\em Order}, 8(3):299--315, 1991.

\bibitem[DP02]{DP02}
B.~A. Davey and H.~A. Priestley.
\newblock {\em Introduction to lattices and order}.
\newblock Cambridge University Press, New York, second edition, 2002.

\bibitem[Geh16]{Geh16}
M.~Gehrke.
\newblock Duality in computer science.
\newblock In {\em Proceedings of the 31st {A}nnual {ACM}-{IEEE} {S}ymposium on
  {L}ogic in {C}omputer {S}cience ({LICS} 2016)}, page~15. ACM, New York, 2016.

\bibitem[GH01]{GH01}
M.~Gehrke and J.~Harding.
\newblock Bounded lattice expansions.
\newblock {\em J. Algebra}, 238(1):345--371, 2001.

\bibitem[GHK{\etalchar{+}}03]{GHKLMS03}
G.~Gierz, K.~H. Hofmann, K.~Keimel, J.~D. Lawson, M.~Mislove, and D.~S. Scott.
\newblock {\em Continuous lattices and domains}.
\newblock Cambridge University Press, Cambridge, 2003.

\bibitem[GJ94]{GJ94}
M.~Gehrke and B.~J\'onsson.
\newblock Bounded distributive lattices with operators.
\newblock {\em Math. Japon.}, 40(2):207--215, 1994.

\bibitem[GJ04]{GJ04}
M.~Gehrke and B.~J\'onsson.
\newblock Bounded distributive lattice expansions.
\newblock {\em Math. Scand.}, 94(1):13--45, 2004.

\bibitem[Gol89]{Gol89}
R.~Goldblatt.
\newblock Varieties of complex algebras.
\newblock {\em Ann. Pure Appl. Logic}, 44(3):173--242, 1989.

\bibitem[Gr\"{a}11]{Gra11}
G.~Gr\"{a}tzer.
\newblock {\em Lattice theory: foundation}.
\newblock Birkh\"{a}user/Springer Basel AG, Basel, 2011.

\bibitem[GvG14]{GvG14}
M.~Gehrke and S.~J. van Gool.
\newblock Distributive envelopes and topological duality for lattices via
  canonical extensions.
\newblock {\em Order}, 31(3):435--461, 2014.

\bibitem[GW99]{GW99}
B.~Ganter and R.~Wille.
\newblock {\em Formal concept analysis}.
\newblock Springer-Verlag, Berlin, 1999.
\newblock Mathematical foundations, Translated from the 1996 German original by
  Cornelia Franzke.
  
\bibitem[Hal56]{Hal56}
P.~R. Halmos.
\newblock Algebraic logic. {I}. {M}onadic {B}oolean algebras.
\newblock {\em Compositio Math.}, 12:217--249, 1956.

\bibitem[Har92]{Har92}
G.~Hartung.
\newblock A topological representation of lattices.
\newblock {\em Algebra Universalis}, 29(2):273--299, 1992.

\bibitem[Har97]{Har97}
C.~Hartonas.
\newblock Duality for lattice-ordered algebras and for normal algebraizable
  logics.
\newblock {\em Studia Logica}, 58(3):403--450, 1997.

\bibitem[HD97]{HD97}
C.~Hartonas and J.~M. Dunn.
\newblock Stone duality for lattices.
\newblock {\em Algebra Universalis}, 37(3):391--401, 1997.

\bibitem[HMS74]{HMS74}
K.~H. Hofmann, M.~Mislove, and A.~Stralka.
\newblock {\em The {P}ontryagin duality of compact {${\rm O}$}-dimensional
  semilattices and its applications}.
\newblock Lecture Notes in Mathematics, Vol. 396. Springer-Verlag, Berlin-New
  York, 1974.

\bibitem[HP08]{HP08}
G.~Hansoul and C.~Poussart.
\newblock Priestley duality for distributive semilattices.
\newblock {\em Bull. Soc. Roy. Sci. Li\`ege}, 77:104--119, 2008.

\bibitem[HR94]{HR94}
E.~Hewitt and K.~A. Ross.
\newblock {\em Abstract harmonic analysis. {Volume} {I}: {Structure} of
  topological groups, integration theory, group representations.}, volume 115
  of {\em Grundlehren Math. Wiss.}
\newblock Berlin: Springer-Verlag, 2nd ed. edition, 1994.

\bibitem[IM09]{IM09}
W.~Iberkleid and W.~W. McGovern.
\newblock A natural equivalence for the category of coherent frames.
\newblock {\em Algebra Universalis}, 62(2-3):247--258, 2009.

\bibitem[Joh82]{Joh82}
P.~T. Johnstone.
\newblock {\em Stone spaces}, volume~3 of {\em Cambridge Studies in Advanced
  Mathematics}.
\newblock Cambridge University Press, Cambridge, 1982.

\bibitem[MJ14a]{MJ14a}
M.~A. Moshier and P.~Jipsen.
\newblock Topological duality and lattice expansions, {I}: {A} topological
  construction of canonical extensions.
\newblock {\em Algebra Universalis}, 71(2):109--126, 2014.

\bibitem[MJ14b]{MJ14b}
M.~A. Moshier and P.~Jipsen.
\newblock Topological duality and lattice expansions, {II}: {L}attice
  expansions with quasioperators.
\newblock {\em Algebra Universalis}, 71(3):221--234, 2014.

\bibitem[Nac49]{Nac49}
L.~Nachbin.
\newblock On a characterization of the lattice of all ideals of a {B}oolean
  ring.
\newblock {\em Fund. Math.}, 36:137--142, 1949.

\bibitem[Nac65]{Nac65}
L.~Nachbin.
\newblock {\em Topology and order}.
\newblock Van Nostrand Mathematical Studies, No. 4. D. Van Nostrand Co., Inc.,
  Princeton, N.J.-Toronto, Ont.-London, 1965.
\newblock Translated from the Portuguese by Lulu Bechtolsheim.

\bibitem[Pan13]{Pan13}
P.~Panangaden.
\newblock Duality in logic and computation.
\newblock In {\em 28th {A}nnual {ACM}/{IEEE} {S}ymposium on {L}ogic in
  {C}omputer {S}cience ({LICS} 2013)}, pages 4--11. IEEE Computer Soc., Los
  Alamitos, CA, 2013.

\bibitem[PP12]{PP12}
J.~Picado and A.~Pultr.
\newblock {\em Frames and locales: Topology without points}.
\newblock Frontiers in Mathematics. Birkh\"auser/Springer Basel AG, Basel,
  2012.

\bibitem[Plo08]{Plo08}
M.~Plo\v{s}\v{c}ica.
\newblock A natural representation of bounded lattices.
\newblock Has not been published, 2008.

\bibitem[Pri70]{Pri70}
H.~A. Priestley.
\newblock Representation of distributive lattices by means of ordered {S}tone
  spaces.
\newblock {\em Bull. London Math. Soc.}, 2:186--190, 1970.

\bibitem[Pri72]{Pri72}
H.~A. Priestley.
\newblock Ordered topological spaces and the representation of distributive
  lattices.
\newblock {\em Proc. London Math. Soc. (3)}, 24:507--530, 1972.

\bibitem[Sto36]{Sto36}
M.~H. Stone.
\newblock The theory of representations for {B}oolean algebras.
\newblock {\em Trans. Amer. Math. Soc.}, 40(1):37--111, 1936.

\bibitem[Sto37]{Sto37c}
M.~H. Stone.
\newblock Topological representations of distributive lattices and {Brouwerian}
  logics.
\newblock {\em {\v{C}}as. Mat. Fys.}, 67:1--25, 1937.

\bibitem[Urq78]{Urq78}
A.~Urquhart.
\newblock A topological representation theory for lattices.
\newblock {\em Algebra Universalis}, 8(1):45--58, 1978.

\end{thebibliography}
\end{document}